\documentclass[11pt,twoside]{article}

\usepackage{amsmath, amsthm,amsfonts,amssymb,mathtools,mathdots,scalerel,mathrsfs,stmaryrd,cmll}
\usepackage[utf8]{inputenc}
\usepackage[T1]{fontenc}
\usepackage{lmodern, tikz-cd}

\usepackage{enumitem, tocloft, csquotes}
\usepackage[hidelinks]{hyperref}
\usepackage{tikz, fancyhdr, xparse, xcolor}
\usepackage[british]{babel}
\usepackage[a4paper, top=4.5cm, bottom=4.5cm, left=4cm, right=4cm, asymmetric]{geometry}

\usepackage{ahtitle,titlesec}
\usepackage{etoolbox}
\makeatletter
\patchcmd{\ttlh@hang}{\parindent\z@}{\parindent\z@\leavevmode}{}{}
\patchcmd{\ttlh@hang}{\noindent}{}{}{}
\makeatother

\usetikzlibrary{arrows, decorations.markings, shapes.geometric, decorations.pathmorphing, decorations.pathreplacing, intersections, patterns,calc, backgrounds}

\pgfdeclarelayer{bg}
\pgfdeclarelayer{mid}
\pgfdeclarelayer{fg}
\pgfsetlayers{bg,mid,fg,main}

\tikzset{
	0c/.style={circle, draw, fill, inner sep=.7pt},
	1c/.style={->, shorten <=2pt, shorten >=2pt},
	1clong/.style={->},
	edge/.style={shorten <=2pt, shorten >=2pt},
	equal/.style={shorten <=2pt, shorten >=2pt, double},
	1cinc/.style={right hook->, shorten <=2pt, shorten >=2pt},
	1csurj/.style={->>, shorten <=2pt, shorten >=2pt},
	1cincl/.style={left hook->, shorten <=2pt, shorten >=2pt},
	arloop/.tip={Glyph[glyph math command=looparrowleft, swap]},
	1cloop/.style={arloop->, shorten <=2pt, shorten >=2pt},
	2c/.style={double distance=1.5pt, shorten <=6pt, shorten >=8pt, decoration={markings,mark=at position -6pt with {\arrow[scale=1.5]{>}}}, preaction={decorate}},
	follow/.style={->, >=stealth, ultra thick, shorten <=3pt, shorten >=3pt, color=gray!70},
	bar/.style={ultra thick, shorten <=3pt, shorten >=3pt, color=magenta},
	arlabel/.style={scale=.8},
	wire/.style={line width=.5pt, color=black},
	wirecov/.style={line width=5pt, color=gray!10},
	cover/.style={circle, draw=gray!10, fill=gray!10, inner sep=2.5pt},
	aura2/.style={rounded corners, line width=.5pt, fill, color=gray, fill opacity=.1, draw opacity=.3},
	aura/.style={rounded corners, line width=.5pt, fill, color=magenta, fill opacity=.1, draw opacity=.3},
	wiredot/.style={densely dotted, line width=.5pt, color=gray},
	wirethin/.style={line width=.5pt, draw opacity=.2},
	outline/.style={draw=gray, line width=.5pt, draw opacity=.3},
	dotthin/.style={circle, draw=gray!60, fill=gray!60, inner sep=1pt},
	dot/.style={circle, draw, fill=black, arlabel, inner sep=1pt},
	dotwhite/.style={circle, draw=black, line width=.5pt, fill=white, inner sep=1.1pt},
}

\newcommand\eqdef{\coloneqq}
\newcommand\nbd{\nobreakdash-\hspace{0pt}}
\newcommand\idd[1]{\mathrm{id}_{#1}}

\newcommand\invrs[1]{#1^{-1}}

\newcommand\incl{\hookrightarrow}
\newcommand\incliso{\stackrel{\sim}{\hookrightarrow}}

\newcommand\surj{\twoheadrightarrow}
\newcommand\pfun{\rightharpoonup}

\newcommand\undl[1]{\underline{#1}}
\newcommand\slice[2]{{#1}/{\raisebox{-2pt}{$#2$}}}

\newcommand\gray{\,{\otimes}\,}
\newcommand\gsmash{\,{\owedge}\,}

\newcommand\tensorp{\,{\otimes}_\mathbb{S}\,}

\newcommand\coo[1]{{#1}^\mathrm{co}}
\newcommand\oppall[1]{{#1}^\circ}

\newcommand\hasse[1]{\mathscr{H}#1}
\newcommand\hasseo[1]{\mathscr{H}^o#1}
\newcommand\dmn[1]{\mathrm{dim}(#1)}
\newcommand\rank[1]{\mathrm{rk}(#1)}
\newcommand\frdmn[1]{\mathrm{frdim}(#1)}
\newcommand\clos[1]{\mathrm{cl}#1}
\NewDocumentCommand \bord{g g} {\IfNoValueTF{#2}{%
	\IfNoValueTF{#1}{\partial}{\partial_{#1}}}{\partial_{#1}^{#2}}}
\NewDocumentCommand \sbord{g g} {\IfNoValueTF{#2}{%
	\IfNoValueTF{#1}{\Delta}{\Delta_{#1}}}{\Delta_{#1}^{#2}}}
\newcommand\cp[1]{\,{\scriptstyle\#}_{#1}\,}

\newcommand\mol[2]{\mathcal{M}o\ell{#2}_{#1}}
\newcommand\submol{\sqsubseteq}
\newcommand\maxd[2]{\mathscr{M}_{#1}#2}

\newcommand\infl[1]{O{(#1)}}
\newcommand\eps[1]{\varepsilon_{#1}}
\newcommand\cls[1]{\mathcal{#1}}

\newcommand\celto{\Rightarrow}
\newcommand\compos[1]{\langle#1\rangle}

\newcommand\atom{{\raisebox{-.02em}{%
\begin{tikzpicture}[baseline={(current bounding box.south)}]%
	\node[circle, draw, line width=.05em, inner sep=.25em] at (0,0) {};%
	\node[circle, fill, inner sep=.06em] at (0,0) {};%
	\node[circle, fill, inner sep=.06em] at (0,.35em) {};%
\end{tikzpicture}}}}

\newcommand\rimp[2]{#1\!\multimap\!#2}
\newcommand\limp[2]{#2\!\multimapinv\!#1}

\newcommand\skel[2]{\sigma_{\leq #1}#2}
\newcommand\coskel[2]{\tau_{\leq #1}#2}
\newcommand\gen[1]{\mathscr{#1}}

\newcommand\intp[1]{\llbracket{#1}\rrbracket}

\newcommand\realis[1]{|#1|}

\newcommand\cat[1]{\mathbf{#1}}
\newcommand\fun[1]{\mathsf{#1}}
\newcommand\theory[1]{\textit{#1}}
\newcommand\psh[2]{\mathrm{PSh}_{#1}(#2)}

\newcommand\omegagphref{\omega\cat{Gph}_\textit{ref}}
\newcommand\omegacat{\omega\cat{Cat}}
\newcommand\ncat[1]{#1\cat{Cat}}
\newcommand\nprecat[1]{#1\cat{PreCat}}
\newcommand\omegaprecat{\omega\cat{PreCat}}
\newcommand\pomegacat{\cat{p}\omega\cat{Cat}}
\newcommand\graycat{\cat{GrayCat}}
\newcommand\brmoncat{\cat{BrMonCat}_\textit{str}}

\newcommand\bipro{\cat{Pro}_\textit{bi}}
\newcommand\pro{\cat{Pro}}
\newcommand\prob{\cat{Prob}}
\newcommand\propp{\cat{Prop}}

\newcommand\rdcpx{\cat{DCpx}^\cls{R}}

\newcommand\dgmset{\atom\cat{Set}}
\newcommand\dgmpoint{\dgmset_\bullet}

\newcommand\cghaus{\cat{cgHaus}}
\newcommand\pointed{\cat{cgHaus}_\bullet}

\newtheoremstyle{ittheorem}
  {\topsep}   
  {\topsep}   
  {\itshape}  
  {0pt}       
  {\itshape \bfseries} 
  { ---}         
  {5pt plus 1pt minus 1pt} 
  {}          

\newtheoremstyle{itdfn}
  {\topsep}   
  {\topsep}   
  {}  
  {0pt}       
  {\bfseries} 
  {.}         
  {5pt plus 1pt minus 1pt} 
  {}          

\makeatletter
  \renewcommand\@upn{\textit}
\makeatother

\theoremstyle{ittheorem}
\newtheorem{thm}{Theorem}[section]
\newtheorem{prop}[thm]{Proposition}
\newtheorem{cor}[thm]{Corollary}
\newtheorem{lem}[thm]{Lemma}

\theoremstyle{itdfn}
\newtheorem{dfn}[thm]{}
\theoremstyle{remark}
\newtheorem{rmk}[thm]{Remark}
\newtheorem{comm}[thm]{Comment}
\newtheorem{exm}[thm]{Example}

\setlist{leftmargin=20pt,itemsep=0pt,topsep=1ex}

\titleformat{\section}
 {\large\scshape}{\thesection.}{1em}{}

\titleformat{\subsection}
 {\normalsize\itshape}{\thesubsection.}{1em}{}
\titlespacing*{\subsection}
{0pt}{1.5ex plus 1ex minus .2ex}{1.5ex plus .2ex}

\renewcommand{\cftsecpagefont}{\mdseries}

\makeatletter \renewcommand{\cftsecfillnum}[1]{%
  {\cftsecleader}\nobreak
  \makebox[\@pnumwidth][\cftpnumalign]{\cftsecpagefont \oldstylenums{#1}}\cftsecafterpnum\par
} \makeatother

\newcommand\runtitle{the smash product of monoidal theories}
\newcommand\runauthor{amar hadzihasanovic}

\title{The smash product of monoidal theories}

\author{Amar Hadzihasanovic}

\institution{Tallinn University of Technology}

\begin{document}


\maketitle 

\noindent\makebox[\textwidth][r]{%
\begin{minipage}[t]{.7\textwidth}
\small \textit{Abstract.}
The tensor product of props was defined by Hackney and Robertson as an extension of the Boardman--Vogt product of operads to more general monoidal theories. Theories that factor as tensor products include the theory of commutative monoids and the theory of bialgebras. We give a topological interpretation (and vast generalisation) of this construction as a low-dimensional projection of a ``smash product of pointed directed spaces''. Here directed spaces are embodied by combinatorial structures called diagrammatic sets, while Gray products replace cartesian products. The correspondence is mediated by a web of adjunctions relating diagrammatic sets, pros, probs, props, and Gray\nbd categories. The smash product applies to presentations of higher-dimensional theories and systematically produces higher-dimensional coherence cells.
\end{minipage}}

\vspace{20pt}

\makeaftertitle

\normalsize


\noindent\makebox[\textwidth][c]{%
\begin{minipage}[t]{.5\textwidth}
\setcounter{tocdepth}{1}
\tableofcontents
\end{minipage}}

\section*{Introduction}

In a categorical tradition of universal algebra dating back to F.\ William Lawvere's thesis \cite{lawvere1963functorial}, algebraic theories are embodied by cartesian monoidal categories whose objects are freely generated from a set of sorts. Models of the theory, embodied by strong monoidal functors, may live in an arbitrary cartesian monoidal category: we specialise to the category of sets and functions to recover the classical notion of model.

Following this fundamental shift in perspective, and considering that cartesianness of a monoidal structure can be defined equationally \cite{fox1976coalgebras}, it is a relatively small step to consider more general \emph{monoidal theories} whose models live in arbitrary monoidal categories. 

Monoidal theories are embodied by structures called \emph{pros} \cite{maclane1965categorical}.\footnote{In some sources, the term \emph{pro} or \emph{PRO} is reserved for a one-sorted theory, and the multi-sorted variant is called a \emph{coloured} pro.} Intermediate between monoidal theories and algebraic theories, there are \emph{braided} and \emph{symmetric} monoidal theories, embodied respectively by \emph{probs} and \emph{props}. The familiar term\nbd algebraic calculus is inadequate for these generalised theories, and is commonly replaced by a calculus of string diagrams \cite{selinger2010survey}. 

It is common for a mathematical object to have both a structure of $T$\nbd model and of $S$\nbd model for some theories $T, S$, satisfying some compatibility condition. For example, a bimodule is both a left and a right module, in such a way that the left and right actions commute. A natural question is: can we systematically \emph{compose} theories, so that a model of the composite of $T$ and $S$ is an object with compatible $T$ and $S$\nbd model structures?

In a line of work that has been attracting attention in theoretical computer science \cite{bonchi2014interacting}, composition of monoidal theories is mediated by distributive laws which specify a factorisation system between operations of $T$ and $S$, as described by Steve Lack \cite{lack2004composing}. 

A less flexible, yet more uniform composition is the \emph{tensor product of props}, which applies to all props and does not require additional data. The tensor product was defined and studied by Philip Hackney and Marcy Robertson \cite{hackney2015category},\footnote{Although its possibility was noticed earlier by John C.\ Baez \cite{baez2006universal}. Baez's lectures are also a nice survey of the relations between pros, probs, props, and algebraic theories.} who also proved that it extends, in a precise sense, the product of symmetric operads introduced by J.\ M.\ Boardman and R.\ M.\ Vogt \cite{boardman2006homotopy}.

Some intuition about the tensor product may be gained as follows. If $\cat{M}$ is a symmetric monoidal category, the category of $T$\nbd models in $\cat{M}$ inherits a symmetric monoidal structure: to compose two models, we ``run their operations in parallel'', using the symmetric structure of $\cat{M}$ to rearrange inputs and outputs as needed. The data of a model in $\cat{M}$ of the tensor product $T \tensorp S$ is equivalent to the data of an $S$\nbd model in the category of $T$\nbd models in $\cat{M}$.\footnote{Or, symmetrically, a $T$\nbd model in the category of $S$\nbd models in $\cat{M}$.} 

As remarked in \cite[\S 4.2]{lack2004composing}, there is something mysterious about the r\^ole of symmetric braidings in the composition of monoidal theories. From a certain perspective, a symmetric braiding is just another operation in a pro, yet it plays an inescapable structural r\^ole in the tensor product. 

Consider the theories of \emph{monoids} and \emph{comonoids}. These are \emph{planar} monoidal theories, naturally embodied by a pro: in the corresponding prop, symmetric braidings are added freely, so models in the sense of props are equivalent to models in the sense of pros. Nevertheless, their tensor product --- the theory of \emph{bialgebras} --- features the non-planar equation
\begin{equation} \label{eq:non_planar}
	\input{img/non_planar}
\end{equation}
where the symmetric braiding in the left-hand side cannot clearly be attributed to either factor.

In particular, the tensor product of props does not restrict to a monoidal structure on pros. At most, as shown in Section \ref{sec:tensor_pros}, we can define an ``external'' tensor product which takes two pros and returns a prob, from which we can then universally reconstruct the tensor product of props.

A few years ago, we noticed that equation (\ref{eq:non_planar}) admits the following topological interpretation.\footnote{A similar observation was made, around the same time, by J.\ Scott Carter \cite{carter2018graphs}.} Take the string diagrams corresponding to monoid multiplication and comonoid comultiplication, and extend them along perpendicular directions in the plane so that they form \emph{branching surfaces}:
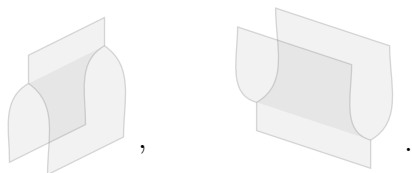
\begin{equation*} 
	\input{img/branching_surfaces}
\end{equation*}
Intersect the two branching surfaces and ``slide'' one past another along the vertical axis. As one branching slides past the other branching, the intersection --- a ``string diagram in 3\nbd dimensional space'' --- evolves as in the following figure:
\begin{equation} \label{eq:sliding_surfaces}
	\input{img/sliding_surfaces}
\end{equation}
The two sides of (\ref{eq:non_planar}) arise as planar projections of the two sides of (\ref{eq:sliding_surfaces}). 

This interpretation extends to all ``compatibility'' equations in the tensor product of props, and recasts the tensor product as a \emph{dimension-raising} construction: given planar diagrams, it produces equations of 3\nbd dimensional diagrams. This solves our conundrum about braidings: they are absent in the 3\nbd dimensional picture, and only appear in the 2\nbd dimensional picture as an artefact of planar projection.

What is going on? As first suggested in \cite[Section 2.3]{hadzihasanovic2017algebra}, the correct interpretation of (\ref{eq:sliding_surfaces}) is that it arises from a \emph{smash product of pointed directed spaces}, in a sense that we will soon explain. 

Our model of directed space is a \emph{diagrammatic set} \cite{hadzihasanovic2020diagrammatic}. We developed the theory of diagrammatic sets partly as a foundation for this work, which requires the ability to do rewriting and diagrammatic reasoning in weak higher categories of arbitrary dimension, to an extent that pre-existing frameworks did not seem to support.

The aim of this article is the statement and proof of Theorem \ref{thm:mainthm}: the ``external'' tensor product of pros factors functorially through the smash product of pointed diagrammatic sets. Through this result, we can attribute a precise meaning to our earlier statements, such as the assertion that equation (\ref{eq:non_planar}) arises from (\ref{eq:sliding_surfaces}). On our way, we develop a great deal of combinatorics in order to relate diagrammatic sets, pros, and probs through a web of adjunctions involving a few ``ancillary'' higher structures.

We see this result not as an end point, but as an opening. Far from being just a reinterpretation, our smash product is a vast generalisation of the tensor product of props, and transitively of the Boardman--Vogt product of operads. 

Indeed, pointed diagrammatic sets can embody higher\nbd dimensional theories with non-invertible generators in arbitrarily high dimension.\footnote{As opposed to structures used in homotopical algebra, such as $\infty$\nbd operads, that embody theories with \emph{invertible} higher data.} From these, the smash product generates non-invertible higher\nbd dimensional cells rather than equations. 

Already when applied to 3\nbd dimensional presentations of monoidal theories \cite{mimram2014towards}, not only this construction produces a \emph{presentation} of their tensor product, that is, it produces \emph{oriented} equations, or rewrites; it also produces interesting higher\nbd dimensional \emph{coherence} cells, or syzygies, up to dimension 6.

In higher\nbd dimensional rewriting and universal algebra \cite{guiraud2016polygraphs,guiraud2019rewriting}, coherence is usually pursued with \emph{analytic} methods of rewriting theory such as the computation of critical branchings. We believe that our results may be a gateway to new \emph{synthetic} and compositional methods.

\subsection*{Monoidal theories, directed spaces, and diagrammatic sets}

The connection between monoidal theories and directed spaces is based on four conceptual leaps. The first leap, as mentioned, is the realisation that monoidal categories can embody algebraic theories. 

The second leap is John C.\ Baez and James Dolan's formulation of the \emph{periodic table of $n$\nbd categories}, by which a monoidal category is equivalent to a bicategory with a single 0\nbd cell, but a braided monoidal category is equivalent to a \emph{tricategory} with a single 0\nbd cell and 1\nbd cell \cite{baez1995higher}. A variant of this result implies that a pro is a special kind of 2\nbd category, while a prob is a special kind of Gray\nbd category, a semistrict notion of tricategory \cite{gordon1995coherence}. This matches the intuition that tensoring pros to obtain a prob is dimension-raising.

The third leap is Grothendieck's \emph{homotopy hypothesis}, that ``spaces'', more precisely homotopy types, are modelled in a precise combinatorial sense by higher groupoids. In models where higher groupoids are higher categories whose cells are all invertible in a weak sense, this leaves open the possibility of interpreting higher categories as ``spaces of directed cells''.

The fourth and final leap is Albert Burroni's observation that various notions of presentations by generators and relations, or rewrite systems, can be unified as presentations of ``cell complexes in a category of higher categories'', a notion of directed space with combinatorial structure \cite{burroni1993higher}.

Following the sequence, we can reinterpret a monoidal theory with its set of sorts as a kind of directed 2\nbd dimensional space containing a 1\nbd dimensional cell complex. A braided monoidal theory is the same thing one dimension up. 

These spaces are canonically pointed with the unique 0\nbd cell in the cell complex structure. It is natural, at this point, to wonder about a directed counterpart of the classical smash product of pointed spaces. The correct generalisation replaces the cartesian product of spaces with a version of the \emph{Gray product} \cite{gray2006formal}. 

In \cite[Section 2.3]{hadzihasanovic2017algebra}, we considered smash products in the context of Burroni and Ross Street's theory of polygraphs, based on strict $\omega$\nbd categories. This had the advantage that a theory of Gray products had already been developed \cite{steiner2004omega,ara2020joint}, and that we could identify a pro directly with a pointed 2\nbd category. However, in this context the smash product of pros produces a strict 3\nbd category equivalent not to a braided monoidal category, but to a highly degenerate \emph{commutative} monoidal category.\footnote{This is connected to the known failure of the homotopy hypothesis for strict $\omega$\nbd categories, see \cite[Chapter 4]{simpson2009homotopy}.}

In \cite{hadzihasanovic2020diagrammatic}, based on an abandoned idea of Mikhail Kapranov and Vladimir Voevodsky, we developed the theory of diagrammatic sets as an alternative to polygraphs that would avoid this pitfall and support rewriting and diagrammatic reasoning in weak higher categories. 

While the model of a directed cell in a polygraph is algebraic, diagrammatic sets adopt a combinatorial model. Roughly, a model of a directed $n$\nbd cell is the face poset of a regular CW\nbd decomposition of the topological closed $n$\nbd ball, together with an orientation subdividing the boundary of each cell into an \emph{input} and \emph{output} half, in such a way that the input and output half are also face posets of regular CW\nbd balls, and their orientations determine a composable pasting diagram in a strict $\omega$\nbd category. Common higher-categorical shapes such as oriented simplices and cubes appear as special cases. 

A pleasant outcome of the transition to the combinatorial setup is that Gray products and smash products are much easier to define and compute. On the other hand, the identification of pros or probs with certain pointed diagrammatic sets is non-trivial. The technical core of this article is the definition of a full and faithful diagrammatic nerve of pros (Section \ref{sec:diag_nerve}), and then of a non-trivial\footnote{There is also a ``trivial'' functor passing through strict 3\nbd categories.} realisation functor of diagrammatic sets in Gray\nbd categories (Section \ref{sec:diag_gray}), which allows us to recover the tensor product of two pros as the realisation of the smash product of their nerves.\footnote{More precisely, of their nerves with an orientation reversal in the second factor, as explained in Section \ref{sec:comparison}.}

Diagrammatic sets are related to (nice) topological spaces by a nerve and realisation pair \cite[Section 8.3]{hadzihasanovic2020diagrammatic}, where the nerve realises a version of the homotopy hypothesis. As detailed in Section \ref{sec:smash_diag}, the geometric realisation sends Gray products to cartesian products, so it sends smash products to smash products. 

Altogether, our results amount to the surprising fact that \emph{the tensor product of pros and the smash product of pointed spaces are two facets of the same construction.}

\subsection*{Related work}

We have paid tribute to our main influences on the conceptual side. On the other hand, this article is technically most indebted to three sources.

The first is Hackney and Robertson's article on the category of props \cite{hackney2015category}: beyond the fact that they defined the tensor product of props, our proofs in Section \ref{sec:theories} that certain categories of pros and probs have small limits and colimits are essentially lifted from their work on props, with minor tweaks.

The second is John Power's work on pasting diagrams \cite{power1991pasting}. While our formalisation of diagrams is based on Richard Steiner's combinatorial framework \cite{steiner1993algebra}, an analogue of Power's \emph{domain replacement condition} turns out to be key to the constructions of Section \ref{sec:prodiag}, and the technical Section \ref{sec:combinatorics} is devoted to showing that it holds for all our 3\nbd dimensional diagrams. 

In particular, our Theorem \ref{thm:acyclic_3} is roughly equivalent in meaning to \cite[Theorem 4.14]{power1991pasting}. Interestingly, though, Power's topological setup seems to have completely different strengths and weaknesses compared to our combinatorial setup: when translating Power's proofs, we discovered that every single non-trivial step in his proofs followed easily from our definitions, whereas the trivial steps would require non-trivial proofs, as in Proposition \ref{prop:sim_substitution}. No formal comparison has been made, to our knowledge, between Steiner's and Power's theory, so we think it is justified to consider our results original. 

The third is Simon Forest and Samuel Mimram's article on the rewriting theory of Gray\nbd categories \cite{forest2018coherence}. Not only we learnt from them a convenient axiomatisation of Gray\nbd categories, but the construction of Section \ref{sec:diag_gray} draws directly on their ideas and results and can be seen as a continuation of their work, showing that every diagrammatic set presents a Gray\nbd category.

\subsection*{Structure of the article}

Most of the article is aimed at the proof of Theorem \ref{thm:mainthm}. 

The statement of this result involves many different structures, related via a number of ancillary structures, each in need of definition. Some of these are obscure enough that basic technical aspects could not be found in the literature and had to be developed expressly. That said, we tried to keep redundancy to a minimum by treating a structure as a special case of another whenever possible, even if it results in unconventional choices, such as the definition of reflexive $\omega$\nbd graphs after diagrammatic sets.

Section \ref{sec:structures} recaps the elementary theory of directed complexes, diagrammatic sets, and strict $\omega$\nbd categories. Section \ref{sec:theories} introduces categories of pros, probs, and props, proves some of their properties, and clarifies the relation between probs and Gray\nbd categories. Section \ref{sec:combinatorics} proves some technical results about directed complexes in low dimension. Section \ref{sec:prodiag} is the technical core of the article, constructing the adjunctions that relate diagrammatic sets, pros, and Gray\nbd categories. Section \ref{sec:smash} defines the tensor product of props and the smash product of pointed diagrammatic sets, then proves the main theorem. Section \ref{sec:higher} takes the first steps into diagrammatic sets as a framework for higher\nbd dimensional rewriting and universal algebra.

Every reader should get at least acquainted with the definitions in the first two sections. On a first read, they can then skip to Section \ref{sec:smash}, using the diagram that concludes Section \ref{sec:prodiag} as a reference: most of the time, knowing that certain functors exist and are left or right adjoints should be enough to follow the outline of the proof.

Some readers may be content with understanding the picture (\ref{eq:sliding_surfaces}) and want to stop there. Those interested in higher\nbd dimensional rewriting and universal algebra should move on to Section \ref{sec:higher}.

Section \ref{sec:combinatorics} may appeal to the reader who appreciates the combinatorics of higher\nbd categorical diagrams. The reader who enjoyed Forest and Mimram's \cite{forest2018coherence} can read Section \ref{sec:prodiag} as a follow-up of sorts.

We use the diagrammatic order $f;g$ for the composition of morphisms $f$ and $g$ in a category, but the ``classical'' order $\fun{GF}$ for the composition of functors $\fun{F}$ and $\fun{G}$. Other notational choices are explained when they are introduced.

\subsection*{Outlook and open problems}

Section \ref{sec:higher} is an extended outlook towards our main prospect, namely, the introduction of new compositional methods in higher\nbd dimensional rewriting and universal algebra.

We briefly mention other potential developments. Christoph Dorn, David Reutter, and Jamie Vicary have defined a semistrict algebraic model of $n$\nbd categories, called \emph{associative $n$\nbd categories}, which is equivalent to Gray\nbd categories for $n = 3$ \cite{dorn2018associative,reutter2019high}. It is conceivable that the adjunction of Section \ref{sec:diag_gray} relating diagrammatic sets to Gray\nbd categories may generalise to associative $n$\nbd categories for $n > 3$. We note, however, that our construction uses a property, frame acyclicity, which holds in general up to dimension 3 but fails in dimension 4 or higher, so it is likely that new ideas will be needed.

The theory of diagrammatic sets is based on simple data structures: a cell model $U$ can be encoded as the directed graph $\hasseo{U}$ of \S \ref{dfn:hasseo} together with a grading of its vertices; the Gray product is then encoded as a cartesian product of directed graphs,\footnote{With some edges reversed, depending on the degree of first factor.} while the degrees of vertices are summed. We expect that this setup should lend itself to computational formalisation.

This is of particular interest considering that the theory of associative $n$\nbd categories is formalised in the graphical proof assistant \texttt{homotopy.io}: implementing the constructions of Section \ref{sec:diag_gray} would give us access to visualisations of Gray and smash products through this graphical frontend.

\clearpage

\section{Some higher structures} \label{sec:structures}

\subsection{Directed complexes and diagrammatic sets}

We quickly go through the main definitions, and refer the reader to \cite{hadzihasanovic2020diagrammatic} for an in-depth development.

\begin{dfn}[Graded poset]
Let $P$ be a finite poset with order relation $\leq$. For all elements $x, y \in P$, we say that $y$ \emph{covers} $x$ if $x < y$ and, for all $y' \in X$, if $x < y' \leq y$ then $y' = y$. 

The \emph{Hasse diagram} of $P$ is the finite directed graph $\hasse{P}$ with $\hasse{P}_0 \eqdef P$ as set of vertices and $\hasse{P}_1 \eqdef \{y \to x \mid y \text{ covers } x\}$ as set of edges. 

Let $P_\bot$ be $P$ extended with a least element $\bot$. We say that $P$ is \emph{graded} if, for all $x \in P$, all directed paths from $x$ to $\bot$ in $\hasse{P_\bot}$ have the same length. If this length is $n+1$, we let $\dmn{x} \eqdef n$ be the \emph{dimension} of $x$.
\end{dfn}

\begin{dfn}[Closed and pure subsets]
Let $P$ be a poset and $U \subseteq P$. The \emph{closure} of $U$ is the subset $\clos{U} \eqdef \{x \in P \mid \exists y \in U \; x \leq y\}$ of $P$. We say that $U$ is \emph{closed} if $U = \clos{U}$. 

Suppose $P$ is graded and $U \subseteq P$ is closed. Then $U$ is graded with the partial order inherited from $P$. The \emph{dimension} $\dmn{U}$ of $U$ is $\max\{\dmn{x} \mid x \in U\}$ if $U$ is inhabited, $-1$ otherwise. In particular, $\dmn{\clos{\{x\}}} = \dmn{x}$. 

We say that $U$ is \emph{pure} if its maximal elements all have dimension $\dmn{U}$.
\end{dfn}

\begin{dfn}[Oriented graded poset]
An \emph{orientation} on a finite poset $P$ is an edge-labelling $o\colon \hasse{P}_1 \to \{+,-\}$ of its Hasse diagram.

An \emph{oriented graded poset} is a finite graded poset with an orientation.
\end{dfn}

\begin{dfn} 
We will often let variables $\alpha, \beta$ range implicitly over $\{+,-\}$.
\end{dfn}

\begin{dfn}[Boundaries]
Let $P$ be an oriented graded poset and $U \subseteq P$ a closed subset. Then $U$ inherits an orientation from $P$ by restriction. 

For all $\alpha \in \{+,-\}$ and $n \in \mathbb{N}$, we define
\begin{align*}
	\sbord{n}{\alpha} U & \eqdef \{x \in U \mid \dmn{x} = n \text{ and if $y \in U $ covers $x$, then $o(y \to x) = \alpha$} \}, \\
	\bord{n}{\alpha} U & \eqdef \clos{(\sbord{n}{\alpha} U)} \cup \{ x \in U \mid \text{for all $y \in U$, if $x \leq y$, then $\dmn{y} \leq n$} \}, \\
	\sbord{n} U & \eqdef \sbord{n}{+}U \cup \sbord{n}{-}U, \quad \qquad \quad \bord{n}U \eqdef \bord{n}{+}U \cup \bord{n}{-} U.
\end{align*}
We call $\bord{n}{-}U$ the \emph{input $n$\nbd boundary} and $\bord{n}{+}U$ the \emph{output $n$\nbd boundary} of $U$. 

If $U$ is $(n+1)$\nbd dimensional, we write $\sbord{}{\alpha}U \eqdef \sbord{n}{\alpha}U$ and $\bord{}{\alpha}U \eqdef \bord{n}{\alpha}U$. For each $x \in P$, we write $\sbord{n}{\alpha}x \eqdef \sbord{n}{\alpha}\clos{\{x\}}$ and $\bord{n}{\alpha}x \eqdef \bord{n}{\alpha}\clos{\{x\}}$.
\end{dfn}

\begin{dfn}[Atoms and molecules]
Let $P$ be an oriented graded poset. We define a family of closed subsets of $P$, the \emph{molecules} of $P$, by induction on proper subsets. If $U$ is a closed subset of $P$, then $U$ is a molecule if either
\begin{itemize}
	\item $U$ has a greatest element, in which case we call it an \emph{atom}, or
	\item there exist molecules $U_1$ and $U_2$, both properly contained in $U$, and $n \in \mathbb{N}$ such that $U_1 \cap U_2 = \bord{n}{+}U_1 = \bord{n}{-}U_2$ and $U = U_1 \cup U_2$.
\end{itemize}
We define $\submol$ to be the smallest partial order relation such that, if $U_1$ and $U_2$ are molecules and $U_1 \cap U_2 = \bord{n}{+}U_1 = \bord{n}{-}U_2$, then $U_1, U_2 \submol U_1 \cup U_2$.

We say \emph{$n$\nbd molecule} for an $n$\nbd dimensional molecule. We say that $P$ itself is a molecule if $P \subseteq P$ is a molecule.
\end{dfn}

\begin{dfn}[Spherical boundary] 
An $n$\nbd molecule $U$ in an oriented graded poset \emph{has spherical boundary} if, for all $k < n$,
\begin{equation*}
	\bord{k}{+}U \cap \bord{k}{-}U = \bord{k-1}{}U.
\end{equation*}
\end{dfn}

\begin{dfn}[Regular directed complex]
An oriented graded poset $P$ is a \emph{regular directed complex} if, for all $x \in P$ and $\alpha, \beta \in \{+,-\}$, 
\begin{enumerate}
	\item $\clos\{x\}$ has spherical boundary,
	\item $\bord{}{\alpha}x$ is a molecule, and
	\item $\bord{}{\alpha}(\bord{}{\beta}x) = \bord{n-2}{\alpha}x$ if $n \eqdef \dmn{x} > 1$.
\end{enumerate}
A \emph{map} $f\colon P \to Q$ of regular directed complexes is a function of their underlying sets that satisfies
\begin{equation*}
	\bord{n}{\alpha}f(x) = f(\bord{n}{\alpha}x)
\end{equation*}
for all $x \in P$, $n \in \mathbb{N}$, and $\alpha \in \{+,-\}$. We call an injective map an \emph{inclusion}. With their maps, regular directed complexes form a category $\rdcpx$.
\end{dfn}

\begin{rmk}
As shown in \cite[Section 1.3]{hadzihasanovic2020diagrammatic}, $\rdcpx$ has an initial object, a terminal object, and pushouts of inclusions.
\end{rmk}

\begin{dfn}[Regular molecule]
A \emph{regular molecule} is a molecule which is a regular directed complex. 

By \cite[Proposition 1.38]{hadzihasanovic2020diagrammatic}, if two regular molecules are isomorphic in $\rdcpx$, they are isomorphic in a unique way. As customary in these situations, we will treat isomorphic regular molecules as ``equal'' under appropriate circumstances.
\end{dfn}

\begin{dfn}[Globe]
For each $n \in \mathbb{N}$, let $O^n$ be the poset with a pair of elements $\undl{k}^+, \undl{k}^-$ for each $k < n$ and a greatest element $\undl{n}$, with the partial order defined by $\undl{j}^\alpha \leq \undl{k}^\beta$ if and only if $j \leq k$. This is a graded poset, with $\dmn{\undl{n}} = n$ and $\dmn{\undl{k}^\alpha} = k$ for all $k < n$. 

With the orientation $o(y \to \undl{k}^\alpha) \eqdef \alpha$ if $y$ covers $\undl{k}^\alpha$, $O^n$ becomes a regular directed complex, in particular a regular atom. We call $O^n$ the \emph{$n$\nbd globe}.
\end{dfn}

\begin{dfn}[Pasting of molecules]
Let $U_1, U_2$ be regular molecules and suppose that $\bord{k}{+}U_1$ and $\bord{k}{-}U_2$ are isomorphic in $\rdcpx$. Given an isomorphic copy $V$ of the two, there is a unique span of inclusions $V \incl U_1$ and $V \incl U_2$ whose images are, respectively, $\bord{k}{+}U_1$ and $\bord{k}{-}U_2$. We let $U_1 \cp{k} U_2$ be the pushout
\begin{equation*}
\begin{tikzpicture}[baseline={([yshift=-.5ex]current bounding box.center)}]
	\node (0) at (0,1.5) {$V$};
	\node (1) at (2.5,0) {$U_1 \cp{k} U_2$};
	\node (2) at (0,0) {$U_1$};
	\node (3) at (2.5,1.5) {$U_2$};
	\draw[1cinc] (0) to (3);
	\draw[1cincl] (0) to (2);
	\draw[1cinc] (2) to (1);
	\draw[1cincl] (3) to (1);
	\draw[edge] (1.6,0.2) to (1.6,0.7) to (2.3,0.7);
\end{tikzpicture}
\end{equation*}
in $\rdcpx$. Then $U_1 \cp{k} U_2$ is a regular molecule, decomposing as $U_1 \cup U_2$ with $U_1 \cap U_2 = \bord{k}{+}U_1 = \bord{k}{-}U_2$. 
\end{dfn}

\begin{dfn}[$- \celto -$ construction]
Let $U, V$ be regular $n$\nbd molecules with spherical boundary such that $\bord{}{\alpha}U$ is isomorphic to $\bord{}{\alpha}V$ for all $\alpha \in \{+,-\}$.

Form the pushout $U \cup V$ of the span of inclusions $\bord U \incl U$, $\bord U \incl V$ whose images are $\bord U$ and $\bord V$, respectively. We define $U \celto V$ to be the oriented graded poset obtained from $U \cup V$ by adjoining a greatest element $\top$ with $\bord{}{-}\top \eqdef U$ and $\bord{}{+}\top \eqdef V$. Then $U \celto V$ is an $(n+1)$\nbd dimensional atom with spherical boundary.
\end{dfn}

\begin{dfn}[{$\compos{-}$ construction}]
Let $U$ be a regular molecule with spherical boundary. Then $\bord{}{-}U \celto \bord{}{+}U$ is defined, and we denote it by $\compos{U}$. 
\end{dfn}

\begin{dfn}
There is a unique 0\nbd atom, namely, the $0$\nbd globe $1 \eqdef O^0$, which is also the terminal object of $\rdcpx$.

We define a sequence $\{I_n\}_{n > 0}$ of 1\nbd molecules by 
\begin{equation*}
	I_1 \eqdef O^1, \quad \quad I_n \eqdef I_{n-1} \cp{0} O^1 \text{ for $n > 1$}.
\end{equation*}
Every regular 1\nbd molecule is of the form $I_n$ for some $n > 0$.

For each pair $n, m > 0$, let $U_{n,m} \eqdef (I_n \celto I_m)$. Every regular 2\nbd atom is of the form $U_{n,m}$ for some $n, m > 0$. Regular 2\nbd molecules are then generated by $I_1$ and the $U_{n,m}$ under the pasting operations $\cp{0}, \cp{1}$.
\end{dfn}

\begin{dfn}[Diagrammatic set]
We write $\atom$ for a skeleton of the full subcategory of $\rdcpx$ on the atoms of every dimension. 

A \emph{diagrammatic set} is a presheaf on $\atom$. Diagrammatic sets and their morphisms of presheaves form a category $\dgmset$.
\end{dfn}

\begin{comm}
The definition in \cite{hadzihasanovic2020diagrammatic} is relative to a fixed ``convenient'' class of molecules; for simplicity, here we pick the class of all molecules with spherical boundary.
\end{comm}

\begin{dfn}
We identify $\atom$ with a full subcategory $\atom \incl \dgmset$ via the Yoneda embedding. With this identification, we use morphisms in $\dgmset$ as our notation for both elements and structural operations of a diagrammatic set $X$:
\begin{itemize}
	\item $x \in X(U)$ becomes $x\colon U \to X$, and
	\item for each map $f\colon V \to U$ in $\atom$, $X(f)(x) \in X(V)$ becomes $f;x\colon V \to X$.
\end{itemize}
As described in \cite[\S 4.4]{hadzihasanovic2020diagrammatic}, the embedding $\atom \incl \dgmset$ extends to an embedding $\rdcpx \incl \dgmset$.
\end{dfn}

\begin{dfn}[Diagrams and cells] 
Let $X$ be a diagrammatic set and $U$ a regular molecule. A \emph{diagram of shape $U$ in $X$} is a morphism $x\colon U \to X$. It is \emph{composable} if $U$ has spherical boundary and a \emph{cell} if $U$ is an atom. For all $n \in \mathbb{N}$, we say that $x$ is \emph{$n$\nbd diagram} or an \emph{$n$\nbd cell} when $\dmn{U} = n$.

If $U$ decomposes as $U_1 \cp{k} U_2$, we write $x = x_1 \cp{k} x_2$ for $x_i \eqdef \imath_i;x$, where $\imath_i$ is the inclusion $U_i \incl U$ for $i \in \{1,2\}$. This extends associatively to $n$\nbd ary decompositions for $n > 2$.

If $x\colon U \to X$ is a diagram in $X$ and $f\colon X \to Y$ a morphism of diagrammatic sets, we may write $f(x)$ for the diagram $x;f\colon U \to Y$. 
\end{dfn}

\begin{dfn}[Boundaries of diagrams] 
Let $X$ be a diagrammatic set, $x\colon U \to X$ a diagram, and let $\imath_k^\alpha\colon \bord{k}{\alpha}U \incl U$ be the inclusions of the $k$\nbd boundaries of $U$. The \emph{input $k$\nbd boundary} of $x$ is the diagram $\bord{k}{-}x \eqdef \imath_k^-;x$ and the \emph{output $k$\nbd boundary} of $x$ is the diagram $\bord{k}{+}x \eqdef \imath_k^+;x$. We may omit the index $k$ when $k = \dmn{U} -1$. 

We write $x\colon y^- \celto y^+$ to express that $\bord{k}{\alpha}x = y^\alpha$ for each $\alpha \in \{+,-\}$, and say that $x$ is of \emph{type} $y^- \celto y^+$. We say that two diagrams $x_1, x_2$ are \emph{parallel} if they have the same type.
\end{dfn}

\begin{dfn}
A 1\nbd cell $a$ in a diagrammatic set has shape $I_1$. A 2\nbd cell $\varphi$ has shape $U_{n,m}$ for some $n, m > 0$, so it is of type 
\begin{equation*}
	a_1 \cp{0} \ldots \cp{0} a_n \celto b_1 \cp{0} \ldots \cp{0} b_m
\end{equation*}
for some 1\nbd cells $a_1,\ldots,a_n, b_1,\ldots, b_m$. We may depict such cells as string diagrams
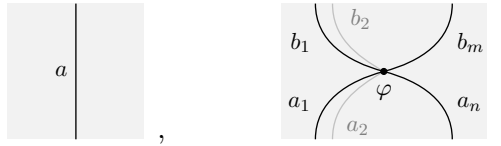
\begin{equation*}
	\input{img/generic_2cell}
\end{equation*}
where a lighter shade indicates a repeated pattern. Each region bounded by wires corresponds to a potentially different 0\nbd cell; in practice, we will mostly work with diagrammatic sets that have a single 0\nbd cell. Labels will be omitted when irrelevant, or implied by the shape of a cell.

A 2\nbd diagram decomposes into 1\nbd cells and 2\nbd cells under the $\cp{0}, \cp{1}$ operations. In string diagrams, $\cp{0}$ is horizontal juxtaposition and $\cp{1}$ is vertical juxtaposition with the output wires of one diagram connecting to the input wires of another. For example,
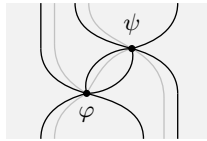
\begin{equation*}
	\input{img/generic_2diagram}
\end{equation*}
depicts a generic 2\nbd diagram of the form
\begin{equation*}
	(\varphi \cp{0} a_1 \cp{0} \ldots \cp{0} a_n) \cp{1} (b_1 \cp{0} \ldots \cp{0} b_m \cp{0} \psi)
\end{equation*}
for some 2\nbd cells $\varphi, \psi$ and 1\nbd cells $a_1, \ldots, a_n, b_1,\ldots, b_m$.

Observe that, in our setting, there is no need to attribute a topological nature to string diagrams, \emph{\`a la} Joyal and Street \cite{joyal1991geometry}: they should instead be interpreted as compact encodings of regular molecules -- a discrete, combinatorial structure -- and their morphisms to diagrammatic sets. 

A 2\nbd diagram is composable if and only if it is \emph{connected} as a string diagram. For example, of the diagrams
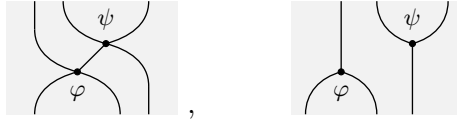
\begin{equation*}
	\input{img/nonconnected_2diagram}
\end{equation*}
only the first one is composable. 

We may depict a 3\nbd diagram as a \emph{sequence of rewrites} on composable subdiagrams of a diagram. For example, a diagrammatic set with a single 0\nbd cell, a single 1\nbd cell, and 3\nbd cells $\varphi, \psi$ of the form
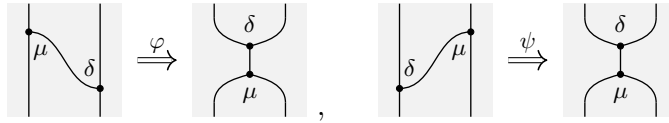
\begin{equation*}
	\input{img/frobenius}
\end{equation*}
admits a 3\nbd diagram of the form
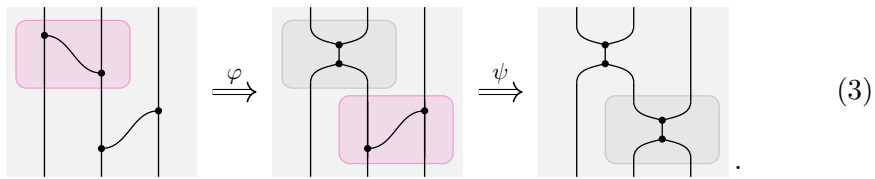
\begin{equation} \label{eq:frobenius_diagram}
	\input{img/frobenius_diagram}
\end{equation}
where the input boundary of each 3\nbd cell is highlighted in pink.

We will also use string diagrams to describe certain regular molecules directly. This is justified by the interpretation of a molecule $U$ as the ``tautologous'' diagram $\idd{U}\colon U \to U$ in $\dgmset$.
\end{dfn}

\begin{dfn}[Dual diagrammatic set]
Let $U$ be a regular atom. The oriented graded poset $\oppall{U}$ with the same underlying poset as $U$ and the opposite orientation $\oppall{o}(y \to x) \eqdef -o(y \to x)$ is a regular atom. If $f\colon U \to V$ is a map in $\atom$, its underlying function also defines a map $\oppall{f}\colon \oppall{U} \to \oppall{V}$. This determines an involution $\oppall{-}$ on $\atom$. 

Let $X$ be a diagrammatic set. Its \emph{dual} $\oppall{X}$ is the diagrammatic set defined by $\oppall{X}(-) \eqdef X(\oppall{-})$. This extends to morphisms in the obvious way, and extends the involution on $\atom$ to an involution on $\dgmset$.
\end{dfn}

\begin{rmk}
If $x\colon U \to X$ is a 2\nbd diagram, the depiction of $\oppall{x}\colon \oppall{U} \to \oppall{X}$ in string diagrams is the horizontal and vertical reflection of the depiction of $x$. 
\end{rmk}


\subsection{Higher-categorical structures}

\begin{dfn}[Reflexive $\omega$\nbd graph]
Let $\cat{O}$ be the full subcategory of $\atom$ whose objects are the globes $O^n$. For all $n$ and $k < n$,
\begin{itemize}
	\item the $k$\nbd boundary inclusions $\imath_k^+, \imath_k^-$ are the only inclusions of $O^k$ into $O^n$;
	\item the map $\tau\colon O^n \surj O^k$, defined by $\tau(\undl{n}),\tau(\undl{j}^\alpha) \eqdef \undl{k}$ if $j \geq k$ and $\tau(\undl{j}^\alpha) \eqdef \undl{j}^\alpha$ if $j < k$, is the only surjective map from $O^n$ onto $O^k$. 
\end{itemize}
A \emph{reflexive $\omega$\nbd graph} is a presheaf $X$ on $\cat{O}$. With their morphisms of presheaves, reflexive $\omega$\nbd graphs form a category $\omegagphref$.
\end{dfn}

\begin{dfn}
The embedding $\cat{O} \incl \atom$ induces a restriction functor $\dgmset \to \omegagphref$ with a full and faithful left adjoint $\omegagphref \incl \dgmset$; we can thus identify reflexive $\omega$\nbd graphs with particular diagrammatic sets, and use for them the same terminology and notation.

Because all $n$\nbd cells in a reflexive $\omega$\nbd graph $X$ have the same shape $O^n$, we leave it implicit and write $X_n \eqdef X(O^n)$.
\end{dfn}

\begin{dfn}[Units]
Let $x$ be a $k$\nbd cell in a reflexive $\omega$\nbd graph $X$. For $n > k$, we let $\eps{n}x \eqdef \tau;x$ where $\tau$ is the unique surjective map $O^n \surj O^k$. We call $\eps{n}x$ a \emph{unit} on $x$. We may omit the index when $n = k+1$. 
\end{dfn}

\begin{dfn}[Rank of a cell]
Let $x$ be an $n$\nbd cell in a reflexive $\omega$\nbd graph. The \emph{rank} $\rank{x}$ of $x$ is defined inductively on $n$ as follows:
\begin{itemize}
	\item if $n = 0$, then $\rank{x} \eqdef 0$;
	\item if $n > 0$, if $x = \eps{}y$ for an $(n-1)$\nbd cell $y$, then $\rank{x} \eqdef \rank{y}$, otherwise $\rank{x} \eqdef n$.
\end{itemize}
\end{dfn}

\begin{dfn}[Partial $\omega$\nbd category]
A \emph{partial $\omega$\nbd category} is a reflexive $\omega$\nbd graph $X$ together with partial $k$-composition operations
\begin{equation*}
	\cp{k}\colon X_n \times X_n \pfun X_n
\end{equation*}
for all $n \in \mathbb{N}$ and $k < n$, satisfying the following axioms:
\begin{enumerate}
	\item for all $n$\nbd cells $x, y$ and all $k < n$ such that $x \cp{k} y$ is defined, 
	\begin{equation*}
		\bord{k}{+}x = \bord{k}{-}y \text{ and } \eps{}(x \cp{k} y) = \eps{}x \cp{k} \eps{}y;
	\end{equation*}
	\item for all $n$\nbd cells $x$ and all $k < n$, the $k$\nbd compositions
	\begin{equation*}
		x \cp{k} \eps{n}(\bord{k}{+}x) \text{ and } \eps{n}(\bord{k}{-}x) \cp{k} x
	\end{equation*}
	are defined and equal to $x$;
	\item for all $(n+1)$\nbd cells $x, y$ and $k < n$, whenever the left-hand side is defined, the right-hand side is defined and
	\begin{align*}
		\bord{}{-}(x \cp{n} y) & = \bord{}{-} x, \\
		\bord{}{+}(x \cp{n} y) & = \bord{}{+} y, \\
		\bord{}{\alpha}(x \cp{k} y) & = \bord{}{\alpha} x \cp{k} \bord{}{\alpha} y;
	\end{align*}
	\item for all cells $x, y, z$ and all $k$ such that both sides are defined,
	\begin{equation*}
		(x \cp{k} y) \cp{k} z = x \cp{k} (y \cp{k} z);
	\end{equation*}
	\item for all cells $x, y, x', y'$, all $n$ and all $k < n$ such that both sides are defined,
	\begin{equation} \label{eq:interchange}
		(x \cp{n} x') \cp{k} (y \cp{n} y') = (x \cp{k} y) \cp{n} (x' \cp{k} y').
	\end{equation}
\end{enumerate} 
A \emph{functor} $f\colon X \to Y$ of partial $\omega$\nbd categories is a morphism of the underlying reflexive $\omega$\nbd graphs such that, for all cells $x, y$ in $X$, if $x \cp{n} y$ is defined in $X$ then $f(x) \cp{n} f(y)$ is defined and equal to $f(x \cp{n} y)$ in $Y$. Partial $\omega$\nbd categories and their functors form a category $\pomegacat$. 
\end{dfn}

\begin{dfn}
We will generally confuse the notation for a $k$\nbd cell and the units on it: for example, if $x$ is an $n$\nbd cell and $y$ a $k$\nbd cell, $k < n$, such that $x \cp{m} \eps{n} y$ is defined, we will write $x \cp{m} y \eqdef x \cp{m} \eps{n}y$.
\end{dfn}

\begin{dfn}[$\omega$\nbd Precategory]
An \emph{$\omega$\nbd precategory} is a partial $\omega$\nbd category $X$ such that, for all $n$\nbd cells $x, y$ in $X$, the $k$\nbd composition $x \cp{k} y$ is defined if and only if $\bord{k}{+}x = \bord{k}{-}y$ and $\min \{\rank{x}, \rank{y}\} \leq k+1$. With their functors, $\omega$\nbd precategories form a category $\omegaprecat$.
\end{dfn}

\begin{dfn}[$\omega$\nbd Category]
An \emph{$\omega$\nbd category} is a partial $\omega$\nbd category such that, for all $n$\nbd cells $x, y$ in $X$, the $k$\nbd composition $x \cp{k} y$ is defined if and only if $\bord{k}{+}x = \bord{k}{-}y$. With their functors, $\omega$\nbd categories form a category $\omegacat$.
\end{dfn}

\begin{dfn}
The inclusion $\omegacat \incl \pomegacat$ has a left adjoint $-^*\colon \pomegacat \to \omegacat$; if $X$ is a partial $\omega$\nbd category, then $X^*$ is the free $\omega$\nbd category on the underlying reflexive $\omega$\nbd graph of $X$, quotiented by all the equations involving compositions that are defined in $X$. 

By \cite[Proposition 1.23]{hadzihasanovic2020diagrammatic}, if $P$ is a regular directed complex, there is a  partial $\omega$\nbd category $\mol{}{P}$ where
\begin{enumerate}
	\item the set $\mol{n}{P}$ of $n$-cells is the set of molecules $U \subseteq P$ with $\dmn{U} \leq n$,
	\item $\bord{k}{\alpha}\colon \mol{n}{P} \to \mol{k}{P}$ is $U \mapsto \bord{k}{\alpha}U$,
	\item $\eps{n}\colon \mol{k}{P} \to \mol{n}{P}$ is $U \mapsto U$,
	\item $U \cp{k} V$ is defined if and only if $U \cap V = \bord{k}{+}U = \bord{k}{-}V$, and in that case it is equal to $U \cup V$.
\end{enumerate}
We will write $W = U \cp{k} V$ to indicate that $W$ is a molecule decomposing as $U \cup V$, where $U$ and $V$ are molecules with $U \cap V = \bord{k}{+}U = \bord{k}{-}V$.

As detailed in [Section 7, \emph{ibid.}], the assignment $P \mapsto \mol{}{P}^*$ extends to a functor $\mol{}{-}^*\colon \rdcpx \to \omegacat$ which is faithful and injective on objects.
\end{dfn}

\begin{dfn}[Principal composition] \label{dfn:principal}
Let $x, y$ be $n$\nbd cells in an $\omega$\nbd precategory or an $\omega$\nbd category, and let $k \eqdef \min \{\rank{x},\rank{y}\} - 1$. If $\bord{k}{+}x = \bord{k}{-}y$, the \emph{principal composition} of $x$ and $y$ is
\begin{equation*}
	x \cp{} y \eqdef x \cp{k} y.
\end{equation*}
\end{dfn}

\begin{comm}
Given cells $x, y$ in an $\omega$\nbd precategory, suppose that $x \cp{k} y$ is defined. Then either 
\begin{itemize}
	\item $\min \{\rank{x},\rank{y}\} = k+1$, in which case $x \cp{k} y = x \cp{} y$, or 
	\item $\rank{x} \leq k$, in which case $x \cp{k} y = y$, or 
	\item $\rank{y} \leq k$, in which case $x \cp{k} y = x$. 
\end{itemize}
This implies that $\omega$\nbd precategories admit an axiomatisation involving only principal compositions, at the cost of explicitly handling some corner cases in the axioms.

Moreover, the two sides of (\ref{eq:interchange}) can both be defined only if 
\begin{itemize}
	\item both $x$ and $x'$ have rank lower or equal than $k+1$, in which case (\ref{eq:interchange}) is equivalent to 
	\begin{equation*}
		x \cp{k} (y \cp{n} y') = (x \cp{k} y) \cp{n} (x \cp{k} y'),
	\end{equation*}
	or dually
	\item both $y$ and $y'$ have rank lower or equal than $k+1$, in which case (\ref{eq:interchange}) is equivalent to 
	\begin{equation*}
		(x \cp{n} x') \cp{k} y = (x \cp{k} y) \cp{n} (x' \cp{k} y).
	\end{equation*}
\end{itemize}
These two observations can be used to establish an equivalence between our definition of $\omega$\nbd precategory and the one given in \cite[Section 4.1]{forest2018coherence}.
\end{comm}

\begin{dfn}
There is a forgetful functor $\fun{U}\colon \omegacat \to \omegaprecat$ which makes $x \cp{k} y$ undefined whenever $\min \{\rank{x},\rank{y}\} > k+1$.
\end{dfn}

\begin{prop} \label{prop:omega_preomega}
The functor $\fun{U}\colon \omegacat \to \omegaprecat$ is full and faithful. Its image consists of the $\omega$\nbd precategories satisfying
\begin{equation} \label{eq:preinterchange}
	(x \cp{k-1} \bord{k}{-} y) \cp{k} (\bord{k}{+}x \cp{k-1} y) = (\bord{k}{-}x \cp{k-1} y) \cp{k} (x \cp{k-1} \bord{k}{+} y)
\end{equation}
for all cells $x, y$ with $\min \{\rank{x},\rank{y}\} = k + 1$ and $\bord{k-1}{+}x = \bord{k-1}{-}y$.
\end{prop}
\begin{proof}
First of all, observe that both sides of (\ref{eq:preinterchange}) are defined in an $\omega$\nbd precategory when $x, y$ satisfy the conditions of the statement. If this precategory is of the form $\fun{U}X$ for some $\omega$\nbd category $X$, then both sides are equal to $x \cp{k-1} y$ in $X$, so (\ref{eq:preinterchange}) is satisfied. 

Let $X'$ be an $\omega$\nbd precategory such that (\ref{eq:preinterchange}) holds for all cells in $X'$ in the conditions of the statement. We will define an $\omega$\nbd category $X$ such that $\fun{U}X = X'$, which necessarily has the same underlying reflexive $\omega$\nbd graph as $X'$.

Let $x, y$ be cells such that $\bord{k-1}{+}x = \bord{k-1}{-}y$ for some $k > 0 $. We must define $x \cp{k-1} y$ in $X$. If $\min \{\rank{x},\rank{y}\} \leq k$, then $x \cp{k-1} y$ is defined in $X'$, and we declare it to be the same in $X$.

Otherwise, $\min \{\rank{x},\rank{y}\} = k + 1 + m$ for some $m \geq 0$. If $m = 0$, we define $x \cp{k-1} y$ to be equal to either side of (\ref{eq:preinterchange}). For $m > 0$, observe that 
\begin{equation*}
	\min \{\rank{x}, \rank{\bord{k+m}{\alpha}y}\}, \min \{\rank{\bord{k+m}{\alpha}x}, \rank{y}\} \leq k + m
\end{equation*}
for all $\alpha \in \{+,-\}$, but 
\begin{equation*}
	\bord{k-1}{+}x = \bord{k-1}{+}(\bord{k+m}{\alpha}x) = \bord{k-1}{-}(\bord{k+m}{\alpha}y) = \bord{k-1}{-}y.
\end{equation*}
We may thus assume, inductively, that $x \cp{k-1} \bord{k+m}{\alpha}y$ and $\bord{k+m}{\alpha}x \cp{k-1} y$ have already been defined, and let
\begin{equation*}
	x \cp{k-1} y \eqdef (x \cp{k-1} \bord{k+m}{-}y) \cp{k+m} (\bord{k+m}{+}x \cp{k-1} y).
\end{equation*}
It is an exercise to show, by induction, that this is equal to 
\begin{equation*}
	(\bord{k+m}{-}x \cp{k-1} y) \cp{k+m} (x \cp{k-1} \bord{k+m}{+}y)
\end{equation*}
and derive that $X$ is an $\omega$\nbd category. Because all the definitions are enforced by the axioms of $\omega$\nbd categories, $X$ is unique with the property that $\fun{U}X = X'$. 

Since all compositions in $X$ are defined in terms of compositions in $\fun{U}X$, every functor $f'\colon \fun{U}X \to \fun{U}Y$ of $\omega$\nbd precategories lifts to a functor $f\colon X \to Y$ of $\omega$\nbd categories. This proves fullness; faithfulness is immediate from the fact that $f$ and $\fun{U}f$ have the same underlying morphism of reflexive $\omega$\nbd graphs.
\end{proof}

\begin{dfn}[Skeleta]
Let $X$ be an $\omega$\nbd (pre)category, $n \in \mathbb{N}$. The \emph{$n$\nbd skeleton} $\skel{n}{X}$ of $X$ is the restriction of $X$ to cells of rank $\leq n$. We let $\skel{-1}{X} \eqdef \emptyset$, the initial $\omega$\nbd (pre)category. The $n$\nbd skeleton operation extends functorially to morphisms in the obvious way.
\end{dfn}

\begin{dfn}[$n$\nbd Category]
An $\omega$\nbd (pre)category is an \emph{$n$\nbd (pre)category} if it is equal to its $n$\nbd skeleton. An $n$\nbd (pre)category is determined by its restriction to $k$\nbd cells with $k \leq n$.

Let $\nprecat{n}$ denote the full subcategory of $\omegaprecat$ and $\ncat{n}$ the full subcategory of $\omegacat$ on the $n$\nbd (pre)categories. In both cases, the inclusion of subcategories has a right adjoint and $\skel{n}{}$ is the comonad induced by the adjunction.

The inclusion also has a left adjoint, inducing a monad $\coskel{n}{}$: given $X$, the $n$\nbd (pre)category $\coskel{n}{X}$ is obtained from $\skel{n}{X}$ by identifying all pairs of $n$\nbd cells $x,y$ such that there exists an $(n+1)$\nbd cell $e\colon x \celto y$ in $X$. 
\end{dfn}

\begin{rmk}
Both $\omegaprecat$ and $\omegacat$ are categories of algebras for finitary monads on $\omegagphref$, a presheaf topos. By the Remark at the end of \cite[\S 2.78]{adamek1994locally}, they are locally finitely presentable, and in particular have all small limits and colimits. The same applies to their reflective subcategories $\nprecat{n}$ and $\ncat{n}$ for all $n \in \mathbb{N}$.
\end{rmk}

\begin{dfn}[Polygraph]
Let $\bord O^n \eqdef \skel{n-1}{O^n}$. 

A \emph{(pre)polygraph} is an $\omega$\nbd (pre)category $X$ together with a set $\gen{X} = \sum_{n \in \mathbb{N}}\gen{X}_n$ of \emph{generating} cells such that, for all $n \in \mathbb{N}$, 
\begin{equation*}
\begin{tikzpicture}[baseline={([yshift=-.5ex]current bounding box.center)}]
	\node (0) at (-.5,1.5) {$\coprod_{x \in \gen{X}_n} \bord O^n$};
	\node (1) at (2.5,0) {$\skel{n}{X}$};
	\node (2) at (-.5,0) {$\skel{n-1}{X}$};
	\node (3) at (2.5,1.5) {$\coprod_{x \in \gen{X}_n} O^n$};
	\draw[1cinc] (0) to (3);
	\draw[1c] (0) to (2);
	\draw[1cinc] (2) to (1);
	\draw[1c] (3) to node[auto,arlabel] {$(x)_{x \in \gen{X}_n}$} (1);
	\draw[edge] (1.6,0.2) to (1.6,0.7) to (2.3,0.7);
\end{tikzpicture}
\end{equation*}
is a pushout in $\omegaprecat$ or $\omegacat$. 

An \emph{$n$\nbd (pre)polygraph} is a (pre)polygraph whose underlying $\omega$\nbd (pre)category is an $n$\nbd (pre)category. In an $n$\nbd (pre)polygraph, $\gen{X}_m = \emptyset$ for $m > n$.
\end{dfn}

\begin{rmk}
In every (pre)polygraph $(X,\gen{X})$, all cells in $\gen{X}_n$ have rank $n$. The set $\gen{X}_0$ is the entire set $X_0$ of 0-cells.
\end{rmk}

\begin{exm}
Both 1-precategories and 1-categories coincide with small categories; a 1-(pre)polygraph is a category free on a graph.
\end{exm}

\section{Monoidal theories} \label{sec:theories}

\subsection{Planar monoidal theories}

For us, monoidal theories are embodied by \emph{pros}. For technical reasons, we treat these as a special case of a more general structure of \emph{bicoloured pro}, whose relation to pros is the same as the relation of bicategories to monoidal categories.

\begin{dfn}[Bicoloured pro] 
A \emph{bicoloured pro} is a 2\nbd category $T$ together with the structure of a 1\nbd polygraph $(\skel{1}{T}, \gen{T})$ on its 1\nbd skeleton.

A \emph{morphism} $f\colon (T,\gen{T}) \to (S,\gen{S})$ of bicoloured pros is a functor $f\colon T \to S$ of 2\nbd categories with the property that $f(a) \in \eps{}(\gen{S}_0) \cup \gen{S}_1$ for all $a \in \gen{T}_1$. Bicoloured pros and their morphisms form a category $\bipro$.
\end{dfn}

\begin{dfn}[Strict monoidal category]
A \emph{strict monoidal category} is a 2\nbd category with a single 0-cell.
\end{dfn}

\begin{comm}
We are using the characterisation of strict monoidal categories in \cite[Theorem 4.1]{cheng2007periodic} as a definition.
\end{comm}

\begin{dfn}
When some object has a single 0\nbd cell, we denote that 0\nbd cell by $\bullet$.
\end{dfn}

\begin{dfn}[Pro] 
A \emph{pro} is a bicoloured pro with a single 0-cell. We let $\pro$ denote the full subcategory of $\bipro$ on pros.
\end{dfn}

\begin{comm}
That is, a pro is a bicoloured pro whose underlying 2\nbd category is a strict monoidal category.
\end{comm}

\begin{comm}
Equivalently, a bicoloured pro $(T, \gen{T})$ is a 2\nbd category whose 1\nbd cells of type $x \celto y$ are finite paths from $x$ to $y$ in the graph
\begin{equation*}
\begin{tikzpicture}[baseline={([yshift=-.5ex]current bounding box.center)}]
	\node (0) at (0,0) {$\gen{T}_0$};
	\node (1) at (2,0) {$\gen{T}_1$.};
	\draw[1c] (1.west |- 0,.15) to node[auto,swap,arlabel] {$\bord{}{+}$} (0.east |- 0,.15);
	\draw[1c] (1.west |- 0,-.15) to node[auto,arlabel] {$\bord{}{-}$} (0.east |- 0,-.15);
\end{tikzpicture}
\end{equation*}
When $x = y$, the path is allowed to be of length 0, in which case it is interpreted as the unit $\eps{}x$. 

If $T$ has a single 0-cell, these are the same as finite ordered lists of elements of $\gen{T}_1$, that is, elements of the free monoid on $\gen{T}_1$. Seeing a pro as the embodiment of a monoidal theory, we interpret the elements of $\gen{T}_1$ as \emph{sorts}, and 2\nbd cells 
\begin{equation*}
	\varphi\colon (a_1, \ldots, a_n) \celto (b_1, \ldots, b_m)
\end{equation*}
as \emph{operations} taking $n$ inputs of sorts $a_1, \ldots, a_n$ and returning $m$ outputs of sorts $b_1, \ldots, b_m$. 

In particular, if the monoidal theory is one-sorted, then $\gen{T}_1$ is a singleton, the 1\nbd cells of $T$ are in bijection with natural numbers, and the type of a 2\nbd cell is fixed by the \emph{arity} of its input and its output. In that case, we may write $\varphi\colon (n) \celto (m)$ for an operation with $n$ inputs and $m$ outputs.

If we forget the structure of 1\nbd polygraph on a pro, we get a strict monoidal category $T$, which we may see as a special kind of monoidal category. Given a monoidal category $\cat{M}$, a \emph{model} in $\cat{M}$ of the monoidal theory $(T, \gen{T})$ is a strong monoidal functor from $T$ to $\cat{M}$.
\end{comm}

\begin{rmk}
Our definition of a morphism of pros allows a sort to be ``collapsed'' by mapping it onto a unit. We make this choice for technical reasons, even though it seems more common to disallow it, as done in \cite{hackney2015category}. This choice does not have any impact on models.
\end{rmk}

\begin{exm}
The monoid $\mathbb{N}$ of natural numbers with addition, seen as a strict monoidal category with no rank-2 cells, is a one-sorted pro. This corresponds to the ``trivial'' theory of \emph{objects} in a monoidal category.
\end{exm}

\begin{exm}
There is a one-sorted pro $\theory{Mon}$ whose 1\nbd cell $(n)$ is identified with the finite ordinal $\{0 < \ldots < n-1\}$ for each $n \in \mathbb{N}$, and 2\nbd cells $\varphi\colon (n) \celto (m)$ are order-preserving maps. The 0\nbd composite of $\varphi\colon (n) \celto (m)$ and $\psi\colon (p) \celto (q)$ is given by ``concatenation'', that is,
\begin{equation*}
	\varphi\cp{0}\psi\colon (n+p) \celto (m+q), \quad \quad k \mapsto \begin{cases} \varphi(k) & \text{if $k < n$}, \\
	m+\psi(k-n) & \text{if $k \geq n$}.
	\end{cases}
\end{equation*}
This corresponds to the theory of \emph{monoids}.
\end{exm}

\begin{exm}
If $(T, \gen{T})$ is a pro, then $(\coo{T}, \coo{\gen{T}})$, obtained by reversing the orientation of all 2\nbd cells of $T$, is also a pro. For example, $\coo{\theory{Mon}}$ is the theory of \emph{comonoids}.
\end{exm}

\begin{exm}
Let $\theory{Bimod}$ be the strict monoidal category whose 1\nbd cells are injective maps $\imath\colon (k) \incl (n)$ in $\theory{Mon}$, 2\nbd cells $\varphi\colon (\imath\colon (k) \incl (n)) \celto (j\colon (k) \incl (m))$ are commutative triangles
\begin{equation*} 
\begin{tikzpicture}[baseline={([yshift=-.5ex]current bounding box.center)}]
	\node (0) at (-1.5,-1) {$(n)$};
	\node (1) at (0,0) {$(k)$};
	\node (2) at (1.5,-1) {$(m)$,};
	\draw[1cincl] (1) to node[auto,swap,arlabel] {$\imath$} (0);
	\draw[1cinc] (1) to node[auto,arlabel] {$j$} (2);
	\draw[1c] (0) to node[auto,swap,arlabel] {$\varphi$} (2);
\end{tikzpicture}
\end{equation*}
and the 0\nbd composite of a pair of 2\nbd cells 
\begin{align*}
	\varphi\colon & (\imath\colon (k) \incl (n)) \celto (j\colon (k) \incl (m)), \\
	\psi\colon & (\imath'\colon (\ell) \incl (p)) \celto (j'\colon (\ell) \incl (q))
\end{align*}
is the commutative triangle
\begin{equation*} 
\begin{tikzpicture}[baseline={([yshift=-.5ex]current bounding box.center)}]
	\node (0) at (-1.5,-1) {$(n+p)$};
	\node (1) at (0,0) {$(k+\ell)$};
	\node (2) at (1.5,-1) {$(m+q)$.};
	\draw[1cincl] (1) to node[auto,swap,arlabel] {$\imath \cp{0} \imath'$} (0);
	\draw[1cinc] (1) to node[auto,arlabel] {$j \cp{0} j'$} (2);
	\draw[1c] (0) to node[auto,swap,arlabel] {$\varphi \cp{0} \psi$} (2);
\end{tikzpicture}
\end{equation*}
The 1\nbd cells in $\theory{Bimod}$ are freely generated under 0\nbd composition by the pair
\begin{equation*}
	\imath\colon (0) \incl (1), \quad \quad \idd{1}\colon (1) \to (1),
\end{equation*}
where $\imath$ is the unique inclusion of the empty ordinal, so $\theory{Bimod}$ admits the structure of a two-sorted pro. 

There are morphisms of pros $\theory{Mon} \to \theory{Bimod}$ and $\mathbb{N} \to \theory{Bimod}$, sending the generating 1\nbd cell to $\imath$ and $\idd{1}$, respectively. The models of $\theory{Bimod}$ in a monoidal category $\cat{M}$ are given by an object of $\cat{M}$, a monoid in $\cat{M}$, and a two-sided action of the monoid on the object, making the object a \emph{bimodule}.
\end{exm}

\begin{dfn}
There is an obvious functor $\fun{U}\colon \bipro \to \ncat{2}$ forgetting the structure of 1\nbd polygraph. This functor has a right adjoint $\fun{R}\colon \ncat{2} \to \bipro$, described as follows.

Given a 2\nbd category $X$, the 1\nbd skeleton of $\fun{R}X$ is free on the underlying \emph{reflexive} graph of $\skel{1}{X}$: that is, the sorts of $\fun{R}X$ are all the 1\nbd cells of rank 1 in $X$. The 2\nbd cells of type $(a_1, \ldots, a_n) \celto (b_1, \ldots, b_m)$ in $\fun{R}X$ are the 2\nbd cells of type $a_1 \cp{0} \ldots \cp{0} a_n \celto b_1 \cp{0} \ldots \cp{0} b_m$ in $X$. 

Compositions are induced by those of $X$ in the obvious way, and a functor $f\colon X \to Y$ of 2\nbd categories induces a morphism $\fun{R}f\colon \fun{R}X \to \fun{R}Y$ of bicoloured pros by 
\begin{align*}
	\varphi\colon & (a_1, \ldots, a_n) \celto (b_1, \ldots, b_m) \\
	\mapsto f(\varphi)\colon & f(a_1)\cp{0}\ldots\cp{0}f(a_n) \celto f(b_1)\cp{0}\ldots\cp{0}f(b_m).
\end{align*}

For each 2\nbd cell $\varphi$ of type $(a_1, \ldots, a_n) \celto (b_1, \ldots, b_m)$ in a bicoloured pro $(T,\gen{T})$, there is a 2\nbd cell $a_1 \cp{0} \ldots \cp{0} a_n \celto b_1 \cp{0} \ldots \cp{0} b_m$ in $\fun{U}(T,\gen{T})$, which in turn induces a 2\nbd cell of type $(a_1, \ldots, a_n) \celto (b_1, \ldots, b_m)$ in $\fun{RU}(T,\gen{T})$; the unit of the adjunction sends $\varphi$ to this 2\nbd cell. 

Conversely, for each 2\nbd category $X$ and 2\nbd cell $\varphi\colon (a_1, \ldots, a_n) \celto (b_1, \ldots, b_m)$ in $\fun{UR}X$, the counit sends $\varphi$ to the cell of type $a_1 \cp{0} \ldots \cp{0} a_n \celto b_1 \cp{0} \ldots \cp{0} b_m$ in $X$ from which it was induced. It is an exercise to show that the unit and counit satisfy the required equations and determine an adjunction. 
\end{dfn}

\begin{lem} 
\label{lem:bipro_equaliser}
The category $\bipro$ has equalisers.
\end{lem}
\begin{proof}
Let $f, g\colon (T,\gen{T}) \to (S,\gen{S})$ be parallel morphisms of bicoloured pros. Define $T'$ to be the restriction of $T$ to the cells $x$ that satisfy $f(x) = g(x)$ in $S$. Then $T'$ is a 2\nbd category. 

A 1\nbd cell $(a_1, \ldots, a_n)$ in $T$ belongs to $T'$ if and only if $f(a_i) = g(a_i)$ for all $i \in \{1,\ldots,n\}$. It follows that $\gen{T'} \eqdef \{a \in \gen{T} \mid f(a) = g(a)\}$ gives $\skel{1}{T'}$ the structure of a 1\nbd polygraph, so that the inclusion of $T'$ into $T$ is a morphism of bicoloured pros. By a routine argument, it is the equaliser of $f$ and $g$.
\end{proof}

\begin{prop}
\label{prop:bipro_monadic}
The functor $\fun{U}\colon \bipro \to \ncat{2}$ is comonadic.
\end{prop}
\begin{proof}
We have shown that $\fun{U}$ has a right adjoint. Moreover, equalisers as constructed in the proof of Lemma \ref{lem:bipro_equaliser} are evidently created by $\fun{U}$. 

In order to apply the dual of Beck's monadicity theorem \cite[\S VI.7]{maclane1971cats}, it suffices to show that $\fun{U}$ reflects isomorphisms. Let $f\colon (T,\gen{T}) \to (S,\gen{S})$ be a morphism of bicoloured pros and suppose that $\fun{U}f$ is an isomorphism of 2\nbd categories with inverse $g$. Then both $\fun{U}f$ and $g$ must preserve the rank of all cells. 

Let $a \in \gen{S}_1$. Then $g(a)$ can be written uniquely as a finite path $(a_1, \ldots, a_n)$ with $a_i \in \gen{T}_1$ for all $i \in \{1,\ldots,n\}$, and
\begin{equation*}
	a = \fun{U}f(g(a)) = (f(a_1),\ldots,f(a_n))
\end{equation*}
where $n > 0$ and $f(a_i) \in \gen{S}_1$ for all $i \in \{1,\ldots,n\}$. Because $\skel{1}{S}$ is free, this is only possible if $n = 1$ and $f(a_1) = a$. It follows that $g$ sends generators to generators, hence it determines a morphism of bicoloured pros, inverse to $f$ in $\bipro$.
\end{proof}

\begin{cor} \label{cor:bipro_colimits}
The categories $\bipro$ and $\pro$ have all small limits and colimits.
\end{cor}
\begin{proof}
By the dual of \cite[Exercise 2, \S VI.2]{maclane1971cats}, a comonadic functor creates all colimits in its codomain; since $\ncat{2}$ has all small colimits, so does $\bipro$. Moreover, $\ncat{2}$ has all small limits and $\bipro$ has equalisers, so $\bipro$ has all small limits by the dual of \cite[Corollary 2]{linton1969coequalizers}.  

Since $\pro$ is defined equationally, it is a reflective subcategory of $\bipro$. It follows from \cite[Proposition 4.5.15]{riehl2017category} that it also has all small limits and colimits.
\end{proof}


\subsection{Non-planar monoidal theories}

\begin{dfn}[Braided strict monoidal category]
A \emph{braided strict monoidal category} is a strict monoidal category $X$ together with a family of 2\nbd cells
\begin{equation*}
	\sigma_{x,y}\colon x \cp{0} y \celto y \cp{0} x
\end{equation*}
called \emph{braidings}, indexed by 1\nbd cells $x, y$, satisfying the following axioms:
\begin{enumerate}
	\item the braidings are invertible, that is, there are unique 2\nbd cells $\invrs{\sigma_{x,y}}$, called \emph{inverse braidings}, such that $\sigma_{x,y} \cp{1} \invrs{\sigma_{x,y}}$ and $\invrs{\sigma_{x,y}} \cp{1} \sigma_{x,y}$ are units;
	\item they are natural in their parameters, that is, for all 2\nbd cells $\varphi\colon x \celto x'$ and $\psi\colon y \celto y'$, 
	\begin{align*}
		(\varphi \cp{0} y) \cp{1} \sigma_{x',y} & = \sigma_{x,y} \cp{1} (y \cp{0} \varphi), \\
		(x \cp{0} \psi) \cp{1} \sigma_{x,y'} & = \sigma_{x,y} \cp{1} (\psi \cp{0} x);
	\end{align*}
	\item \label{ax:braided3} they are compatible with 0\nbd composition and units, that is, 
	\begin{align*}
		\sigma_{x \cp{0} x', y} & = (x \cp{0} \sigma_{x',y}) \cp{1} (\sigma_{x,y} \cp{0} x'), \\
		\sigma_{x, y \cp{0} y'} & = (\sigma_{x,y} \cp{0} y') \cp{1} (y \cp{0} \sigma_{x,y'}), \\
		\sigma_{\eps{}\bullet,y} & = \eps{}y, \quad \quad \sigma_{x,\eps{}\bullet} = \eps{}x,
	\end{align*}
	whenever the left-hand side is defined.
\end{enumerate}
A \emph{functor} $f\colon X \to Y$ of braided strict monoidal categories is a functor of the underlying 2\nbd categories that preserves braidings, that is, $f(\sigma_{x,y}) = \sigma_{f(x),f(y)}$ for all 1\nbd cells $x, y$ in $X$. With their functors, braided strict monoidal categories form a category $\brmoncat$.
\end{dfn}

\begin{dfn}[Prob]
A \emph{prob} is a pro together with a structure of braided strict monoidal category on its underlying strict monoidal category. 

A \emph{morphism} of probs is a morphism of pros that preserves the braidings. Probs and their morphisms form a category $\prob$. 
\end{dfn}

\begin{comm}
Models of probs live in \emph{braided} monoidal categories $\cat{M}$, not necessarily strict. A model of a prob $(T, \gen{T})$ in $\cat{M}$ is a \emph{braided} strong monoidal functor from $T$ to $\cat{M}$. 
\end{comm}

\begin{rmk} \label{rmk:braiding_only_generator}
To determine a unique structure of braided strict monoidal category on a pro it is, in fact, sufficient to give braidings $\sigma_{a,b}$ for all pairs of generating 1\nbd cells $a,b$; a morphism of pros that preserves these braidings automatically preserves all braidings. This is a consequence of axiom \ref{ax:braided3} of braided strict monoidal categories, since every 1\nbd cell in a pro can be decomposed as a composite of generating 1\nbd cells.
\end{rmk}

\begin{dfn}[Dual braided structure] 
\label{dfn:dualbraided}
Let $X$ be a braided strict monoidal category with braidings $\{\sigma_{x,y}\}$. The family of 2\nbd cells
\begin{equation*}
	\sigma^*_{x,y} \eqdef \invrs{\sigma_{y,x}}
\end{equation*}
defines a second structure $X^*$ of braided strict monoidal category on the underlying strict monoidal category of $X$. 

If $f\colon X \to Y$ is a functor of braided strict monoidal categories, the same underlying functor of 2\nbd categories determines a functor $f^*\colon X^* \to Y^*$. This defines an involution $-^*$ on $\brmoncat$, which also induces a duality on $\prob$.
\end{dfn}

\begin{dfn}[Symmetric strict monoidal category] 
A \emph{symmetric strict monoidal category} is a braided strict monoidal category $X$ satisfying $X = X^*$.
\end{dfn}

\begin{dfn}[Prop]
A \emph{prop} is a prob whose underlying braided strict monoidal category is symmetric. We let $\propp$ denote the full subcategory of $\prob$ on props. 
\end{dfn}

\begin{dfn} \label{dfn:free_prob_prop}
There is an obvious forgetful functor $\fun{U}\colon \prob \to \pro$ and an inclusion of subcategories $\propp \incl \prob$. Both of these have left adjoints:
\begin{itemize}
	\item the left adjoint $\fun{F}\colon \pro \to \prob$ of $\fun{U}$ freely adds braidings $\sigma_{x,y}$ and inverse braidings $\invrs{\sigma_{x,y}}$ for all pairs of 1\nbd cells $x, y$ of a pro (or just the generating ones, see Remark \ref{rmk:braiding_only_generator}), then quotients by the axioms of braided strict monoidal categories;
	\item the reflector $\fun{r}\colon \prob \to \propp$ quotients by the equation $\sigma_{x,y} = \sigma^*_{x,y}$ for all pairs of 1\nbd cells $x, y$ of a prob.
\end{itemize}
Since we have not yet shown that $\prob$ has coequalisers, for the moment we can interpret the latter as a coequaliser in $\pro$, then observe that the images of the $\sigma_{x,y}$ still form a family of braidings in the quotient.
\end{dfn}

\begin{exm}
The free prob $\mathbb{B} \eqdef \fun{F}\mathbb{N}$ is the \emph{theory of braids}. With 1\nbd composition, 2\nbd cells of type $(n) \celto (n)$ in $\mathbb{B}$ form the braid group $B_n$ on $n$ strands.
\end{exm}

\begin{exm}
The prop reflection $\mathbb{S} \eqdef \fun{r}\mathbb{B}$ of the theory of braids is the \emph{theory of permutations}. With 1\nbd composition, 2\nbd cells of type $(n) \celto (n)$ in $\mathbb{S}$ form the symmetric group $S_n$ on $n$ elements.
\end{exm}

\begin{exm}
Let $\theory{CMon}$ be defined as $\theory{Mon}$, but 2\nbd cells of type $(n) \celto (m)$ are \emph{all} functions from $(n)$ to $(m)$, not just the order-preserving ones. This is a one-sorted prop with braidings generated by
\begin{equation*}
	\sigma_{1,1}\colon (2) \celto (2), \quad \quad 0 \mapsto 1, \; 1 \mapsto 0.
\end{equation*}
It corresponds to the theory of \emph{commutative monoids} in symmetric monoidal categories. Similarly, models of $\coo{\theory{CMon}}$ are \emph{commutative comonoids}.
\end{exm}

\begin{exm}
There is a one-sorted pro $\theory{Mat}_\mathbb{Z}$ whose 2\nbd cells $A\colon (n) \celto (m)$ are $(m \times n)$\nbd matrices of integers, the 1\nbd composite $A \cp{1} B$ is the product $BA$ of matrices, and  
\begin{equation*}
	A \cp{0} B \eqdef \begin{pmatrix} A & 0 \\ 0 & B \end{pmatrix}.
\end{equation*}
This is a prop with braidings generated by 
\begin{equation*}
	\sigma_{1,1} \eqdef \begin{pmatrix} 0 & 1 \\ 1 & 0 \end{pmatrix}.
\end{equation*}
There is a morphism of props $\theory{CMon} \to \theory{Mat}_\mathbb{Z}$ sending $\varphi\colon (n) \celto (m)$ to the $(m \times n)$\nbd matrix $A_\varphi$ with entries
\begin{equation*}
	A_\varphi(j,i) \eqdef \begin{cases} 1 & \text{if $j = \varphi(i)$}, \\ 0 & \text{otherwise}, \end{cases}
\end{equation*}
and a morphism $\coo{\theory{CMon}} \to \theory{Mat}_\mathbb{Z}$ sending $\coo{\varphi}\colon (m) \celto (n)$ to the transpose of $A_\varphi$. As a symmetric monoidal theory, $\theory{Mat}_\mathbb{Z}$ corresponds to the theory of \emph{commutative and cocommutative Hopf algebras} \cite[Section 7]{bonchi2017interacting}. 

\end{exm}

\begin{dfn}[Gray-category]
A \emph{Gray-category} is a 3\nbd precategory $G$ together with a family of 3\nbd cells
\begin{equation*}
	\chi_{x,y}\colon (x \cp{0} \bord{}{-} y) \cp{1} (\bord{}{+}x \cp{0} y) \celto (\bord{}{-}x \cp{0} y) \cp{1} (x \cp{0} \bord{}{+} y)
\end{equation*}
called \emph{interchangers}, indexed by 2\nbd cells $x, y$ with $\bord{0}{+}x = \bord{0}{-}y$, satisfying the following axioms:
\begin{enumerate}
	\item the interchangers are invertible, that is, there are unique 3\nbd cells 
\begin{equation*}
	\invrs{\chi_{x,y}}\colon (\bord{}{-}x \cp{0} y) \cp{1} (x \cp{0} \bord{}{+} y) \celto (x \cp{0} \bord{}{-} y) \cp{1} (\bord{}{+}x \cp{0} y)
\end{equation*}
	called \emph{inverse interchangers}, such that $\chi_{x,y} \cp{2} \invrs{\chi_{x,y}}$ and $\invrs{\chi_{x,y}} \cp{2} \chi_{x,y}$ are units;
	\item the interchangers are natural in their parameters, that is, for all 3\nbd cells $\varphi\colon x \celto x'$ and $\psi\colon y \celto y'$ with $\bord{0}{+}\varphi = \bord{0}{-}\psi$,
	\begin{align*}
		((\varphi \cp{0} \bord{}{-} y) \cp{1} (\bord{1}{+}\varphi \cp{0} y)) \cp{2} \chi_{x',y} & =\chi_{x,y} \cp{2} ((\bord{1}{-}\varphi \cp{0} y) \cp{1} (\varphi \cp{0} \bord{}{+} y)), \\
		((x \cp{0} \bord{1}{-} \psi) \cp{1} (\bord{}{+}x \cp{0} \psi)) \cp{2} \chi_{x,y'} & = \chi_{x,y} \cp{2} ((\bord{}{-}x \cp{0} \psi) \cp{1} (x \cp{0} \bord{1}{+} \psi));
	\end{align*}
	\item the interchangers are compatible with 1\nbd compositions and units, that is, 
	\begin{align*}
		\chi_{x\cp{1}x',y} & = ((x \cp{0} \bord{}{-}y) \cp{1} \chi_{x',y}) \cp{2} (\chi_{x,y} \cp{1} (x' \cp{0} \bord{}{+}y)), \\
		\chi_{x,y\cp{1}y'} & = (\chi_{x,y} \cp{1} (\bord{}{+}x \cp{0} y')) \cp{2} ((\bord{}{-}x \cp{0} y) \cp{1} \chi_{x,y'})), \\
		\chi_{\eps{}x, y} & = \eps{}(\eps{}x \cp{0} y), \quad \quad \chi_{x, \eps{}y} = \eps{}(x \cp{0} \eps{}y),
	\end{align*}
	whenever the left-hand side is defined;
	\item \label{ax:3interchange} for all pairs of 3\nbd cells $\varphi, \psi$ with $\bord{1}{+}\varphi = \bord{1}{-}\psi$, the equation
	\begin{equation*} 
		(\varphi \cp{1} \bord{2}{-}\psi) \cp{2} (\bord{2}{+}\varphi \cp{1} \psi) = (\bord{2}{-}\varphi \cp{1} \psi) \cp{2} (\varphi \cp{1} \bord{2}{+} \psi)
	\end{equation*}
	holds in $G$.
\end{enumerate}
A \emph{functor} $f\colon G \to H$ of Gray\nbd categories is a functor of the underlying 3\nbd precategories that preserves the interchangers, that is, $f(\chi_{x,y}) = \chi_{f(x),f(y)}$ for all suitable 2\nbd cells $x, y$ in $G$. With their functors, Gray\nbd categories form a category $\graycat$.
\end{dfn}

\begin{rmk}
Axiom \ref{ax:3interchange} is an instance of (\ref{eq:preinterchange}), allowing us to univocally define the 1\nbd composition $\varphi \cp{1} \psi$ of 3\nbd cells with $\bord{1}{+}\varphi = \bord{1}{-}\psi$ in a Gray\nbd category. 
\end{rmk}

\begin{comm}
A more concise definition is that a Gray-category is a small category enriched over $\ncat{2}$ with the ``pseudo'' Gray tensor product \cite[Chapter 5]{gordon1995coherence}. As in \cite[\S 1.4]{lack2011quillen}, one derives that $\graycat$ is locally finitely presentable, and in particular has all small limits and colimits.
\end{comm}

\begin{dfn}
By Proposition \ref{prop:omega_preomega}, every 3\nbd category seen as a 3\nbd precategory admits a natural structure of Gray\nbd category with units as interchangers. This defines an embedding $\ncat{3} \incl \graycat$, which makes $\ncat{3}$ a reflective subcategory of $\graycat$: the reflector universally turns the interchangers into units.
\end{dfn}

\begin{dfn}
Given a braided strict monoidal category $X$, we define a Gray-category $\fun{B}X$ as follows. For all $n \in \mathbb{N}$, we let $\fun{B}X_{n+1} \eqdef X_n$, with the same boundary and unit operators as $X$ between $\fun{B}X_{n+2}$ and $\fun{B}X_{n+1}$. We let $\fun{B}X_0 \eqdef \{\bullet\}$, with the only possible unit and boundary operators relating it to $\fun{B}X_1$. This defines the underlying reflexive $\omega$\nbd graph of $\fun{B}X$. 

To make $\fun{B}X$ a 3\nbd precategory, it suffices to define the principal compositions. Because $\fun{B}X$ has no rank-1 cells, the principal compositions are of the form $x \cp{k} y$ where
\begin{itemize}
	\item $k = 1$ and $\min \{\rank{x},\rank{y}\} = 2$, or 
	\item $k = 2$ and $\rank{x} = \rank{y} = 3$.
\end{itemize}
In either case, $x \cp{k-1} y$ is defined in $X$, and we let $x \cp{k} y$ be equal to it in $\fun{B}X$.

Finally, given 2\nbd cells $x, y$ in $\fun{B}X$, we let the interchanger $\chi_{x,y}$ correspond to the braiding $\sigma_{x,y}$ in $X$. It is an exercise to check that this gives $\fun{B}X$ the structure of a Gray-category.

This assignment extends to a functor $\fun{B}\colon \brmoncat \to \graycat$ in the obvious way. By \cite[Theorem 2.16]{cheng2011periodic}, this functor is full and faithful, and its essential image consists exactly of those Gray-categories that have a single 0\nbd cell and a single 1\nbd cell. This can be seen as an alternative characterisation of $\brmoncat$ as a full subcategory of $\graycat$.
\end{dfn}

\begin{rmk}
Through $\fun{B}$, the duality $-^*$ on $\brmoncat$ is the restriction of the duality on $\graycat$ that reverses the orientation of 1\nbd cells. 
\end{rmk}

\begin{dfn} \label{dfn:alternative_prob}
If $X$ is a braided strict monoidal category, the structure of a 1\nbd polygraph on $\skel{1}{X}$ determines a unique structure of 2\nbd prepolygraph on $\skel{2}{\fun{B}X}$, and vice versa. A functor $f$ of braided strict monoidal categories sends generators to generators if and only if $\fun{B}f$ does. 

Thus, a prob is equivalently defined as a Gray\nbd category $T$ with a single 0\nbd cell, a single 1\nbd cell, and the structure of a 2\nbd prepolygraph $(\skel{2}{T}, \gen{T})$ on its 2\nbd skeleton. A morphism $f\colon (T,\gen{T}) \to (S,\gen{S})$ of probs is a functor of Gray\nbd categories such that $f(a) \in \{\eps{2}\bullet\} \cup \gen{S}_2$ for all $a \in \gen{T}_2$. 

We conclude that there is a triangle of functors
\begin{equation*} 
\begin{tikzpicture}[baseline={([yshift=-.5ex]current bounding box.center)}]
	\node (0) at (-2,-1.5) {$\brmoncat$};
	\node (1) at (0,0) {$\prob$};
	\node (2) at (2,-1.5) {$\graycat$};
	\draw[1c] (1) to node[auto,swap,arlabel] {$\fun{U}_2$} (0);
	\draw[1c] (1) to node[auto,arlabel] {$\fun{U}_3$} (2);
	\draw[1cinc] (0) to node[auto,swap,arlabel] {$\fun{B}$} (2);
\end{tikzpicture}
\end{equation*}
commuting up to natural isomorphism, where $\fun{U}_2$ and $\fun{U}_3$ are the forgetful functors associated to the two alternative definitions of prob.
\end{dfn}

\begin{prop}
The categories $\brmoncat$, $\prob$, and $\propp$ have all small limits and colimits.
\end{prop}
\begin{proof}[Sketch of the proof] 
First of all, the essential image of $\fun{B}$ is defined equationally in $\graycat$, so $\brmoncat$ is, up to equivalence, a reflective subcategory. Since $\graycat$ has all small limits and colimits, by \cite[Proposition 4.5.15]{riehl2017category} so does $\brmoncat$. 

Then, we can mimic the proofs of Lemma \ref{lem:bipro_equaliser} and Proposition \ref{prop:bipro_monadic} to show that $\prob$ has equalisers and that the functor $\fun{U}_2\colon \prob \to \brmoncat$ is comonadic. As in the proof of Corollary \ref{cor:bipro_colimits}, we deduce that $\prob$ has all small limits and colimits, and so does its reflective subcategory $\propp$.
\end{proof}

\begin{rmk} \label{rmk:pseudomonic}
The functors $\fun{U}\colon \bipro \to \ncat{2}$, $\fun{U}_2\colon \prob \to \brmoncat$ are \emph{pseudomonic}: that is, in addition to being faithful, they reflect and are full on isomorphisms. This captures the fact that a 2\nbd category admits at most one structure of bicoloured pro, a consequence of the general statement, proved by Michael Makkai \cite[Section 4, Proposition 8]{makkai2005word}, that an $\omega$\nbd category admits at most one structure of polygraph.

Because the composite of a pseudomonic with a full and faithful functor is pseudomonic, it follows that $\fun{U}_3\colon \prob \to \graycat$ is also pseudomonic.
\end{rmk}

\section{Combinatorial results} \label{sec:combinatorics}

\subsection{In generic dimension}

In this section, we use results of \cite[Section 6]{hadzihasanovic2018combinatorial}. These results are stated relative to the restricted class of constructible directed complexes, but the proofs do not involve any properties that are not satisfied by all regular directed complexes. Thus all cited statements hold with \emph{constructible} replaced by \emph{regular}.

\begin{dfn}
Let $U$ be a closed subset of a regular directed complex. For each $n \geq -1$, the bipartite directed graph $\maxd{n}{U}$ has
\begin{equation*}
	\{x \in U \,|\, \dmn{x} \leq n\} + \{x \in U \,|\, x \text{ is maximal and } \dmn{x} > n\}
\end{equation*}
as set of vertices, and an edge $y \to x$ if and only if 
\begin{itemize}
	\item $\dmn{y} \leq n$, $\dmn{x} > n$, and $y \in \bord{n}{-}x \setminus \bord{n-1}{}x$, or
	\item $\dmn{y} > n$, $\dmn{x} \leq n$, and $x \in \bord{n}{+}y \setminus \bord{n-1}{}y$.
\end{itemize}
\end{dfn}

\begin{dfn}[Frame dimension]
Let $U$ be a closed subset of a regular directed complex. The \emph{frame dimension} of $U$ is the integer 
\begin{equation*}
	\frdmn{U} \eqdef \max \{\dmn{\clos\{x\} \cap \clos\{y\}} \mid \text{$x, y$ maximal in $U$, $x \neq y$}\}.
\end{equation*}
\end{dfn}

\begin{rmk}
If $\frdmn{U} = -1$, then $U$ is a disjoint union of atoms.
\end{rmk}

\begin{dfn}[Frame acyclicity]
A regular directed complex $P$ is \emph{frame-acyclic} if, for all molecules $U$ in $P$, if $\frdmn{U} = k$, then $\maxd{k}{U}$ is acyclic.
\end{dfn}

\begin{lem} \label{lem:split_lemma}
Let $P$ be a frame-acyclic regular directed complex. Then
\begin{enumerate}
	\item for all molecules $U$ in $P$, if $U = U_1 \cup U_2$ for some closed subsets $U_1, U_2$ such that $U_1 \cap U_2 = \bord{k}{+}U_1 = \bord{k}{-}U_2$, then $U_1$ and $U_2$ are molecules;
	\item $(\mol{}{P}^*, \{\clos\{x\}\}_{x \in P})$ is a polygraph. 
\end{enumerate}
\end{lem}
\begin{proof}
A corollary of \cite[Proposition 26]{hadzihasanovic2018combinatorial}.
\end{proof}

\begin{dfn}[$k$\nbd Order]
Let $U$ be a regular $n$\nbd molecule. For $k < n$, a \emph{$k$\nbd order} on $U$ is a linear ordering $(x_1, \ldots, x_m)$ of the set 
\begin{equation*}
	\{x \in U \,|\, x \text{ is maximal and } \dmn{x} > k\}
\end{equation*}
with the property that, if there is a path from $x_i$ to $x_j$ in $\maxd{k}{U}$, then $i \leq j$. 
\end{dfn}

\begin{prop} \label{prop:admit_k-order}
Let $U$ be a regular $n$\nbd molecule, $k < n$. If $\maxd{k}{U}$ is acyclic, then $U$ admits a $k$\nbd order.
\end{prop}
\begin{proof}
Every directed acyclic graph admits a topological sorting, that is, a linear order $\preceq$ on its vertices with the property that if there is a path from $x$ to $y$, then $x \preceq y$. The restriction of a topological sorting of $\maxd{k}{U}$ to the maximal elements of dimension greater than $k$ is a $k$\nbd order on $U$.
\end{proof}

\begin{prop} \label{prop:codim1_acyclic}
Let $U$ be a regular $n$\nbd molecule. Then $\maxd{n-1}{U}$ is acyclic and $U$ admits an $(n-1)$\nbd order.
\end{prop}
\begin{proof}
Follows from Proposition \ref{prop:admit_k-order} and \cite[Proposition 20]{hadzihasanovic2018combinatorial}.
\end{proof}

\begin{lem} \label{lem:frame_decomposition}
Let $U$ be a frame-acyclic regular molecule, $k \geq \frdmn{U}$, and let $(x_1, \ldots, x_m)$ be a $k$\nbd order on $U$. There exist molecules $V_1, \ldots, V_m$, such that
\begin{equation*}
	U = V_1 \cp{k} \ldots \cp{k} V_m
\end{equation*}
and $x_i \in V_j$ if and only if $i = j$ for all $i,j \in \{1, \ldots m\}$.
\end{lem}
\begin{proof}
If $m = 1$ then $V_1 \eqdef U$ satisfies the statement. Suppose that $m > 1$ and that $k = \frdmn{U}$. Then we proceed as in the proof of \cite[Proposition 26]{hadzihasanovic2018combinatorial} to produce $i \in \{1,\ldots,m-1\}$ and a decomposition $U = U_1 \cup U_2$ such that
\begin{enumerate}
	\item $U_1$ contains $x_1, \ldots, x_i$ and $U_2$ contains $x_{i+1},\ldots, x_m$,
	\item $U_1 \cap U_2 = \bord{k}{+}U_1 = \bord{k}{-}U_2$.
\end{enumerate}
By Lemma \ref{lem:split_lemma}, both $U_1$ and $U_2$ are molecules. Moreover $(x_1, \ldots, x_i)$ and $(x_{i+1},\ldots, x_m)$ are $k$\nbd orders on $U_1$ and $U_2$, respectively. We conclude by the inductive hypothesis applied to $U_1$ and $U_2$.

Finally, suppose that $k > \ell \eqdef \frdmn{U}$. By frame acyclicity, we can fix an $\ell$\nbd order $(y_1, \ldots, y_p)$ on $U$, and by the first part of the proof we can decompose $U$ as
\begin{equation*}
	W_1 \cp{\ell} \ldots \cp{\ell} W_p
\end{equation*}
with $y_i \in W_j$ if and only if $i = j$. Now for each $i \in \{1,\ldots,m\}$ there is a unique $j(i) \in \{1,\ldots,p\}$ such that $x_i = y_{j(i)}$. Let
\begin{align*}
	V_i & \eqdef \bord{k}{\alpha(i,1)}W_1 \cp{\ell} \ldots \cp{\ell} W_{j(i)} \cp{\ell} \ldots \cp{\ell} \bord{k}{\alpha(i,p)}W_p, \text{ where} \\
	\alpha(i,j) & \eqdef \begin{cases} + & \text{if $j = j(i')$ for some $i' < i$,} \\ 
	- & \text{otherwise}.
	\end{cases}
\end{align*}
Then $U = V_1 \cp{k} \ldots \cp{k} V_m$ is the required decomposition.
\end{proof}

\begin{dfn}[Substitution]
Let $V$ and $W$ be regular $n$\nbd molecules with spherical boundary, let $U$ be a regular $n$\nbd molecule, and suppose $V \submol U$. Then $U \setminus (V \setminus \bord V)$ is a closed subset of $U$.

Suppose that $\bord{}{\alpha}V$ is isomorphic to $\bord{}{\alpha}W$ for all $\alpha \in \{+,-\}$. From \cite[Lemma 2.2]{hadzihasanovic2020diagrammatic} we obtain a unique isomorphism $\imath\colon \bord U \incliso \bord V$. We define $U[W/V]$ to be the pushout
\begin{equation*}
\begin{tikzpicture}[baseline={([yshift=-.5ex]current bounding box.center)}]
	\node (0) at (0,1.5) {$\bord V$};
	\node (1) at (2.5,0) {$U[W/V]$};
	\node (2) at (0,0) {$W$};
	\node (3) at (2.5,1.5) {$U \setminus (V \setminus \bord V)$};
	\draw[1cinc] (0) to (3);
	\draw[1cincl] (0) to (2);
	\draw[1cinc] (2) to (1);
	\draw[1cincl] (3) to (1);
	\draw[edge] (1.6,0.2) to (1.6,0.7) to (2.3,0.7);
\end{tikzpicture}
\end{equation*}
in $\rdcpx$, and call it the \emph{substitution} of $W$ for $V \submol U$. By [Proposition 2.4, \emph{ibid.}] this is an $n$\nbd molecule with boundaries isomorphic to those of $U$, and such that $W \submol U[W/V]$.
\end{dfn}

\begin{rmk}
As shown in \cite[Lemma 2.5]{hadzihasanovic2020diagrammatic}, if $U$ is an $n$\nbd molecule with a decomposition $U = V_1 \cp{n-1} \ldots \cp{n-1} V_m$ as in Lemma \ref{lem:frame_decomposition}, then $\bord{}{\alpha}V_i$ is isomorphic to $\bord{}{-\alpha}V_i[\bord{}{\alpha}x_i/\bord{}{-\alpha}x_i]$ for all $\alpha \in \{+,-\}$ and $i \in \{1,\ldots,m\}$. 
\end{rmk}


\subsection{In low dimension}

\begin{dfn}[Totally loop-free molecule] \label{dfn:hasseo}
Given a regular molecule $U$, let $\hasseo{U}$ be the directed graph obtained from $\hasse{U}$ by reversing all the edges labelled $-$. We say that $U$ is \emph{totally loop-free} if $\hasseo{U}$ is acyclic as a directed graph.

If $U$ is totally loop-free, for all $x, y \in U$, we let $x \preceq y$ if and only if there is a path from $x$ to $y$ in $U$. 
\end{dfn}

\begin{prop} \label{prop:dim2_loopfree}
Let $U$ be a regular molecule. If $\dmn{U} \leq 2$, then $U$ is totally loop-free and $\preceq$ is a linear order on $U$.
\end{prop}
\begin{proof}
If $\dmn{U} \in \{0,1\}$ or if $U$ is a 2\nbd dimensional atom, this is easy. Otherwise, decompose $U$ as $V_1 \cp{1} \ldots \cp{1} V_m$ so that each $V_i$ contains a unique 2\nbd dimensional element $x_i$, and let $V_0 \eqdef \bord{}{-}U$ and $U' \eqdef V_0 \cp{1} \ldots \cp{1} V_{m-1}$. Then $U = U' \cup \clos\{x_m\}$, and we may assume inductively that the statement holds for $U'$. 

Suppose that there is a cycle in $\hasseo{U}$. Because both $U'$ and $\clos\{x_m\}$ are totally loop-free, such a cycle must leave $U'$, enter $\clos\{x_m\} \setminus \bord{}{-}x_m$, then return to $U'$. Such a path either
\begin{itemize}
	\item enters $x_m$ via $\sbord{}{-}x_m$, then enters $\bord{}{+}x_m$ and leaves via $\bord{0}{+}x_m$, or
	\item enters $\sbord{}{+}x_m$ directly via $\bord{0}{-}x_m$, stays in $\bord{}{+}x_m$ and leaves via $\bord{0}{+}x_m$. 
\end{itemize}
In both cases, the path through $\clos\{x_m\} \setminus \bord{}{-}x_m$ can be replaced with the unique path to $\bord{0}{+}x_m$ that stays in $\bord{}{-}x_m \subseteq U'$. In this way, we create a cycle in $\hasseo{U'}$, contradicting the inductive hypothesis. 

This proves that $U$ is totally loop-free. To show that $\preceq$ is a linear order, it suffices to compare elements of $U'$ and of $\clos\{x_m\} \setminus \bord{}{-}x_m$. Let $x \in U'$. There are two possible cases:
\begin{itemize}
	\item $\bord{0}{+}x_m \preceq x$ in $U'$. Then $z \preceq x$ for all elements $z \in \clos\{x_m\}$. 
	\item $x \prec \bord{0}{+}x_m$ in $U'$. Let $y$ be the unique 1\nbd dimensional element of $\sbord{}{-}x_m$ that covers $\bord{0}{+}x_m$. Suppose that $y \prec x$ in $U'$, that is, there is a non-trivial path from $y$ to $x$ in $\hasseo{U'}$. Such a path cannot pass through $\bord{}{+}y = \bord{0}{+}x_m$, for otherwise $\bord{0}{+}x_m \preceq x$; nor it can enter a 2\nbd dimensional element, because $y$ is not covered by any element of $U'$ with orientation $-$. Therefore $x \preceq y$, so $x \prec x_m$ and $x \prec z$ for all elements $z \in \bord{}{+}x_m \setminus \bord{0}{}x_m$. 
\end{itemize}
This proves that $\preceq$ is a linear order on $U$.
\end{proof}

\begin{rmk} \label{rmk:2molecule_loopfree}
If $U$ is a regular molecule with $\dmn{U} \leq 2$, by \cite[Theorem 2.17]{steiner1993algebra} combined with Proposition \ref{prop:dim2_loopfree}, $\mol{}{U}^*$ is equal to $\mol{}{U}$.
\end{rmk}

\begin{prop}
Let $U$ be a regular 2\nbd molecule, $k \in \{0,1\}$, and let $x, y \in U$ be maximal elements of dimension $> k$. If there is a path from $x$ to $y$ in $\maxd{k}{U}$, then $x \preceq y$.
\end{prop}
\begin{proof}
Suppose $k = 1$. A path $x = x_0 \to w_0 \to \ldots \to w_{m-1} \to x_m = y$ in $\maxd{1}{U}$ is a concatenation of two-step paths $x_i \to w_i \to x_{i+1}$ where $\dmn{x_i} = 2$ for all $i \in \{0,\ldots,m\}$ and 
\begin{equation*}
	w_i \in (\bord{}{+}x_i \setminus \bord{0}x_i) \cap (\bord{}{-}x_{i+1} \setminus \bord{0}x_{i+1}).
\end{equation*}
If $\dmn{w_i} = 1$ then $w_i \in \sbord{}{+}x_i \cap \sbord{}{-}x_{i+1}$, so $x_i \prec w_i \prec x_{i+1}$. Suppose $\dmn{w_i} = 0$. Because $\bord{}{+}x_i$ is pure and 1\nbd dimensional, $w_i$ is covered by some element of $\sbord{}{+}x_i$, and because $w_i \notin \bord{0}{}x_i = \sbord (\bord{}{+}x_i)$, by \cite[Lemma 1.16]{hadzihasanovic2020diagrammatic} it is in fact covered by two elements of $\sbord{}{+}x_i$ with opposite orientations. If $w_i \in \sbord{}{+}x_i$ covers $w_i$ with orientation $+$, we have $x_i \prec w'_i \prec w_i$. 

Dually, we find $w''_i \in \sbord{}{-}x_{i+1}$ that covers $w_i$ with orientation $-$, so that $w_i \prec w''_i \prec x_{i+1}$. It follows that $x_i \prec x_{i+1}$ for all $i \in \{0,\ldots,m-1\}$, and we conclude that $x \preceq y$. 

Now suppose that $k = 0$. A path from $x$ to $y$ in $\maxd{0}{U}$ is a concatenation of two-step paths $x_i \to w_i \to x_{i+1}$ where $\dmn{x_i} \in \{1,2\}$ and $w_i$ is the only element of $\bord{0}{+}x_i = \bord{0}{-}x_{i+1}$. If $\dmn{x_i} = 1$ then immediately $x_i \prec w_i$, otherwise there is exactly one element $w'_i \in \sbord{}{+}x_i$ such that $\bord{0}{+}x_i = \bord{}{+}w'_i$, so $x_i \prec w'_i \prec w_i$. Similarly we find that $w_i \prec x_{i+1}$. 
\end{proof}

\begin{cor} \label{cor:normal_order}
If $U$ is a regular 2\nbd molecule, the restriction of $\preceq$ to 2\nbd dimensional elements determines a 1\nbd order on $U$.
\end{cor}

\begin{dfn}[Normal 1\nbd order]
Let $U$ be a regular 2\nbd molecule. The \emph{normal 1\nbd order} on $U$ is the 1\nbd order determined by Corollary \ref{cor:normal_order}.
\end{dfn}

\begin{exm}
In the shape of the 2\nbd diagrams
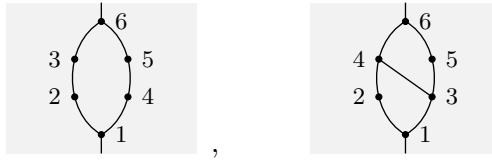
\begin{equation*}
	\input{img/normalorder}
\end{equation*}
the normal 1\nbd order is indicated by the labels of 2\nbd cells. 

In general, a rule-of-thumb for reconstructing the normal 1\nbd order from a string diagram is:
\begin{enumerate}
	\item if there is an upward path between two 2\nbd cells, then the \emph{lowermost} precedes the \emph{uppermost};
	\item if there is no such path, then the \emph{leftmost} precedes the \emph{rightmost}.
\end{enumerate}
Due to a certain flexibility in the depiction of string diagrams, the second rule may not always strictly hold (but it will hold up to a harmless deformation of the picture).
\end{exm}

\begin{cor} \label{cor:dim2_acyclic}
Let $P$ be a regular directed complex with $\dmn{P} \leq 2$. Then $P$ is frame-acyclic.
\end{cor}

\begin{lem} \label{lem:slice_decomposition}
Let $U$ be a regular molecule with $\dmn{U} \leq 2$ and let $I \subseteq U$ be a 1\nbd molecule with $\bord{}{-}I = \bord{0}{-}U$ and $\bord{}{+}I = \bord{0}{+}U$. Then
\begin{enumerate}[label=(\alph*)]
	\item there is a unique decomposition $U = U_+ \cp{1} U_-$ with $\bord{1}{+}U_+ = \bord{1}{-}U_- = I$;
	\item for all $\alpha \in \{+,-\}$, if $V \submol U$ is a 2\nbd molecule with spherical boundary and $V \cap I = \bord{}{\alpha}V$, then $V \submol U_{\alpha}$.
\end{enumerate}
\end{lem}
\begin{proof}
By induction on the number $m$ of 2\nbd dimensional elements of $U$: if $m = 0$, then necessarily $U = I$ and $U = I \cp{1} I$ is the unique decomposition. If $m > 0$, we can write $U = U' \cp{1} U_x$ where $U_x$ contains a single 2\nbd dimensional element $x$. Now either 
\begin{itemize}
	\item $I \subseteq U'$, in which case we have a unique decomposition $U' = U'_+ \cp{1} U'_-$ and we can set $U_+ \eqdef U'_+$ and $U_- \eqdef U'_- \cp{1} U_x$, or
	\item $\bord{}{+}x \subseteq I$, since $I$ traces a path in $\hasseo{U}$ through 0\nbd dimensional and 1\nbd dimensional elements, and given that $\bord{}{+}x \subseteq \bord{}{+}U$, such a path can only enter $\bord{}{+}x$ through $\bord{0}{-}x$, traverse the entire $\bord{}{+}x$, and leave through $\bord{0}{+}x$. Then $I' \eqdef I[\bord{}{-}x/\bord{}{+}x]$ is well-defined and a 1\nbd molecule in $U'$; by the inductive hypothesis, we have a decomposition $U' = U'_+ \cp{1} U'_-$ relative to $I'$. Then setting $U_+ \eqdef U'_+ \cup \clos\{x\}$ and $U_- \eqdef U'_-$ produces a decomposition of $U$ relative to $I$. 
\end{itemize}
Uniqueness is straightforward since the removal of $\{x\} \cup (\bord{}{+}x \setminus \bord{0}{} x)$ from a decomposition of $U$ produces a decomposition of $U'$ either relative to $I$ or to $I'$. 

Let $V \submol U$ be a 2\nbd molecule with spherical boundary and $V \cap I = \bord{}{\alpha}V$. If $V$ is an atom, then clearly $V \submol U_\alpha$. Otherwise, observe that $I$ is not affected by the substitution $U[\compos{V}/V]$ and $\bord{}{\alpha}\compos{V} = \bord{}{\alpha}V \subseteq I$. Decomposing $U[\compos{V}/V]$ as $U'_+ \cp{1} U'_-$, by the atom case we have $\compos{V} \submol U'_\alpha$. Now $U'_\alpha[V/\compos{V}]$ and $U'_{-\alpha}$ are factors of a decomposition of $U$ relative to $I$, so by uniqueness $U_\alpha = U'_\alpha[V/\compos{V}]$ and $V \submol U_\alpha$. 
\end{proof}

\begin{prop} \label{prop:sim_substitution}
Let $U, V, W$ be regular 2\nbd molecules. Suppose $V$ and $W$ have spherical boundary, $V, W \submol U$, and $V \cap W \subseteq \bord V \cup \bord W$. Then $W \submol U[\compos{V}/V]$ and $V \submol U[\compos{W}/W]$. 
\end{prop}
\begin{proof}
Fix $\alpha \in \{+,-\}$ and let $U' \eqdef U[\compos{V}/V]$; by assumption, as a closed subset $W$ is unaffected by this substitution. 

We construct a sequence of 1\nbd molecules $I_0,\ldots, I_n$ as follows. Let $I_0 \eqdef \bord{}{\alpha}W$. For $i \geq 0$, if $\bord{}{-}I_i = \bord{0}{-}U$, then let $k \eqdef i$ and move to the next cycle, otherwise pick a 1\nbd dimensional element $x$ with $\bord{}{+}x = \bord{}{-}I_i$ and let $I_{i+1} \eqdef \clos\{x\} \cp{0} I_i$. 

For $i \geq k$, if $\bord{}{+}I_i = \bord{0}{+}U$, then let $n \eqdef i$ and stop, otherwise pick a 1\nbd dimensional element $x$ with $\bord{}{-}x = \bord{}{+}I_i$ and let $I_{i+1} \eqdef I_i \cp{0} \clos\{x\}$. This process terminates by finiteness of $U'$ and acyclicity of $\hasseo{U}$. 

Now $I \eqdef I_n$ is unaffected by the reverse substitution $U = U'[V/\compos{V}]$, has $\bord{}{-}I = \bord{0}{-}U$ and $\bord{}{+}I = \bord{0}{+}U$, and $W \cap I = \bord{}{\alpha}W$. Consider the unique decomposition $U'_+ \cp{1} U'_-$ of $U'$ relative to $I$ given by Lemma \ref{lem:slice_decomposition}. Clearly $\compos{V} \submol U'_\beta$ for some $\beta \in \{+,-\}$, so $U_\beta \eqdef U'_\beta[V/\compos{V}]$ and $U_{-\beta} \eqdef U'_{-\beta}$ produces the unique decomposition of $U$ relative to $I$. 

Now observe that if we decompose relative to $I' \eqdef I[\bord{}{-\alpha}W/\bord{}{\alpha}W]$ instead of $I$, only the 2\nbd dimensional elements of $W$ ``switch sides'' in the factorisation, so we can vary $\alpha$ without affecting $\beta$. Choosing $\alpha \eqdef -\beta$, we have
\begin{equation*}
	V \submol U_{\beta}, \quad W \submol U_{-\beta}
\end{equation*}
and the substitution of $\compos{V}$ for $V$, or of $\compos{W}$ for $W$, only affects one factor. 
\end{proof}

\begin{comm}
As a consequence of Proposition \ref{prop:sim_substitution}, if $V$ and $W$ are submolecules with spherical boundary of a regular 2\nbd molecule $U$ that only overlap on their boundaries, then they can \emph{both} be substituted in $U$: if $U[W'/W]$ and $U[V'/V]$ are both defined as 2\nbd molecules, then so are $U[W'/W][V'/V]$ and $U[V'/V][W'/W]$, which are in fact equal. This generalises to an arbitrary number $V_1,\ldots,V_n \submol U$ of 2\nbd molecules such that $V_i \cap V_j \subseteq \bord V_i \cup \bord V_j$ for all $i, j \in \{1,\ldots,n\}$, $i \neq j$. 

Dimension 2 is, in fact, the largest dimension in which this result holds. The following is an example of a regular 3\nbd molecule for which the analogous statement fails; it is a simplified version of \cite[Section 8]{steiner1993algebra}, itself based on \cite[Example 3.11]{power1991pasting}.

The point in our proof that fails to generalise to higher dimensions is the seemingly innocuous fact that $\bord{}{+}W$ can always be extended to a 1\nbd molecule $I$ with $\bord I = \bord{0}{} U$. In the example below, $\bord{}{+}W$ cannot be extended to any 2\nbd molecule in $U[\compos{V}/V]$ whose boundary is equal to $\bord{1}{} U$.
\end{comm}

\begin{exm}
Let $U$ be the shape of the 3\nbd diagram
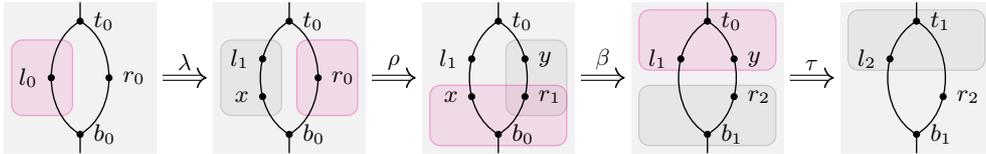
\begin{equation*}
	\input{img/counterexample}
\end{equation*}
where we use the labels of cells to refer to the corresponding atoms of $U$. Then both $V \eqdef \lambda \cup \tau$ and $W \eqdef \rho \cup \beta$ are submolecules of $U$, they have spherical boundary, and they do not share any 3\nbd atoms, so they only intersect in the boundary. 

However, $W$ is \emph{not} a submolecule of $U[\compos{V}/V]$, and $V$ is \emph{not} a submolecule of $U[\compos{W}/W]$. Indeed, there are paths
\begin{equation*}
	\rho \to y \to \compos{V} \to x \to \beta, \quad \quad \lambda \to x \to \compos{W} \to y \to \tau
\end{equation*}
in $\maxd{2}{U[\compos{V}/V]}$ and $\maxd{2}{U[\compos{W}/W]}$, respectively; note that we are confusing an atom with its greatest element. If $W \submol U[\compos{V}/V]$ or $V \submol U[\compos{W}/W]$, then it would be possible to substitute $\compos{W}$ for $W$ in $U[\compos{V}/V]$, or $\compos{V}$ for $V$ in $U[\compos{W}/W]$, to obtain a regular 3\nbd molecule $U'$. These paths would then become cycles in $\maxd{2}U'$, contradicting Proposition \ref{prop:codim1_acyclic}.
\end{exm}

\begin{thm} \label{thm:acyclic_3}
Let $P$ be a regular directed complex with $\dmn{P} \leq 3$. Then $P$ is frame-acyclic.
\end{thm}
\begin{proof}
It suffices to show that for all regular 3\nbd molecules $U$, if $\frdmn{U} = k$, then $\maxd{k}{U}$ is acyclic. The case $k = 2$ is handled by Proposition \ref{prop:codim1_acyclic}, so suppose $k \in \{0,1\}$.

By \cite[Lemma 2.5]{hadzihasanovic2020diagrammatic} we can decompose $U$ as $V_1 \cp{2} \ldots \cp{2} V_m$, where $V_i$ contains a unique 3\nbd dimensional element $x_i$  and $\bord{}{\alpha}V_i = \bord{}{-\alpha}V_i[\bord{}{\alpha}x_i/\bord{}{-\alpha}x_i]$ for all $\alpha \in \{+,-\}$ and $i \in \{1,\ldots,m\}$. 

Since $\clos\{x_i\} \cap \clos\{x_j\}$ has dimension at most 1 when $i \neq j$, we have that
\begin{enumerate}
	\item $\bord{}{-}x_i \subseteq \bord{}{-}U$ for all $i \in \{1,\ldots, m\}$, and
	\item by \cite[Lemma 18]{hadzihasanovic2018combinatorial}, $\clos\{x_i\} \cap \clos\{x_j\} = \bord{}{-}x_i \cap \bord{}{-}x_j \subseteq \bord{1}{}x_i \cup \bord{1}{}x_j$ when $i \neq j$.
\end{enumerate}
Since for all $i \in \{1,\ldots,m\}$ we have $\bord{}{-}x_i \submol \bord{}{-}V_i$ and
\begin{equation*}
	\bord{}{-}U = \bord{}{-}V_{i}[\bord{}{-}x_{i-1}/\bord{}{+}x_{i-1}]\ldots[\bord{}{-}x_1/\bord{}{+}x_1],
\end{equation*}
applying Proposition \ref{prop:sim_substitution} repeatedly we find that $\bord{}{-}x_i \submol \bord{}{-}U$ and the simultaneous substitution
\begin{equation} \label{eq:sim_sub}
	U' \eqdef \bord{}{-}U[\compos{\bord{}{-}x_1}/\bord{}{-}x_1]\ldots[\compos{\bord{}{-}x_m}/\bord{}{-}x_m]
\end{equation}
is defined as a regular 2\nbd molecule with the same frame dimension as $U$.

Now from every path in $\maxd{k}{U}$, we construct a path in $\maxd{k}{U'}$ as follows. The path in $\maxd{k}{U}$ is a concatenation of two-step paths $y_- \to x \to y_+$, where $x$ is maximal in $U$, $y_- \in \bord{1}{-}x$ and $y_+ \in \bord{1}{+}x$. 

If $\dmn{x} < 3$, then this path stays inside $\bord{}{-}U$, and $x, y_-, y_+$ are unaffected by the substitution (\ref{eq:sim_sub}). If $\dmn{x} = 3$, then this path can be replaced by a path $y_- \to \tilde{x} \to y_+$ in $\maxd{k}{U'}$, where $\tilde{x}$ is the greatest element of $\compos{\bord{}{-}x}$. 

Assuming there is a cycle in $\maxd{k}{U}$, with this procedure we construct a cycle in $\maxd{k}{U'}$, which contradicts Corollary \ref{cor:dim2_acyclic}. Thus $\maxd{k}{U}$ is acyclic.
\end{proof}

\section{Pros and diagrammatic sets} \label{sec:prodiag}

\subsection{Diagrammatic nerve of a pro} \label{sec:diag_nerve}

\begin{dfn}
Given a regular directed complex $P$ and $n \in \mathbb{N}$, let $\skel{n}{P} \subseteq P$ be the closed subset of elements $x \in P$ with $\dmn{x} \leq n$. Then $\mol{}{(\skel{n}{P})}^*$ and $\skel{n}{\mol{}{P}^*}$ are isomorphic $n$\nbd categories. 

By Lemma \ref{lem:split_lemma} combined with Theorem \ref{thm:acyclic_3}, for $n \leq 3$ the $n$\nbd category $\skel{n}\mol{}{P}^*$ admits the structure of a polygraph with $\{\clos\{x\} \mid \dmn{x} \leq n\}$ as generating cells. Because, in general, for an $\omega$\nbd category $X$,
\begin{equation*}
	\skel{k}{X} = \skel{k}{(\coskel{n}{X})} \text{ when } k < n,
\end{equation*}
the 2-category $\coskel{2}{\mol{}{P}^*}$ has the structure of a bicoloured pro with generators $\{\clos\{x\} \mid \dmn{x} \leq 1\}$.

Moreover, if $f\colon P \to Q$ is a morphism in $\rdcpx$, then $\coskel{2}{\mol{}{f}^*}$ sends each generator $\clos\{x\}$ to a generator $\clos\{f(x)\}$, so it is compatible with this structure. This defines a functor $\fun{P}\colon \rdcpx \to \bipro$ that fits into a commutative square
\begin{equation*}
\begin{tikzpicture}[baseline={([yshift=-.5ex]current bounding box.center)}]
	\node (0) at (-.5,1.5) {$\rdcpx$};
	\node (1) at (2.5,1.5) {$\bipro$};
	\node (2) at (-.5,0) {$\omegacat$};
	\node (3) at (2.5,0) {$\ncat{2}$.};
	\draw[1c] (0) to node[auto,arlabel] {$\fun{P}$} (1);
	\draw[1c] (0) to node[auto,arlabel,swap] {$\mol{}{-}^*$} (2);
	\draw[1c] (2) to node[auto,arlabel,swap] {$\coskel{2}$} (3);
	\draw[1c] (1) to node[auto,arlabel] {$\fun{U}$} (3);
\end{tikzpicture}
\end{equation*}
Because $\rdcpx$ is small, $\dgmset$ is locally small, and by Corollary \ref{cor:bipro_colimits} $\bipro$ has all small colimits, by \cite[Corollary 6.2.6]{riehl2017category} the left Kan extension of $\fun{P}\colon \rdcpx \to \bipro$ along the embedding $\rdcpx \incl \dgmset$ exists. This produces a functor $\fun{P}\colon \dgmset \to \bipro$.
\end{dfn}

\begin{rmk}
We may reason as in \cite[Proposition 7.10]{hadzihasanovic2020diagrammatic} to show that $\fun{P}\colon \rdcpx \to \bipro$ preserves the colimits that are already in $\rdcpx$, and deduce from [Corollary 1.34, \emph{ibid.}] that the left Kan extension of $\fun{P}$ along $\rdcpx \incl \dgmset$ is the left Kan extension of its restriction to $\atom$ along the Yoneda embedding.
\end{rmk}

\begin{dfn}[Diagrammatic nerve of bicoloured pros]
The \emph{diagrammatic nerve of bicoloured pros} is the right adjoint
\begin{equation*}
	\fun{N}\colon \bipro \to \dgmset
\end{equation*}
to the functor $\fun{P}\colon \dgmset \to \bipro$.
\end{dfn}

\begin{dfn}
In each bicoloured pro $(T,\gen{T})$, 
\begin{itemize}
	\item morphisms $\fun{P}I_n \to (T,\gen{T})$ classify 1\nbd cells in $T$ of the form $a_1 \cp{0} \ldots \cp{0} a_n$, where $a_i \in \gen{T}$ (including $\gen{T}_0$) for all $i \in \{1,\ldots,n\}$, and 
	\item morphisms $\fun{P}U_{n,m} \to (T,\gen{T})$ in $\bipro$ classify 2\nbd cells of type
\begin{equation*}
	a_1 \cp{0} \ldots \cp{0} a_n \celto b_1 \cp{0} \ldots \cp{0} b_m
\end{equation*}
in $T$ where $a_i, b_j \in \gen{T}$. 
\end{itemize} 
These correspond to morphisms $I_n \to \fun{N}(T,\gen{T})$ and $U_{n,m} \to \fun{N}(T,\gen{T})$, respectively, in $\dgmset$, that is, 1\nbd diagrams and 2\nbd cells in $\fun{N}(T,\gen{T})$.

If $U$ is a 3\nbd atom, a morphism $e\colon U \to \fun{N}(T,\gen{T})$ restricts, for each $\alpha \in \{+,-\}$, to a 2\nbd diagram $\bord{}{\alpha}e$ of shape $\bord{}{\alpha}U$ in $\fun{N}(T,\gen{T})$, whose transpose $\widehat{\bord{}{\alpha}e}\colon \fun{P}\bord{}{\alpha}U \to T$ is a diagram of 2\nbd cells in $T$. Because $T$ is a 2\nbd category,
\begin{enumerate}
	\item the morphism $\widehat{e}\colon \fun{P}U \to (T,\gen{T})$ exhibits an \emph{equation} between the composites of the diagrams $\widehat{\bord{}{+}e}$ and $\widehat{\bord{}{-}e}$ in $T$, and
	\item if $e'\colon U \to \fun{N}(T,\gen{T})$ is another 3\nbd cell with $\bord{}{\alpha}e' = \bord{}{\alpha}e$ for all $\alpha \in \{+,-\}$, then $e = e'$. 
\end{enumerate}
More in general, if $U$ is an atom, then a cell $U \to \fun{N}(T,\gen{T})$ is uniquely determined by its restriction $\skel{2}{e}$ to  $\skel{2}{U} \subseteq U$. 
\end{dfn}

\begin{lem} \label{lem:nerve_truncated}
Let $X$ be a diagrammatic set, $(T,\gen{T})$ a bicoloured pro, and let $f,g\colon X \to \fun{N}(T,\gen{T})$ be morphisms of diagrammatic sets. If $f(x) = g(x)$ for all 2\nbd cells $x$ in $X$, then $f = g$.
\end{lem}
\begin{proof}
Let $x\colon U \to X$ be a cell in $X$ with $\dmn{U} > 2$. Then $f(x)$ and $g(x)$ are uniquely determined by their restrictions 
\begin{equation*}
	\skel{2}{(f(x))} = (\skel{2}{x});f, \quad \quad \skel{2}{(g(x))} = (\skel{2}{x});g
\end{equation*}
to the directed complex $\skel{2}{U}$. If $f$ and $g$ agree on 2\nbd cells, these are equal. 
\end{proof}

\begin{prop} \label{prop:full_faithful}
The functor $\fun{N}$ is full and faithful.
\end{prop}
\begin{proof}
Suppose $\fun{N}f = \fun{N}g$ for two morphisms $f, g\colon (T,\gen{T}) \to (S,\gen{S})$ in $\bipro$. Given a 2\nbd cell $\varphi\colon (a_1, \ldots, a_n) \celto (b_1, \ldots, b_m)$ in $T$, classified by a morphism $\varphi\colon \fun{P}U_{n,m} \to (T,\gen{T})$ with transpose $\widehat{\varphi}\colon U_{n,m} \to \fun{N}(T,\gen{T})$, we have 
\begin{equation*}
	\widehat{\varphi};\fun{N}f = \widehat{\varphi};\fun{N}g
\end{equation*}
in $\dgmset$. It follows that $\varphi;f = \varphi;g$ in $\bipro$, that is, $f(\varphi) = g(\varphi)$. Because $f$ and $g$ agree on all 2\nbd cells, they are equal. This proves that $\fun{N}$ is faithful.

Let $f'\colon \fun{N}(T,\gen{T}) \to \fun{N}(S,\gen{S})$ be a morphism of diagrammatic sets. Given a 2\nbd cell $\varphi\colon (a_1, \ldots, a_n) \celto (b_1, \ldots, b_m)$ in $T$, classified by $\varphi\colon \fun{P}U_{n,m} \to (T,\gen{T})$ with transpose $\widehat{\varphi}\colon U_{n,m} \to \fun{N}(T,\gen{T})$, we define $f(\varphi)$ to be the unique 2\nbd cell in $S$ whose classifying morphism $f(\varphi)\colon \fun{P}U_{n,m} \to (S,\gen{S})$ is the transpose of $\widehat{\varphi}; f'\colon U_{n,m} \to \fun{N}(S,\gen{S})$. 

We want to show that $f$ determines a morphism of bicoloured pros. It is straightforward to verify that $f$ is compatible with all boundaries and with composition and units for 1\nbd cells. 

Let $x$ be a 1\nbd cell in $T$, classified by $x\colon \fun{P}I_n \to (T,\gen{T})$ with transpose $\widehat{x}$; the unit $\eps{}{x}$ is classified by $\eps{}{x}\colon \fun{P}U_{n,n} \to (T,\gen{T})$ with transpose $\widehat{\eps{}{x}}$. Let
\begin{equation*}
	U \eqdef \infl{I_n} \celto U_{n,n},
\end{equation*}
where $\infl{-}$ is the construction of \cite[\S 2.21]{hadzihasanovic2020diagrammatic}. This is well-defined as a regular 3\nbd atom. There is a unique cell $e\colon U \to \fun{N}(T,\gen{T})$ such that
\begin{enumerate}
	\item $e$ is equal to $\eps{}{\widehat{x}}$ on $\bord{}{-}U$, see [\S 4.16, \emph{ibid.}], and
	\item $e$ is equal to $\widehat{\eps{}{x}}$ on $\bord{}{+}U$.
\end{enumerate}
Then $e;f'\colon U \to \fun{N}(S,\gen{S})$ is a 3\nbd cell of type $\eps{}f'(\widehat{x}) \celto f'(\widehat{\eps{}{x}})$, whose transpose exhibits the equation $\eps{}{f(x)} = f(\eps{}{x})$ in $S$. 

Next, let $\varphi, \psi$ be 2\nbd cells in $T$, classified by morphisms 
\begin{equation*}
	\varphi\colon \fun{P}U_{n,m} \to (T,\gen{T}), \quad \quad \psi\colon \fun{P}U_{p, \ell} \to (T,\gen{T})
\end{equation*}
with transposes $\widehat{\varphi}$, $\widehat{\psi}$. Suppose that $\varphi \cp{1} \psi$ is defined; then we may assume $p = m$, and the composite is classified by
\begin{equation*}
	\varphi \cp{1} \psi \colon \fun{P}U_{n,\ell} \to (T,\gen{T})
\end{equation*} 
with transpose $\widehat{\varphi \cp{1} \psi}$. Let 
\begin{equation*}
	U \eqdef (U_{n,m} \cp{1} U_{m,\ell}) \celto U_{n,\ell};
\end{equation*}
this is a regular 3\nbd atom. There is a unique cell $e\colon U \to \fun{N}(T,\gen{T})$ such that
\begin{enumerate}
	\item $e$ is equal to $\widehat{\varphi}$ on $U_{n,m} \incl \bord{}{-}U$ and to $\widehat{\psi}$ on $U_{m,\ell} \incl \bord{}{-}U$, and
	\item $e$ is equal to $\widehat{\varphi \cp{1} \psi}$ on $\bord{}{+}U$. 
\end{enumerate}
Then $e;f'\colon U \to \fun{N}(S,\gen{S})$ is a cell of type 
\begin{equation*}
	f'(\widehat{\varphi}) \cp{1} f'(\widehat{\psi}) \celto f'(\widehat{\varphi \cp{1} \psi}),
\end{equation*}
whose transpose exhibits an equation $f(\varphi) \cp{1} f(\psi) = f(\varphi \cp{1} \psi)$ in $S$.

Finally, suppose that $\varphi \cp{0} \psi$ is defined; this composite is classified by 
\begin{equation*}
	\varphi \cp{0} \psi \colon \fun{P}U_{n+p,m+\ell} \to (T,\gen{T})
\end{equation*}
with transpose $\widehat{\varphi \cp{0} \psi}$. Let
\begin{equation*}
	U \eqdef ((U_{n,m} \cp{0} U_{p,\ell}) \cp{1} U_{m+\ell,m+\ell}) \celto U_{n+p,m+\ell}.
\end{equation*}
This is a regular 3\nbd atom and there is a unique cell $e\colon U \to \fun{N}(T,\gen{T})$ such that
\begin{enumerate}
	\item $e$ is equal to $\widehat{\varphi}$ on $U_{n,m} \incl \bord{}{-}U$ and to $\widehat{\psi}$ on $U_{p,\ell} \incl \bord{}{-}U$,
	\item $e$ is equal to the transpose of $\eps{}{(\bord{}{+}\varphi\cp{0}\bord{}{+}\psi)}$ on $U_{m+\ell,m+\ell} \incl \bord{}{-}U$, and
	\item $e$ is equal to $\widehat{\varphi \cp{0} \psi}$ on $\bord{}{+}U$.
\end{enumerate}
Then $e;f'\colon U \to \fun{N}(S,\gen{S})$ is a cell of type 
\begin{equation*}
	(f'(\widehat{\varphi}) \cp{0} f'(\widehat{\psi})) \cp{1} f'(\widehat{\eps{}{(\bord{}{+}\varphi\cp{0}\bord{}{+}\psi)}}) \celto f'(\widehat{\varphi \cp{0} \psi}).
\end{equation*}
whose transpose exhibits the equation
\begin{equation*}
	(f(\varphi) \cp{0} f(\psi)) \cp{1} f(\eps{}{(\bord{}{+}\varphi\cp{0}\bord{}{+}\psi)}) = f(\varphi \cp{0} \psi)
\end{equation*}
in $S$. Because we already know that $f$ is compatible with units, we deduce that $f(\varphi) \cp{0} f(\psi) = f(\varphi \cp{0} \psi)$. 

This proves that $f\colon (T,\gen{T}) \to (S,\gen{S})$ is a morphism of bicoloured pros. Now $\fun{N}f$ and $f'$ are morphisms $\fun{N}(T,\gen{T}) \to \fun{N}(S,\gen{S})$ that, by construction, agree on all 2\nbd cells of $\fun{N}(T,\gen{T})$. It follows from Lemma \ref{lem:nerve_truncated} that $\fun{N}f = f'$. This proves that $\fun{N}$ is full.
\end{proof}

\begin{comm}
String diagrams are commonly used to depict cells in a pro, usually after an appeal to the Joyal--Street soundness result \cite{joyal1991geometry}. The diagrammatic nerve construction offers an alternative justification, where diagrams are attributed a combinatorial, rather than topological interpretation. 

Unless otherwise stated, our string diagrams will represent diagrams in a diagrammatic set. A \emph{caveat} is that, contrary to custom, we are not allowed to have nodes with no input or output wires; instead, we need to explicitly introduce \emph{units} and \emph{unitors} \cite[\S 4.17]{hadzihasanovic2020diagrammatic} where necessary. 

To distinguish them visually, we draw unit 1\nbd cells as dotted wires, and unitor 2\nbd cells as ``dotless nodes'': for example, a 2\nbd cell of type $(0) \celto (1)$ in a one-sorted pro will be depicted as
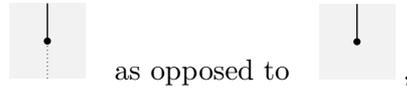
\begin{equation*}
	\begin{tikzpicture}[scale=.5]
\begin{scope}
\begin{pgfonlayer}{bg}
	\path[fill, color=gray!10] (-1,-1) rectangle (1,1);
\end{pgfonlayer}
\begin{pgfonlayer}{mid}
	\draw[wire] (0,0) to (0,1); 
	\draw[wiredot] (0,-1) to (0,0);
	\node[dot] at (0,0) {};
\end{pgfonlayer}
\end{scope}
\end{tikzpicture}
 \quad \text{as opposed to} \quad \begin{tikzpicture}[scale=.5]
\begin{scope}
\begin{pgfonlayer}{bg}
	\path[fill, color=gray!10] (-1,-1) rectangle (1,1);
\end{pgfonlayer}
\begin{pgfonlayer}{mid}
	\draw[wire] (0,0) to (0,1); 
	\node[dot] at (0,0) {};
\end{pgfonlayer}
\end{scope}
\end{tikzpicture}
\; ,
\end{equation*}
while a left unitor 2\nbd cell will be depicted as
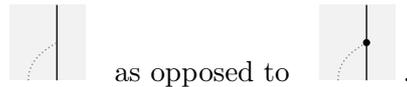
\begin{equation*}
	\begin{tikzpicture}[scale=.5]
\begin{scope}
\begin{pgfonlayer}{bg}
	\path[fill, color=gray!10] (-1,-1) rectangle (1,1);
\end{pgfonlayer}
\begin{pgfonlayer}{mid}
	\draw[wire] (.25,-1) to (.25,1); 
	\draw[wiredot, out=90, in=-150] (-.5,-1) to (.25,0);
\end{pgfonlayer}
\end{scope}
\end{tikzpicture}
 \quad \text{as opposed to} \quad 	\begin{tikzpicture}[scale=.5]
\begin{scope}
\begin{pgfonlayer}{bg}
	\path[fill, color=gray!10] (-1,-1) rectangle (1,1);
\end{pgfonlayer}
\begin{pgfonlayer}{mid}
	\draw[wire] (.25,-1) to (.25,1); 
	\draw[wiredot, out=90, in=-150] (-.5,-1) to (.25,0);
	\node[dot] at (.25,0) {};
\end{pgfonlayer}
\end{scope}
\end{tikzpicture}
\;.
\end{equation*}
This may seem like unnecessary trouble in dimension 2; the pay-off is that diagrammatic sets provide sound diagrammatic reasoning in \emph{all} dimensions.
\end{comm}


\subsection{Realisation of diagrammatic sets in Gray-categories} \label{sec:diag_gray}

\begin{dfn}
Our next goal is to construct a functor $\fun{G}\colon \rdcpx \to \graycat$, different from the ``obvious'' one obtained by composing $\mol{}{-}^*\colon \rdcpx \to \omegacat$ with $\coskel{3}{}\colon \omegacat \to \ncat{3}$ and then including $\ncat{3}$ in $\graycat$. In particular, $\fun{G}P$ will in general have non-trivial interchangers, so it will not be a strict 3\nbd category.

Every regular directed complex is the colimit of the diagram of inclusions of its atoms \cite[Corollary 1.34]{hadzihasanovic2020diagrammatic}. We impose that $\fun{G}$ preserve these colimit diagrams. Then it suffices to define $\fun{G}$ on atoms of increasing dimension. For each $n \in \mathbb{N} + \{-1\}$, let $\atom_n$ be the full subcategory of $\atom$ on the atoms of dimension $\leq n$. 
\end{dfn}

\begin{dfn}[{$\fun{G}$ in dimension $\leq 2$}]
On regular atoms of dimension $\leq 2$, we define $\fun{G}$ to be $\mol{}{-}\colon \atom_2 \to \ncat{3}$ followed by the embedding $\ncat{3} \incl \graycat$. We extend $\fun{G}$ along colimits to all regular directed complexes of dimension $\leq 2$.
\end{dfn}

\begin{dfn} \label{comm:gray_presentation}
Let $P$ be a 2\nbd dimensional regular directed complex. Then $\fun{G}(\skel{1}{P})$ is equal to (the image under the embedding $\ncat{3} \incl \graycat$ of) $\mol{}{\skel{1}{P}}^*$ and has the structure of a 1\nbd (pre)polygraph with the 1\nbd atoms of $P$ as generators. Now, for all 2\nbd atoms $x \in P$,
\begin{equation*} 
\begin{tikzpicture}[baseline={([yshift=-.5ex]current bounding box.center)}]
	\node (0) at (-.5,1.5) {$\bord O^2$};
	\node (1) at (2.5,0) {$\mol{}{(\clos\{x\})}$};
	\node (2) at (-.5,0) {$\mol{}{(\bord x)}$};
	\node (3) at (2.5,1.5) {$O^2$};
	\draw[1cinc] (0) to (3);
	\draw[1c] (0) to (2);
	\draw[1cinc] (2) to (1);
	\draw[1c] (3) to node[auto,arlabel] {$\clos\{x\}$} (1);
	\draw[edge] (1.6,0.2) to (1.6,0.7) to (2.3,0.7);
\end{tikzpicture}
\end{equation*}
is a pushout both in $\omegaprecat$ and $\graycat$. By the dual of the pullback lemma, the pushout of the span
\begin{align*}
	\coprod_{\dmn{x} = n} \mol{}{(\bord x)} & \incl \coprod_{\dmn{x} = n} \mol{}{(\clos\{x\})}, \\ 
	\coprod_{\dmn{x} = n} \mol{}{(\bord x)} & \incl \mol{}{\skel{1}{P}}^*
\end{align*}
in $\omegaprecat$ determines a 2\nbd prepolygraph $(\fun{G}P)_2$, while in $\graycat$ it is equivalent to the construction of $\fun{G}P$. The results of \cite[Section 1.6]{forest2018coherence} imply that 
\begin{enumerate}
	\item $\fun{G}P$ is obtained from $(\fun{G}P)_2$ by freely attaching some 3\nbd cells (interchange generators) indexed by generating cells of $(\fun{G}P)_2$, and imposing some equations of 3\nbd cells, so in particular
	\item $(\fun{G}P)_2$ is the 2\nbd skeleton of $\fun{G}P$.
\end{enumerate}
In the terminology of Forest and Mimram, $P$ determines a \emph{presentation} of the 2\nbd precategory $(\fun{G}P)_2$, which can be completed to a \emph{Gray presentation} of $\fun{G}P$ by freely adding the necessary \emph{structural generators}. 
\end{dfn}

\begin{lem} \label{lem:2cells_in_gu}
Let $U$ be a regular 2\nbd molecule. There is a bijective correspondence between
\begin{enumerate}
	\item cells of rank 2 in $\fun{G}U$, and 
	\item 2\nbd molecules $V \subseteq U$ together with a 1\nbd order.
\end{enumerate}
\end{lem}
\begin{proof}
By the discussion in \S\ref{comm:gray_presentation}, the 2\nbd cells in $\fun{G}U$ are the same as the 2\nbd cells in the 2\nbd prepolygraph $(\fun{G}U)_2$, so they are freely generated by the atoms of $U$ under principal compositions (\S \ref{dfn:principal}), subject to the axioms of $\omega$\nbd precategories. 

Let $V$ be a 2\nbd molecule with a 1\nbd order $(x_1, \ldots, x_m)$. By Lemma \ref{lem:frame_decomposition}, we obtain a decomposition
\begin{equation} \label{eq:2decomposition}
	V = V_1 \cp{1} \ldots \cp{1} V_m.
\end{equation} 
Now each $V_i$ has frame dimension 0 or -1, so it has a (clearly unique) decomposition 
\begin{equation} \label{eq:1decomposition}
	\clos\{y_{i,1}\} \cp{0} \ldots \cp{0} \clos\{y_{i,k}\} \cp{0} \clos\{x_i\} \cp{0} \clos\{y_{i,k+1}\} \cp{0} \ldots \cp{0} \clos\{y_{i,p}\}
\end{equation}
where $\dmn{y_{i,j}} = 1$ for all $j \in \{1,\ldots,p\}$. Replacing the (\ref{eq:1decomposition}) into (\ref{eq:2decomposition}), we obtain a decomposition of $U$ into atoms using only principal compositions, which determines a cell of rank 2 in $(\fun{G}U)_2$.

Conversely, by \cite[Proposition 2]{forest2018coherence}, every cell $y$ of rank 2 in $(\fun{G}U)_2$ has a unique expression of the form $y_1 \cp{1} \ldots \cp{1} y_m$ where $y_i$ is an expression of the form (\ref{eq:1decomposition}). Now the expression of $y$ is also a valid expression for a 2\nbd cell in $\mol{}{U}^*$, which by Remark \ref{rmk:2molecule_loopfree} is equal to $\mol{}{U}$, so it determines a 2\nbd molecule $V \subseteq U$ together with a decomposition into atoms. From this decomposition we recover uniquely a 1\nbd order $(x_1, \ldots, x_m)$ on $V$. The two constructions are clearly inverse to each other.
\end{proof}

\begin{rmk}
By Lemma \ref{lem:2cells_in_gu}, every cell of rank 2 in $\fun{G}U$ is identified uniquely by a pair $(V, (x_i)_{i=1}^m)$ of a 2\nbd molecule and a 1\nbd order. 

More in general, if $P$ is a 2\nbd dimensional regular directed complex, a pair $(V, (x_i)_{i=1}^m)$ of a 2\nbd molecule in $P$ and a 1\nbd order on it determines a unique cell of rank 2 in $\fun{G}P$, although these may not exhaust all cells of rank 2 when $P$ is not totally loop-free. 
\end{rmk}

\begin{prop} \label{prop:gu_coherence}
Let $U$ and $V \subseteq U$ be regular 2\nbd molecules and let $(x_1, \ldots, x_m)$ and $(x'_1, \ldots, x'_m)$ be two 1\nbd orders on $V$. Then in $\fun{G}U$ there is a unique 3\nbd cell from $(V, (x_i)_{i=1}^m)$ to $(V, (x'_i)_{i=1}^m)$.
\end{prop}
\begin{proof}
For each 1\nbd order $(x_1, \ldots, x_m)$ on $V$, let $\fun{w}((x_i)_{i=1}^m)$ be equal to the number of pairs $(i,j)$ such that $i < j$ but $x_j \prec x_i$ in the total order on $U$. Then
\begin{itemize}
	\item $0 \leq \fun{w}((x_i)_{i=1}^m) \leq {m \choose 2}$, 
	\item $\fun{w}((x_i)_{i=1}^m) = 0$ if and only if $(x_i)_{i=1}^m$ is the normal 1\nbd order, and
	\item for all non-trivial interchangers $\chi_{x,y}\colon (V, (x_i)_{i=1}^m) \celto (V, (x'_i)_{i=1}^m)$, we have $\fun{w}((x_i)_{i=1}^m) < \fun{w}((x'_i)_{i=1}^m)$.
\end{itemize}
It follows that $\fun{w}(-)$ induces a \emph{termination order} on the 2\nbd cells of $\fun{G}U$ in the terminology of \cite[Section 2.2]{forest2018coherence}. Because the Gray presentation of $\fun{G}U$ determined by $U$ as in \S\ref{comm:gray_presentation} has no non-structural 3\nbd generators, it is always locally confluent, so [Theorem 11, \emph{ibid.}] applies and $\fun{G}U$ has at most one 3\nbd cell between any parallel pair of 2\nbd cells. This proves uniqueness. 

For existence, it suffices to observe that, if $\fun{w}((x_i)_{i=1}^m) > 0$, then there is a non-trivial inverse interchanger with input $(V, (x_i)_{i=1}^m)$; we leave the proof as an exercise. Applying inverse interchangers repeatedly, we obtain invertible 3\nbd cells of type
\begin{equation*}
	(V, (x_i)_{i=1}^m) \celto (V, \text{normal 1\nbd order}), \quad (V, (x'_i)_{i=1}^m) \celto (V, \text{normal 1\nbd order}).
\end{equation*}
Composing the first with the inverse of the second produces a 3\nbd cell of type $(V, (x_i)_{i=1}^m) \celto (V, (x'_i)_{i=1}^m)$.
\end{proof}

\begin{dfn}[{$\fun{G}$ in dimension 3}]
Let $U$ be a regular 3\nbd atom. We define $\fun{G}U$ to be the pushout
\begin{equation*} 
\begin{tikzpicture}[baseline={([yshift=-.5ex]current bounding box.center)}]
	\node (0) at (-.5,1.5) {$\bord O^3$};
	\node (1) at (2.5,0) {$\fun{G}U$};
	\node (2) at (-.5,0) {$\fun{G}(\bord U)$};
	\node (3) at (2.5,1.5) {$O^3$};
	\draw[1cinc] (0) to (3);
	\draw[1c] (0) to node[auto,swap,arlabel] {$f$} (2);
	\draw[1cinc] (2) to (1);
	\draw[1c] (3) to (1);
	\draw[edge] (1.6,0.2) to (1.6,0.7) to (2.3,0.7);
\end{tikzpicture}
\end{equation*}
in $\graycat$, where $f$ sends $\undl{2}^\alpha$ to $(\bord{}{\alpha}U, \text{normal $1$\nbd order})$ for each $\alpha \in \{+,-\}$. 

Now every map $f\colon U \to V$ in $\atom_3$ determines an assignment of generators of $\fun{G}V$ to generators of $\fun{G}U$ which is compatible with boundaries, hence extends uniquely to a functor $\fun{G}f\colon \fun{G}U \to \fun{G}V$. This defines $\fun{G}\colon \atom_3 \to \graycat$. We extend $\fun{G}$ along colimits to all regular directed complexes of dimension $\leq 3$.
\end{dfn}

\begin{dfn}
By construction, if $P$ is a regular directed complex of dimension 3, we can associate to each 3\nbd atom $U$ of $P$ a 3\nbd cell
\begin{equation*}
	\intp{U}\colon (\bord{}{-}U, \text{normal 1\nbd order}) \celto (\bord{}{+}U, \text{normal 1\nbd order})
\end{equation*}
in $\fun{G}P$. We want to extend this assignment to all 3\nbd molecules in $P$.
\end{dfn}

\begin{dfn} \label{dfn:interpret_single_atom}
Suppose that $U$ contains a single 3\nbd dimensional element $x$. Then $\bord{}{\alpha}x \submol \bord{}{\alpha}U$ for all $\alpha \in \{+,-\}$, the substitution $\bord{}{\alpha}U[\compos{\bord{}{\alpha}x}/\bord{}{\alpha}x]$ is well-defined, and 
\begin{equation*}
	\bord{}{-}U[\compos{\bord{}{-}x}/\bord{}{-}x] = \bord{}{+}U[\compos{\bord{}{+}x}/\bord{}{+}x].
\end{equation*} 
Pick a 1\nbd order $(x_i)_{i=1}^m$ on $\bord{}{\alpha}U[\compos{\bord{}{\alpha}x}/\bord{}{\alpha}x]$; then $\clos\{x_k\} = \compos{\bord{}{\alpha}x}$ for a unique $k \in \{1,\ldots,m\}$. Let $y_1,\ldots,y_p$ be the normal 1\nbd order on $\bord{}{-}x$ and let $z_1,\ldots,z_q$ be the normal 1\nbd order on $\bord{}{+}x$. Then
\begin{align*}
	(x^-_i)_{i=1}^{m+p-1} & \eqdef (x_1,\ldots,x_{k-1}, y_1,\ldots,y_p, x_{k+1},\ldots,x_m), \\
	(x^+_i)_{i=1}^{m+q-1} & \eqdef (x_1,\ldots,x_{k-1}, z_1,\ldots,z_q, x_{k+1},\ldots,x_m)
\end{align*}
are 1\nbd orders on $\bord{}{-}U$ and $\bord{}{+}U$, respectively.

Now substituting $\intp{\clos\{x\}}$ for $\clos\{x_k\}$ in the decomposition of $\bord{}{\alpha}U[\compos{\bord{}{\alpha}x}/\bord{}{\alpha}x]$ corresponding to the 1\nbd order $(x_i)_{i=1}^m$ yields a valid expression for a 3\nbd cell
\begin{equation} \label{eq:3_cell_in_context}
	c[x]\colon (\bord{}{-}U, (x^-_i)_{i=1}^{m+p-1}) \celto (\bord{}{+}U, (x^+_i)_{i=1}^{m+q-1})
\end{equation}
in $\fun{G}P$. By Proposition \ref{prop:gu_coherence}, there are unique 3\nbd cells
\begin{align*}
	\chi^{-}\colon & (\bord{}{-}U, \text{normal 1\nbd order}) \celto (\bord{}{-}U, (x^-_i)_{i=1}^{m+p-1}), \\
	\chi^{+} \colon & (\bord{}{+}U, (x^+_i)_{i=1}^{m+q-1}) \celto (\bord{}{+}U, \text{normal 1\nbd order})
\end{align*}
obtained as composites of interchangers and inverse interchangers, respectively. We define $\intp{U}$ to be the composite 
\begin{equation*}
	\chi^{-} \cp{2} c[x] \cp{2} \chi^+\colon (\bord{}{-}U, \text{normal 1\nbd order}) \celto (\bord{}{+}U, \text{normal 1\nbd order}).
\end{equation*}
We need to show that this is independent of our choice of 1\nbd order $(x_i)_{i=1}^m$. Suppose $(x'_i)_{i=1}^m$ is another 1\nbd order on $\bord{}{\alpha}U[\compos{\bord{}{\alpha}x}/\bord{}{\alpha}x]$, leading to a potentially different interpretation $\chi'^{-} \cp{2} c'[x] \cp{2} \chi'^+$. There are unique 3\nbd cells
\begin{align*}
	\psi^{-}\colon & (\bord{}{-}U, (x^-_i)_{i=1}^{m+p-1}) \celto (\bord{}{-}U, (x'^-_i)_{i=1}^{m+p-1}), \\
	\psi^{+}\colon & (\bord{}{+}U, (x^+_i)_{i=1}^{m+q-1}) \celto (\bord{}{+}U, (x'^+_i)_{i=1}^{m+q-1})
\end{align*}
obtained as composites of interchangers and inverse interchangers, and since they ``fix'' $\bord{}{-}x$ and $\bord{}{+}x$, by naturality of interchangers we have
\begin{equation*}
	\psi^- \cp{2} c'[x] = c[x] \cp{2} \psi^+
\end{equation*}
hence $c'[x] = \invrs{(\psi^-)} \cp{2} c[x] \cp{2} \psi^+$ and
\begin{equation*}
	\chi'^{-} \cp{2} c'[x] \cp{2} \chi'^+ = \chi'^{-} \cp{2} \invrs{(\psi^-)} \cp{2} c[x] \cp{2} \psi^+ \cp{2} \chi'^+.
\end{equation*}
Finally, by Proposition \ref{prop:gu_coherence}, $\chi'^{-} \cp{2} \invrs{(\psi^-)} = \chi^-$ and $\psi^+ \cp{2} \chi'^+ = \chi^+$. 
\end{dfn}

\begin{dfn} \label{dfn:interpret_general}
Let $U$ be any regular 3\nbd molecule in $P$ and fix a 2\nbd order $(x_1, \ldots, x_m)$ on $U$. Then Lemma \ref{lem:frame_decomposition} combined with Theorem \ref{thm:acyclic_3} gives a decomposition $U = V_1 \cp{2} \ldots \cp{2} V_m$ where $x_i$ is the only 3\nbd dimensional element of $V_i$ for each $i \in \{1,\ldots,m\}$. We let $\intp{U}$ in $\fun{G}P$ be the composite $\intp{V_1} \cp{2} \ldots \cp{2} \intp{V_m}$ of the 3\nbd cells
\begin{equation*}
	\intp{V_i}\colon (\bord{}{-}V_i, \text{normal 1\nbd order}) \celto (\bord{}{+}V_i, \text{normal 1\nbd order})
\end{equation*}
defined in \S \ref{dfn:interpret_single_atom}.

We need to show that this interpretation is independent of the 2\nbd order chosen on $U$. Observe that any pair of 2\nbd orders on $U$ is related by a sequence of elementary transpositions of consecutive elements that are not connected by a path in $\maxd{2}{U}$. Thus it suffices to show that if
\begin{equation*}
	W_1 \cp{2} W_2 = W'_1 \cp{2} W'_2
\end{equation*}
as 3\nbd molecules, where $x$ is the only 3\nbd dimensional element in $W_1$ and $W'_2$, while $y$ is the only 3\nbd dimensional element in $W_2$ and $W'_1$, then 
\begin{equation*}
	\intp{W_1} \cp{2} \intp{W_2} = \intp{W'_1} \cp{2} \intp{W'_2}.
\end{equation*}

The interpretation of $W_1$ involves a choice of 1\nbd order on $\bord{}{\alpha}W_1[\compos{\bord{}{\alpha}x}/\bord{}{\alpha}x]$ but it is independent of this choice. Now $\bord{}{-}y \submol \bord{}{-}W_2 = \bord{}{+}W_1$, and since $x$ and $y$ are not connected by a path in $\maxd{2}{U}$, necessarily 
\begin{equation*}
	\bord{}{+}x \cap \bord{}{-}y \subseteq \clos\{x\} \cap \clos\{y\} \subseteq \bord{1}{}x \cup \bord{1}{}y,
\end{equation*}
and by Proposition \ref{prop:sim_substitution} $\bord{}{-}y \submol \bord{}{+}W_1[\compos{\bord{}{+}x}/\bord{}{+}x]$. 

Applying the known equalities between the boundaries of $W_1, W_2, W'_1, W'_2$, we deduce that the double substitutions
\begin{align*}
	& \bord{}{+}W_1[\compos{\bord{}{+}x}/\bord{}{+}x][\compos{\bord{}{-}y}/\bord{}{-}y], \\
	& \bord{}{-}W_2[\compos{\bord{}{-}y}/\bord{}{-}y][\compos{\bord{}{+}x}/\bord{}{+}x], \\
	& \bord{}{+}W'_1[\compos{\bord{}{+}y}/\bord{}{+}y][\compos{\bord{}{-}x}/\bord{}{-}x], \\
	& \bord{}{-}W'_2[\compos{\bord{}{-}x}/\bord{}{-}x][\compos{\bord{}{+}y}/\bord{}{+}y]
\end{align*}
are all well-defined and equal to the same regular 2\nbd molecule. Fix a 1\nbd order $(z_i)_{i=1}^p$ on it. Then $\compos{\bord{}{\alpha}x} = \clos\{z_k\}$ and $\compos{\bord{}{\beta}y} = \clos\{z_\ell\}$ for a unique pair $k, \ell \in \{1,\ldots,p\}$. Now
\begin{itemize}
	\item to interpret $W_1$, choose the 1\nbd order on $\bord{}{+}W_1[\compos{\bord{}{+}x}/\bord{}{+}x]$ obtained by replacing $z_\ell$ with the normal 1\nbd order on $\bord{}{-}y$ in $(z_i)_{i=1}^p$,
	\item to interpret $W_2$, choose the 1\nbd order on $\bord{}{-}W_2[\compos{\bord{}{-}y}/\bord{}{-}y]$ obtained by replacing $z_k$ with the normal 1\nbd order on $\bord{}{+}x$ in $(z_i)_{i=1}^p$,
\end{itemize}
and similarly for $W'_1$ and $W'_2$. With the construction of \S \ref{dfn:interpret_single_atom}, any choices of 1\nbd orders lead to expressions
\begin{align*}
	& \chi_1^- \cp{2} c[x] \cp{2} \chi_1^+ \cp{2} \chi_2^- \cp{2} c[y] \cp{2} \chi_2^+, \\
	& \chi'^-_1 \cp{2} c'[y] \cp{2} \chi'^+_1 \cp{2} \chi'^-_2 \cp{2} c'[x] \cp{2} \chi'^+_2,
\end{align*}
and with the particular choice that we made, 
\begin{equation*}
	\chi_1^- = \chi'^-_1, \quad \quad \chi_2^+ = \chi'^+_2, 
\end{equation*}
while $\chi_1^+ \cp{2} \chi_2^-$ and $\chi'^+_1 \cp{2} \chi'^-_2$ are units, so they can be eliminated. Finally,
\begin{equation*}
	\chi_1^- \cp{2} c[x] \cp{2} c[y] \cp{2} \chi_2^+ = \chi_1^- \cp{2} c'[y] \cp{2} c'[x] \cp{2} \chi_2^+
\end{equation*}
is a consequence of axiom \ref{ax:3interchange} of Gray\nbd categories. This proves that $\intp{W_1} \cp{2} \intp{W_2}$ is equal to $\intp{W'_1} \cp{2} \intp{W'_2}$, and we conclude that $\intp{U}$ is independent of the choice of 2\nbd order.
\end{dfn}

\begin{exm}
Let $U$ be the shape of diagram (\ref{eq:frobenius_diagram}). We introduce names for some atoms of $U$ as follows:
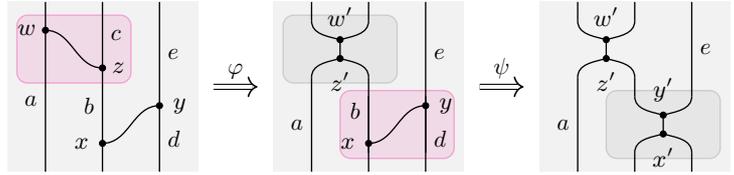
\begin{equation*}
	\input{img/frobenius_exm}
\end{equation*}
The 3\nbd atoms $\varphi$ and $\psi$ are interpreted in $\fun{G}U$ as 3\nbd cells
\begin{align*}
	\intp{\varphi}\colon & (a \cp{0} z) \cp{1} (w \cp{0} c) \celto z' \cp{1} w', \\
	\intp{\psi}\colon & (x \cp{0} d) \cp{1} (b \cp{0} y) \celto x' \cp{1} y';
\end{align*}
notice that in this case both $\bord{}{\alpha}\varphi$ and $\bord{}{\alpha}\psi$ admit a single 1\nbd order, which is necessarily the normal 1\nbd order.

We pick the 2\nbd order $(\varphi, \psi)$ on $U$, which determines the decomposition 
\begin{equation*}
	U = V_1 \cp{2} V_2, \quad 
	V_1 \eqdef \varphi \cup \bord{}{-}\psi, \quad  V_2 \eqdef \psi \cup \bord{}{+}\varphi.
\end{equation*}
To interpret $V_1$ in $\fun{G}U$, first we need to consider $V_1[\compos{\bord{}{-}\varphi}/\bord{}{-}\varphi]$. This is the shape of the diagram
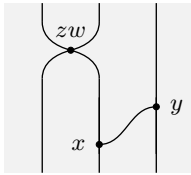
\begin{equation*}
	\input{img/frobenius_exm2}
\end{equation*}
on which we pick the 1\nbd order $(x, y, zw)$. The 3\nbd cell $c[\varphi]$ corresponding to this 1\nbd order, defined as in (\ref{eq:3_cell_in_context}), is
\begin{equation*}
	(a \cp{0} x \cp{0} d) \cp{1} (a \cp{0} b \cp{0} y) \cp{1} (\intp{\varphi} \cp{0} e)
\end{equation*}
of type $(\bord{}{-}V_1, (x, y, z, w)) \celto (\bord{}{+}V_1, (x, y, z', w'))$ in $\fun{G}U$.

The normal 1\nbd order on $\bord{}{-}V_1$ is in fact $(x,z,w,y)$. Applying a pair of interchangers to first move $w$ after $y$, then $z$ after $y$, we obtain a 3\nbd cell
\begin{equation*}
	\chi_1^-\colon (\bord{}{-}V_1, (x,z,w,y)) \celto (\bord{}{-}V_1, (x,y,z,w)).
\end{equation*}
Similarly, the normal 1\nbd order on $\bord{}{+}V_1$ is $(x,z',w',y)$, so we apply a pair of inverse interchangers to move $y$ after $z'$, then $y$ after $w'$, producing a 3\nbd cell
\begin{equation*}
	\chi_1^+\colon (\bord{}{+}V_1, (x,y,z',w')) \celto (\bord{}{+}V_1, (x,z',w',y)).
\end{equation*}
Then $\intp{V_1}$ is defined to be $\chi_1^- \cp{2} c[\varphi] \cp{2} \chi_1^+$.

Next, consider $V_2[\compos{\bord{}{-}\psi}/\bord{}{-}\psi]$. This is the shape of the diagram
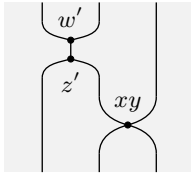
\begin{equation*}
	\input{img/frobenius_exm3}
\end{equation*}
which admits only the 1\nbd order $(xy, z', w')$. Correspondingly, we construct the 3\nbd cell
\begin{equation*}
	c[\psi] \eqdef (a \cp{0} \intp{\psi}) \cp{1} (z' \cp{0} e) \cp{1} (w' \cp{0} e)
\end{equation*}
which is of type $(\bord{}{-}V_2, (x,y,z',w')) \celto (\bord{}{+}V_2, (x',y',z',w'))$. The normal 1\nbd order on $\bord{}{-}V_2$ is $(x, z', w', y)$, and we have a composite of interchangers
\begin{equation*}
	\chi_2^-\colon (\bord{}{-}V_2, (x,z',w',y)) \celto (\bord{}{-}V_2, (x,y,z',w')),
\end{equation*}
which is in fact the inverse of $\chi_1^+$. On the other hand, $(x',y',z',w')$ is already the normal 1\nbd order on $\bord{}{+}V_2$, so  $\intp{V_2}$ is just $\chi_2^- \cp{2} c[\psi]$. Overall, $\intp{U}$ is 
\begin{equation*}
	\chi_1^- \cp{2} c[\varphi] \cp{2} c[\psi]\colon (\bord{}{-}U, (x,z,w,y)) \celto (\bord{}{+}U, (x',y',z',w')).
\end{equation*}

If we had picked the 2\nbd order $(\psi, \varphi)$, we would have instead ended up with the expression $\chi_1^- \cp{2} c'[\psi] \cp{2} c'[\varphi]$ where
\begin{align*}
	c'[\psi] & \eqdef (a \cp{0} \intp{\psi}) \cp{1} (a \cp{0} z \cp{0} e) \cp{1} (w \cp{0} c \cp{0} e), \\
	c'[\varphi] & \eqdef (a \cp{0} x') \cp{1} (a \cp{0} y') \cp{1} (\intp{\varphi} \cp{0} e).
\end{align*}
It follows from axiom \ref{ax:3interchange} of Gray\nbd categories that $c[\varphi] \cp{2} c[\psi] = c'[\psi] \cp{2} c'[\varphi]$, confirming that $\intp{U}$ is independent of the 2\nbd order on $U$.
\end{exm}

\begin{dfn}[{$\fun{G}$ in dimension $\geq 4$}]
Let $U$ be a regular 4\nbd atom. We define $\fun{G}U$ to be the quotient of $\fun{G}(\bord U)$ by the equation
\begin{equation*}
	\intp{\bord{}{-}U} = \intp{\bord{}{+}U},
\end{equation*}
where the 3\nbd molecules $\bord{}{\alpha}U$ are interpreted in $\fun{G}(\bord U)$ as by \S \ref{dfn:interpret_general}. 

If $f\colon U \to V$ is a map in $\atom_4$, its restriction to $\bord U$ determines a functor $\fun{G}(\bord f)\colon \fun{G}(\bord U) \to \fun{G}V$. If $U$ is a 4\nbd atom, then either $\dmn{f(U)} < 4$ and $f(\bord{}{-}U) = f(\bord{}{+}U)$, or $\dmn{f(U)} = 4$, $f(U) = V$ and $f(\bord{}{\alpha}U) = \bord{}{\alpha}V$ for each $\alpha \in \{+,-\}$. In either case, $\fun{G}(\bord f)$ is compatible with the equation $\intp{\bord{}{-}U} = \intp{\bord{}{+}U}$, so it factors uniquely through a functor $\fun{G}f\colon \fun{G}U \to \fun{G}V$. 

This defines $\fun{G}\colon \atom_4 \to \graycat$, and we extend it along colimits to all regular directed complexes of dimension $\leq 4$. 

Finally, if $f\colon P \to Q$ is a map of regular directed complexes of arbitrary dimension, it restricts to a map $\skel{4}{f}\colon \skel{4}{P} \to \skel{4}{Q}$, and we let $\fun{G}{f}$ be equal to $\fun{G}{(\skel{4}{f})}$. This defines $\fun{G}\colon \rdcpx \to \graycat$. 
\end{dfn}

\begin{comm}
By construction, $\fun{G}$ ignores any elements of dimension $> 4$. The idea is that, while 4\nbd atoms can contribute non-trivial equations of 3\nbd cells in a Gray\nbd category, higher\nbd dimensional atoms can only contribute trivial ``equations of equations'' with no visible effect.
\end{comm}

\begin{dfn}
Because $\graycat$ has all small colimits, we are in the conditions of \cite[Corollary 6.2.6]{riehl2017category} and we can define a functor
\begin{equation*}
	\fun{G}\colon \dgmset \to \graycat
\end{equation*}
as the left Kan extension of $\fun{G}\colon \rdcpx \to \graycat$ along the embedding $\rdcpx \incl \dgmset$. 
\end{dfn}

\begin{rmk}
Since we made sure at every step that $\fun{G}$ preserve the colimits in $\rdcpx$, this is in fact equal to the left Kan extension of the restriction of $\fun{G}$ to $\atom$ along the Yoneda embedding. 
\end{rmk}

\begin{rmk}
For the usual reasons, $\fun{G}$ has a right adjoint, of which we will not make use. Unlike the diagrammatic nerve of pros, it is not full; see \cite[Remark 7.20]{hadzihasanovic2020diagrammatic} for a counterexample that also applies to the present case.
\end{rmk}

\begin{comm}
The following (generally non-commutative) diagram of functors recaps the adjunctions that we have established:
	
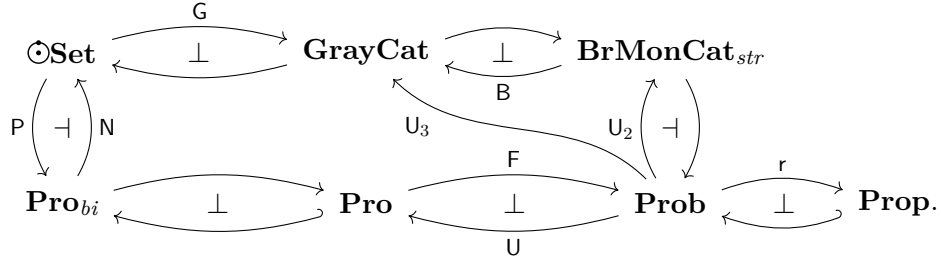
\begin{equation*} 
		\begin{tikzpicture}[baseline={([yshift=-.5ex]current bounding box.center)}]
		\node (1t) at (0,2) {$\dgmset$};
		\node (1b) at (0,0) {$\bipro$};
		\node (2b) at (4,0) {$\pro$};
		\node (3b) at (8,0) {$\prob$};
		\node (4b) at (11,0) {$\propp$.};
		\node (2t) at (4,2) {$\graycat$};
		\node (3t) at (8,2) {$\brmoncat$};
		\draw[1c,out=15,in=165] (1t.east |- 0,2.15) to node[auto,arlabel] {$\fun{G}$} (2t.west |- 0,2.15);
		\node at (1.8,2) {$\bot$};
		\draw[1c,out=-165,in=-15] (2t.west |- 0,1.85) to (1t.east |- 0,1.85);
		\draw[1c,out=20,in=160] (2t.east |- 0,2.15) to (3t.west |- 0,2.15);
		\node at (5.8,2) {$\bot$};
		\draw[1c,out=-160,in=-20] (3t.west |- 0,1.85) to node[auto,arlabel] {$\fun{B}$} (2t.east |- 0,1.85);
		\draw[1c,out=15,in=165] (1b.east |- 0,.15) to (2b.west |- 0,.15);
		\node at (2.05,0) {$\bot$};
		\draw[1cincl,out=-165,in=-15] (2b.west |- 0,-.15) to (1b.east |- 0,-.15);
		\draw[1c,out=-120,in=120] (1t) to node[auto,arlabel,swap] {$\fun{P}$} (1b);
		\node at (0,1) {$\dashv$};
		\draw[1c,out=60,in=-60] (1b) to node[auto,arlabel,swap] {$\fun{N}$} (1t);
		\draw[1c,out=15,in=165] (2b.east |- 0,.15) to node[auto,arlabel] {$\fun{F}$} (3b.west |- 0,.15);
		\node at (5.95,0) {$\bot$};
		\draw[1c,out=-165,in=-15] (3b.west |- 0,-.15) to node[auto,arlabel] {$\fun{U}$} (2b.east |- 0,-.15);
		\draw[1c,out=20,in=160] (3b.east |- 0,.15) to node[auto,arlabel] {$\fun{r}$} (4b.west |- 0,.15);
		\node at (9.5,0) {$\bot$};
		\draw[1cincl,out=-160,in=-20] (4b.west |- 0,-.15) to (3b.east |- 0,-.15);
		\draw[1c,out=120,in=-120] (3b) to node[auto,arlabel] {$\fun{U}_2$} (3t);
		\node at (8,1) {$\dashv$};
		\draw[1c,out=-60,in=60] (3t) to (3b);
		\draw[1c, out=135,in=-45] (3b) to node[auto,arlabel,pos=.8] {$\fun{U}_3$} (2t);
		\end{tikzpicture}
	\end{equation*}

\end{comm}

\section{The smash product} \label{sec:smash}

\subsection{The tensor product of pros} \label{sec:tensor_pros}

We reconstruct the tensor product of props, as defined by Hackney and Robertson, as a reflection of an ``external'' tensor product of pros producing a prob, whose combinatorics are only slightly more involved.

\begin{lem} \label{lem:normal_permutation}
Let $s$ be a permutation on the set $\{1,\ldots,n\}$. Then $s$ is either the identity or admits a unique decomposition
\begin{equation*}
	s = s_1;\ldots;s_p
\end{equation*}
with the following properties. For each $i \in \{1,\ldots,p\}$, let $s^{(i)} \eqdef s_i;\ldots;s_p$. Then
\begin{enumerate}
	\item $s_i$ is an elementary transposition $(k \; k+1)$ of two consecutive elements, and
	\item $k$ is the least element of $\{1,\ldots,n\}$ such that $s^{(i)}(k+1) < s^{(i)}(k)$.
\end{enumerate}
\end{lem}
\begin{proof}
We construct, step by step, decompositions $s = s_1;\ldots;s_{i-1};s^{(i)}$. For $i = 1$, we let $s = s^{(1)}$ trivially. For each $i \geq 1$, if $s^{(i)}$ is the identity, we let $p \eqdef i-1$ and we stop. 

Otherwise, there exists a least $k$ such that $s^{(i)}(k+1) < s^{(i)}(k)$. We let $s_i \eqdef (k \; k+1)$ and $s^{(i+1)} \eqdef \invrs{s_i}; s^{(i)}$. Then $s = s_1;\ldots;s_i;s^{(i+1)}$. 

At each step, the number of pairs $j,j' \in \{1,\ldots,n\}$ such that $j < j'$ but $s^{(i)}(j') < s^{(i)}(j)$ strictly decreases, and it is equal to 0 if and only if $s^{(i)}$ is the identity. It follows that the algorithm terminates after a finite number of steps, producing a decomposition with the desired properties. 

Uniqueness is clear, since the conditions determine the factor $s_i$ uniquely at each step.
\end{proof}

\begin{dfn} \label{dfn:permutation}
Let $s$ be a permutation on the set $\{1,\ldots, n\}$. For all 1\nbd cells $(a_1, \ldots, a_n)$ in a prob $(T, \gen{T})$, we define an invertible 2\nbd cell
\begin{equation*}
	\sigma(s)\colon (a_1, \ldots, a_n) \celto (a_{s(1)}, \ldots, a_{s(n)})
\end{equation*}
in $T$; the dependence of $\sigma(s)$ on $(a_1,\ldots,a_n)$ is left implicit.
\begin{itemize}
	\item If $s$ is the identity, we let $\sigma(s)$ be the unit on $(a_1,\ldots,a_n)$.
	\item If $s$ is an elementary transposition $(k \; k+1)$ of two consecutive elements, we let 
	\begin{equation*}
		\sigma(s) \eqdef a_1 \cp{0} \ldots \cp{0} a_{k-1} \cp{0} \sigma_{a_k,a_{k+1}} \cp{0} a_{k+2} \ldots \cp{0} a_n.
	\end{equation*}
	\item In general, if $s = s_1;\ldots;s_p$ is the decomposition of $s$ given by Lemma \ref{lem:normal_permutation}, we let
	\begin{equation*}
		\sigma(s) \eqdef \sigma(s_1) \cp{1} \ldots \cp{1} \sigma(s_p).
	\end{equation*}
\end{itemize}
We also define a second invertible 2\nbd cell 
\begin{equation*}
	\sigma^*(s)\colon (a_1, \ldots, a_n) \celto (a_{s(1)}, \ldots, a_{s(n)})
\end{equation*}
by $\sigma^*(s) \eqdef \invrs{(\sigma(\invrs{s}))}$.
\end{dfn}

\begin{rmk}
If $(T, \gen{T})$ is a prop, then $\sigma(s) = \sigma^*(s)$ for all permutations $s$. 
\end{rmk}

\begin{exm}
Let $s$ be the permutation $(1,2,3,4,5) \mapsto (3,1,5,4,2)$. The decomposition of $s$ given by Lemma \ref{lem:normal_permutation} is
\begin{equation*}
	s = (2\;3) ; (1\; 2) ; (3\;4) ; (4\;5) ; (3\;4).
\end{equation*}
We use the graphical notation
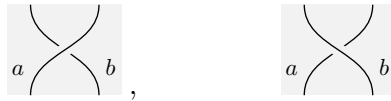
\begin{equation*}
	\input{img/braiding}
\end{equation*}
for the braiding $\sigma_{a,b}$ and the inverse braiding $\sigma^*_{a,b}$, respectively, in a prob. The 2\nbd cells $\sigma(s)$ and $\sigma^*(s)$ of type $(a_1,a_2,a_3,a_4,a_5) \celto (a_3,a_1,a_5,a_4,a_2)$ can be pictured as
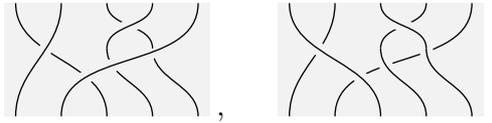
\begin{equation*}
	\input{img/braiding_permutation}
\end{equation*}
respectively. In a prop, these are identical and may both be pictured as their ``shadow''
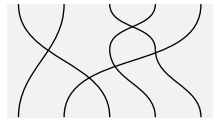
\begin{equation*}
	\input{img/symmetric_permutation}
\end{equation*}
\end{exm}

\begin{dfn}
Let $(T, \gen{T})$ be a prob and let $\{a_{i,j} \mid 1\leq i\leq n, 1 \leq j \leq m\}$ be a doubly indexed collection of 1\nbd cells in $\gen{T}$. We denote by $((a_{i,j})_{i=1}^n)_{j=1}^m$ the 1\nbd cell
\begin{equation} \label{eq:first_ordering}
	(a_{1,1},\ldots,a_{n,1},a_{1,2},\ldots,a_{n,2},\ldots\ldots,a_{1,m},\ldots,a_{n,m})
\end{equation}
and by $((a_{i,j})_{j=1}^m)_{i=1}^n$ the 1\nbd cell
\begin{equation} \label{eq:second_ordering}
	(a_{1,1},\ldots,a_{1,m},a_{2,1},\ldots,a_{2,m},\ldots\ldots,a_{n,1},\ldots,a_{n,m}).
\end{equation}
We let 
\begin{align*}
	\sigma & \colon ((a_{i,j})_{j=1}^m)_{i=1}^n \celto ((a_{i,j})_{i=1}^n)_{j=1}^m, \\
	\sigma^* & \colon ((a_{i,j})_{i=1}^n)_{j=1}^m \celto ((a_{i,j})_{j=1}^m)_{i=1}^n
\end{align*}
be equal to $\sigma(\invrs{s})$ and its inverse $\sigma^*(s)$, respectively, for the permutation $s$ implied by the reordering of (\ref{eq:first_ordering}) into (\ref{eq:second_ordering}).
\end{dfn}

\begin{exm}
The 2\nbd cells
\begin{align*}
	\sigma & \colon (a_{1,1}, a_{1,2}, a_{1,3}, a_{2,1}, a_{2,2}, a_{2,3}) \celto (a_{1,1}, a_{2,1}, a_{1,2}, a_{2,2}, a_{1,3}, a_{2,3}), \\
	\sigma^* & \colon (a_{1,1}, a_{2,1}, a_{1,2}, a_{2,2}, a_{1,3}, a_{2,3}) \celto (a_{1,1}, a_{1,2}, a_{1,3}, a_{2,1}, a_{2,2}, a_{2,3})
\end{align*}
can be pictured as
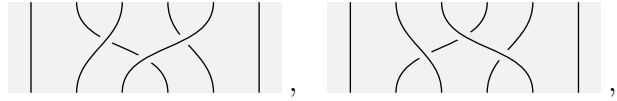
\begin{equation*}
	\input{img/doubly_indexed_permutation}
\end{equation*}
respectively.
\end{exm}

\begin{dfn}[Tensor product of pros] 
The \emph{tensor product} $(T, \gen{T}) \otimes (S, \gen{S})$ of two pros $(T, \gen{T})$ and $(S, \gen{S})$ is the prob $(T \otimes S, \gen{T} \otimes \gen{S})$ constructed as follows.

\begin{enumerate}
	\item Let $(\gen{T} \otimes \gen{S})_0 \eqdef \{\bullet\}$ and $(\gen{T} \otimes \gen{S})_1 \eqdef \{a \otimes c\colon \bullet \celto \bullet \mid a \in \gen{T}_1, c \in \gen{S}_1\}$. This determines $\skel{1}{(T \otimes S)}$ together with its 1\nbd polygraph structure, which makes it a pro.
	
	\item Construct the coproducts
	\begin{equation} \label{eq:indexed_coprod}
		\coprod_{c \in \gen{S}_1} (T, \gen{T}), \quad \quad \coprod_{a \in \gen{T}_1} (S, \gen{S})
	\end{equation}
	in $\pro$. Denote by 
	\begin{equation*}
		- \otimes d\colon (T, \gen{T}) \incl \coprod_{c \in \gen{S}_1} (T, \gen{T}), \quad \quad b \otimes -\colon (S, \gen{S}) \incl \coprod_{a \in \gen{T}_1} (S, \gen{S})
	\end{equation*}
	the inclusions into the $d$\nbd indexed and $b$\nbd indexed summand, respectively. There are morphisms
	\begin{equation*}
		\skel{1}{(T \otimes S)} \to \coprod_{c \in \gen{S}_1} (T, \gen{T}), \quad \quad 		\skel{1}{(T \otimes S)} \to \coprod_{a \in \gen{T}_1} (S, \gen{S})
	\end{equation*}
	uniquely determined by the ``tautologous'' assignments $a \otimes c \mapsto a \otimes c$. Construct the pushout 
	\begin{equation} \label{eq:pre_quotient_tensor}
		\begin{tikzpicture}[baseline={([yshift=-.5ex]current bounding box.center)}]
		\node (0) at (-1,1.5) {$\skel{1}{(T \otimes S)}$};
		\node (1) at (2.5,0) {$T \square S$};
		\node (2) at (-1,0) {$\coprod_{a \in \gen{T}_1} (S, \gen{S})$};
		\node (3) at (2.5,1.5) {$\coprod_{c \in \gen{S}_1} (T, \gen{T})$};
		\draw[1c] (0) to (3);
		\draw[1c] (0) to (2);
		\draw[1c] (2) to (1);
		\draw[1c] (3) to (1);
		\draw[edge] (1.6,0.2) to (1.6,0.7) to (2.3,0.7);
	\end{tikzpicture}
	\end{equation}
	in $\pro$.
	\item Construct the free prob $\fun{F}(T \square S)$ and quotient it by the following equations: for all 2\nbd cells 
	\begin{equation*}
		\varphi\colon (a_1,\ldots,a_n) \celto (b_1,\ldots,b_m) \text{ in $T$}, \; \psi\colon (c_1,\ldots,c_p) \celto (d_1,\ldots,d_q) \text{ in $S$,}
	\end{equation*}
	the 1\nbd composite of
		\begin{align*}
			(a_1 \otimes \psi)\cp{0}\ldots\cp{0}(a_n \otimes \psi)& \colon ((a_i \otimes c_k)_{k=1}^p)_{i=1}^n \celto ((a_i \otimes d_\ell)_{\ell=1}^q)_{i=1}^n, \\
			\sigma & \colon ((a_i \otimes d_\ell)_{\ell=1}^q)_{i=1}^n \celto ((a_i \otimes d_\ell)_{i=1}^n)_{\ell=1}^q, \\ 
			(\varphi \otimes d_1)\cp{0}\ldots\cp{0}(\varphi \otimes d_q)& \colon ((a_i \otimes d_\ell)_{i=1}^n)_{\ell=1}^q \celto ((b_j \otimes d_\ell)_{j=1}^m)_{\ell=1}^q, \\
			\sigma^* & \colon ((b_j \otimes d_\ell)_{j=1}^m)_{\ell=1}^q \celto ((b_j \otimes d_\ell)_{\ell=1}^q)_{j=1}^m
		\end{align*}
	is equal to the 1\nbd composite of 
		\begin{align*}
			\sigma & \colon ((a_i \otimes c_k)_{k=1}^p)_{i=1}^n \celto ((a_i \otimes c_k)_{i=1}^n)_{k=1}^p, \\
			(\varphi \otimes c_1)\cp{0}\ldots\cp{0}(\varphi \otimes c_p)& \colon ((a_i \otimes c_k)_{i=1}^n)_{k=1}^p \celto ((b_j \otimes c_k)_{j=1}^m)_{k=1}^p, \\
			\sigma^* & \colon ((b_j \otimes c_k)_{j=1}^m)_{k=1}^p \celto ((b_j \otimes c_k)_{k=1}^p)_{j=1}^m, \\
			(b_1 \otimes \psi)\cp{0}\ldots\cp{0}(b_m \otimes \psi)& \colon ((b_j \otimes c_k)_{k=1}^p)_{j=1}^m \celto ((b_j \otimes d_\ell)_{\ell=1}^q)_{j=1}^m.
		\end{align*}
	We label this equation $\varphi \otimes \psi$. 
\end{enumerate}
Note that any composite indexed by an empty list must be interpreted as a unit on $\bullet$ of the appropriate dimension.

If $f\colon (T, \gen{T}) \to (T', \gen{T}')$ and $g\colon (S, \gen{S}) \to (S', \gen{S}')$ are morphisms of pros, we can define morphisms
\begin{align*}
	& \coprod_{c \in \gen{S}_1} (T, \gen{T}) \to \fun{U}(T' \otimes S', \gen{T}' \otimes \gen{S'}), && \coprod_{a \in \gen{T}_1} (S, \gen{S}) \to \fun{U}(T' \otimes S', \gen{T}' \otimes \gen{S'}), \\
	& x \otimes c \mapsto f(x) \otimes g(c), && a \otimes y \mapsto f(a) \otimes g(y).
\end{align*}
Taking the transpose morphisms in $\prob$, and using the universal property of the pushout (\ref{eq:pre_quotient_tensor}) which is preserved by $\fun{F}$, we obtain a unique morphism $\fun{F}(T \square S) \to (T' \otimes S', \gen{T}' \otimes \gen{S'})$ of probs which is compatible with the $\varphi \otimes \psi$ equations, hence factors uniquely through a morphism
\begin{equation*}
	f \otimes g\colon (T \otimes S, \gen{T} \otimes \gen{S}) \to (T' \otimes S', \gen{T}' \otimes \gen{S'}).
\end{equation*}
This defines a functor $-\otimes-\colon \pro \times \pro \to \prob$. 
\end{dfn}

\begin{rmk}
When either $\varphi$ or $\psi$ is a unit, the equation $\varphi \otimes \psi$ holds automatically by the axioms of braidings. So $\varphi \otimes \psi$ is only non-trivial when both cells have rank 2.

One can derive, as a consequence, that the monoid $\mathbb{N}$ is a ``relative unit'' for the tensor product, in the sense that the functors $\mathbb{N} \otimes -$ and $- \otimes \mathbb{N}$ are naturally isomorphic to $\fun{F}\colon \pro \to \prob$.
\end{rmk}

\begin{exm} \label{exm:bialgebras}
We compute the tensor product $\theory{Bialg} \eqdef \theory{Mon} \otimes \coo{\theory{Mon}}$ of the theories of monoids and comonoids. Both $\theory{Mon}$ and $\coo{\theory{Mon}}$ are one-sorted, so $\theory{Bialg}$ is also one-sorted. 

In fact, the indexed coproducts (\ref{eq:indexed_coprod}) are equal to $\theory{Mon}$ and $\coo{\theory{Mon}}$, respectively, while $\skel{1}{(\theory{Bialg})}$ is isomorphic to $\mathbb{N}$, so the pushout (\ref{eq:pre_quotient_tensor}) can be computed as
	\begin{equation*} 
		\begin{tikzpicture}[baseline={([yshift=-.5ex]current bounding box.center)}]
		\node (0) at (-.5,1.5) {$\mathbb{N}$};
		\node (1) at (2.5,0) {$\theory{Mon} \uplus \coo{\theory{Mon}}$};
		\node (2) at (-.5,0) {$\coo{\theory{Mon}}$};
		\node (3) at (2.5,1.5) {$\theory{Mon}$};
		\draw[1c] (0) to (3);
		\draw[1c] (0) to (2);
		\draw[1c] (2) to (1);
		\draw[1c] (3) to (1);
		\draw[edge] (1.6,0.2) to (1.6,0.7) to (2.3,0.7);
	\end{tikzpicture}
	\end{equation*}
in $\pro$. The 2\nbd cells in $\theory{Mon} \uplus \coo{\theory{Mon}}$ are freely generated by those of $\theory{Mon}$ and $\coo{\theory{Mon}}$, modulo any equations that hold in the two factors separately: a model of $\theory{Mon} \uplus \coo{\theory{Mon}}$ is a pair of a monoid and a comonoid structure on the same object.

Finally, to obtain $\theory{Bialg}$, we quotient $\fun{F}(\theory{Mon} \uplus \coo{\theory{Mon}})$ by the $\varphi \otimes \psi$ equations. It suffices to let $\varphi$ and $\psi$ range over 2\nbd cells that generate $\theory{Mon}$ and $\coo{\theory{Mon}}$, respectively, under composition. 

An obvious choice is to take the unique maps $\mu\colon (2) \celto (1)$ and $\eta\colon (0) \celto (1)$ as generators of $\theory{Mon}$, and their duals $\delta\colon (1) \celto (2)$ and $\varepsilon\colon (1) \celto (0)$ as generators of $\coo{\theory{Mon}}$. We may picture these as
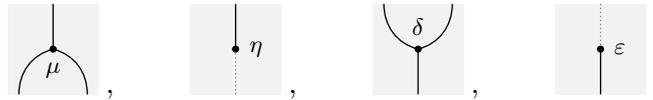
\begin{equation*}
	\input{img/monoid_comonoid}
\end{equation*}
The four corresponding equations are
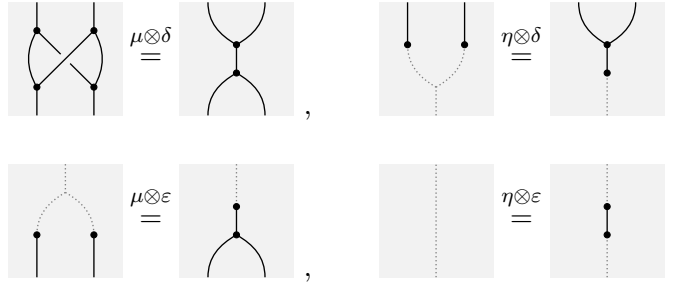
\begin{equation*}
	\input{img/bialgebra}
\end{equation*}
\begin{equation*}
	\input{img/bialgebra2}
\end{equation*}
In a symmetric monoidal category, a pair of a monoid and a comonoid satisfying these equations is a \emph{bialgebra} \cite{pirashvili2002prop}. In a braided monoidal category, this notion forks into two variants, distinguished by the use of braidings or inverse braidings, classified by $\theory{Bialg}$ and by $\theory{Bialg}^* = \coo{\theory{Mon}} \otimes \theory{Mon}$, respectively. 
\end{exm}

\begin{exm} \label{exm:commutative_monoids}
We compute the tensor product $\theory{BrCMon} \eqdef \theory{Mon} \otimes \theory{Mon}$. We proceed as in Example \ref{exm:bialgebras} to derive that $\theory{BrCMon}$ is the quotient of $\fun{F}(\theory{Mon} \uplus \theory{Mon})$ by the equations $\varphi \otimes \psi$ where $\varphi, \psi$ range over $\{\mu, \eta\}$. 

Using different colours to distinguish cells from each copy of $\theory{Mon}$, these can be pictured as
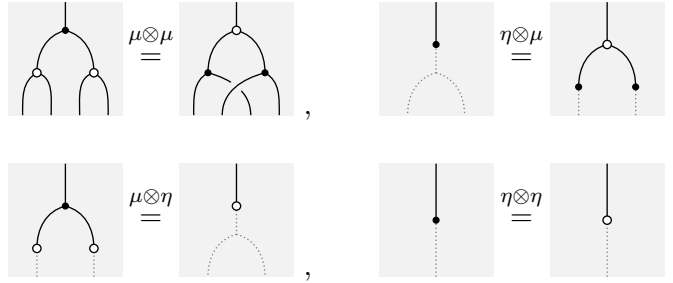
\begin{equation*}
	\input{img/cmonoid}
\end{equation*}
\begin{equation*}
	\input{img/cmonoid2}
\end{equation*}
A model of $\theory{Mon} \uplus \theory{Mon}$ is a pair of monoid structures on the same object. It is a consequence of the Eckmann--Hilton argument, valid in every braided monoidal category, that a pair of monoid structures satisfying the equations $\mu \otimes \mu$ and $\eta \otimes \eta$ coincide with a single commutative monoid structure. The equations $\eta \otimes \mu$ and $\mu \otimes \eta$ are derivable from the rest.

We conclude that $\theory{BrCMon}$ is the \emph{braided} monoidal theory of commutative monoids, whose reflection $\fun{r}(\theory{BrCMon})$ is isomorphic to $\theory{CMon}$. Dually, $\coo{\theory{Mon}} \otimes \coo{\theory{Mon}}$ is the braided monoidal theory of commutative comonoids.
\end{exm}

\begin{dfn}[Tensor product of props]
The \emph{tensor product} $(T,\gen{T}) \tensorp (S,\gen{S})$ of two props $(T, \gen{T})$ and $(S, \gen{S})$ is the quotient of $\fun{r}(\fun{U}(T,\gen{T}) \otimes \fun{U}(S,\gen{S}))$ by the equations
\begin{equation} \label{eq:free_notfree}
		\sigma_{a,b} \otimes c = \sigma_{a \otimes c, b \otimes c}, \quad a \otimes \sigma_{c,d} = \sigma_{a \otimes c, a \otimes d}
\end{equation}
for all $a,b \in \gen{T}_1$ and $c,d \in \gen{S}_1$, where $\sigma_{a,b}$ and $\sigma_{c,d}$ are the original braidings of $T$ and $S$. 

As shown in \cite[Section 3]{hackney2015category}, the tensor product of props is part of a symmetric monoidal closed structure on $\propp$, whose unit is the theory of permutations $\mathbb{S}$. 
\end{dfn}

\begin{exm}
Given a prop $(T, \gen{T})$, the tensor product $(T, \gen{T}) \tensorp \coo{CMon}$ is a \emph{cartesian} prop, also known as a Lawvere theory. It is in fact the free cartesian prop on $(T, \gen{T})$ \cite{baez2006universal}.
\end{exm}

\begin{comm}
The tensor product of props is compatible with the tensor product of pros in the sense that the diagram of functors
	\begin{equation} \label{eq:tensor_pro_prop}
		\begin{tikzpicture}[baseline={([yshift=-.5ex]current bounding box.center)}]
		\node (0) at (-1.5,1.5) {$\pro \times \pro$};
		\node (1) at (2.5,0) {$\propp$};
		\node (2) at (-1.5,0) {$\propp \times \propp$};
		\node (3) at (2.5,1.5) {$\prob$};
		\draw[1c] (0) to node[auto,arlabel] {$\otimes$} (3);
		\draw[1c] (0) to node[auto,swap,arlabel] {$\fun{rF} \times \fun{rF}$} (2);
		\draw[1c] (2) to node[auto,swap,arlabel] {$\tensorp$} (1);
		\draw[1c] (3) to node[auto,arlabel] {$\fun{r}$} (1);
	\end{tikzpicture}
	\end{equation}
commutes up to natural isomorphism. The reason why this works is that, when $\varphi$ or $\psi$ is a braiding $\sigma_{a,b}$, the equation $\varphi \otimes \psi$ combined with (\ref{eq:free_notfree}) holds automatically in a prop. It follows that, while $\fun{UrF}(T, \gen{T}) \otimes \fun{UrF}(S, \gen{S})$ has additional generators and equations compared to $(T, \gen{T}) \otimes (S, \gen{S})$, these are all trivialised by the combined action of $\fun{r}$ and (\ref{eq:free_notfree}). 

This fact is specific to props and does \emph{not} generalise to probs: the quotient of $\fun{UF}(T, \gen{T}) \otimes \fun{UF}(S, \gen{S})$ by (\ref{eq:free_notfree}) in $\prob$ is not in general isomorphic to $(T, \gen{T}) \otimes (S, \gen{S})$. For example, the quotient of $\fun{U}\mathbb{B} \otimes \fun{U}\mathbb{B}$ by (\ref{eq:free_notfree}) is not isomorphic to $\mathbb{N} \otimes \mathbb{N} \simeq \mathbb{B}$. Indeed, if $\sigma_{1,1}\colon (2) \celto (2)$ is a braiding in $\mathbb{B}$, the equation $\sigma_{1,1} \otimes \sigma_{1,1}$ becomes
\begin{equation*}
	\input{img/invalid_braiding}
\end{equation*}
which does not hold in the braid group on 4 strands. This can be checked by considering the link diagrams
\begin{equation*}
	\input{img/unlink}
\end{equation*}
and observing that the first is an unlink while the second is not. On the other hand, the reflected equation
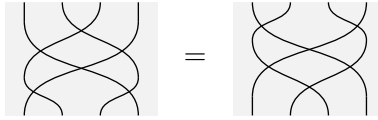
\begin{equation*}
	\input{img/valid_symmetry}
\end{equation*}
is valid in the theory of permutations.

As a consequence, there does not seem to be an interesting monoidal structure on $\prob$ that generalises the one on $\propp$. Following the interpretation of the tensor product as a smash product, we believe that symmetric monoidal theories being closed under the tensor product is a consequence of symmetric monoidal structures being \emph{stable} under smash products in the sense of stable homotopy theory.
\end{comm}

\begin{rmk}
As shown in \cite[Proposition 40]{hackney2015category}, the tensor product of props extends the Boardman-Vogt product of symmetric operads \cite{boardman2006homotopy}, in the sense that there is an embedding of the category of symmetric operads into the category of props which is strong monoidal with respect to the two monoidal structures.
\end{rmk}

\begin{rmk}
The tensor product of pros is \emph{not} symmetric. Up to the definition of $T \square S$ as the pushout (\ref{eq:pre_quotient_tensor}), the construction of $(T, \gen{T}) \otimes (S, \gen{S})$ and of $(S, \gen{S}) \otimes (T, \gen{T})$ is, indeed, identical up to a change of notation. However, in the final quotient, the roles of $\sigma$ and $\sigma^*$, or braidings and inverse braidings, are reversed.

Nevertheless, this argument reveals a natural isomorphism between
\begin{equation*}
	(T, \gen{T}) \otimes (S, \gen{S}) \quad \text{and} \quad ((S, \gen{S}) \otimes (T, \gen{T}))^*,
\end{equation*}
where $-^*$ is the duality defined in \S \ref{dfn:dualbraided}. From this we can recover a symmetry for the tensor product of props.
\end{rmk}


\subsection{The smash product of pointed diagrammatic sets} \label{sec:smash_diag}

\begin{dfn}[Gray product]
Let $P, Q$ be regular directed complexes. The \emph{Gray product} $P \gray Q$ of $P$ and $Q$ is the cartesian product $P \times Q$ of their underlying posets with the following orientation. Write $x \gray y$ for a generic element of $P \gray Q$. For all $x'$ covered by $x$ in $P$ and all $y'$ covered by $y$ in $Q$,
\begin{align*}
	o(x \gray y \to x' \gray y) & \eqdef o_P(x \to x'), \\
	o(x \gray y \to x \gray y') & \eqdef (-)^{\dmn{x}}o_Q(y \to y'),
\end{align*}
where $o_P$ and $o_Q$ are the orientations of $P$ and $Q$, respectively.

As shown in \cite[Section 2.2]{hadzihasanovic2020diagrammatic}, $P \gray Q$ is a regular directed complex. If $f\colon P \to P'$ and $g\colon Q \to Q'$ are maps of regular directed complexes, let $f \gray g\colon P \gray Q \to P' \gray Q'$ have the cartesian product of $f$ and $g$ as underlying function. Then $f \gray g$ is a map of regular directed complexes. 

Gray products determine a monoidal structure on $\rdcpx$ whose unit is the terminal object $1$. 
\end{dfn}

\begin{dfn} \label{dfn:atom_gray_product}
The monoidal structure on $\rdcpx$ restricts to a monoidal structure on $\atom$, which, by Day's theory \cite{day1970closed}, extends along the Yoneda embedding to a monoidal biclosed structure on $\dgmset$. 

Explicitly, let $X$ and $Y$ be diagrammatic sets. The \emph{Gray product} $X \gray Y$ of $X$ and $Y$ is the colimit in $\dgmset$ of the diagram
	\begin{equation} \label{eq:day_diagram}
		\begin{tikzpicture}[baseline={([yshift=-.5ex]current bounding box.center)}]
		\node (0) at (-2,0) {$\slice{\atom}{X} \times \slice{\atom}{Y}$};
		\node (1) at (2,0) {$\atom \times \atom$};
		\node (2) at (4,0) {$\atom$};
		\node (3) at (6,0) {$\dgmset$,};
		\draw[1c] (0) to node[auto,arlabel] {$\mathrm{dom} \times \mathrm{dom}$} (1);
		\draw[1c] (1) to node[auto,arlabel] {$\gray$} (2);
		\draw[1cinc] (2) to (3);
	\end{tikzpicture}
	\end{equation}
where $\slice{\atom}{X}$ is the category whose objects are cells $x\colon U \to X$ and morphisms from $x\colon U \to X$ to $y\colon V \to X$ are commutative triangles
\begin{equation*}
\begin{tikzpicture}[baseline={([yshift=-.5ex]current bounding box.center)}]
	\node (0) at (-1.25,1.25) {$U$};
	\node (1) at (0,0) {$X$};
	\node (2) at (1.25,1.25) {$V$};
	\draw[1c] (0) to node[auto,arlabel] {$f$} (2);
	\draw[1c] (0) to node[auto,swap,arlabel] {$x$} (1);
	\draw[1c] (2) to node[auto,arlabel] {$y$} (1);
	\node at (1.5,0) {,};
\end{tikzpicture}
\end{equation*}
while $\mathrm{dom}$ sends such a triangle to the map $f\colon U \to V$ in $\atom$. 

In particular, for each pair of cells $x\colon U \to X$ and $y\colon V \to Y$, the image of the pair $(x, y)$ through the diagram (\ref{eq:day_diagram}) is $U \gray V$, so we obtain a morphism $U \gray V \to X \gray Y$ to the colimit, that is, a cell of shape $U \gray V$ in $X \gray Y$. This is the cell $x \gray y$ obtained as the Gray product of $x$ and $y$ in $\dgmset$.
\end{dfn}

\begin{rmk}
The dimensions of cells add under the Gray product, that is, if $x$ is an $n$\nbd cell and $y$ is an $m$\nbd cell, then $x \gray y$ is an $(n+m)$\nbd cell.
\end{rmk}

\begin{rmk}
The Gray product is not the cartesian product in $\dgmset$. However, the monoidal unit is the terminal object, which gives us ``projection'' morphisms $X \gray Y \to X$ and $X \gray Y \to Y$. These send a cell $x \gray y$ of shape $U \gray V$ to $p_1;x$ and $p_2;y$, respectively, where $p_1\colon U \gray V \surj U$ and $p_2\colon U \gray V \surj V$ are projections in $\atom$.
\end{rmk}

\begin{comm} \label{comm:string_diagrams_gray}
We use string diagrams to give some intuition about cells $x \gray y$ of shape $U \gray V$ in low dimension; in the pictures, we write $xy$ for $x \gray y$. First of all, if $x$ or $y$ is a 0\nbd cell, then $U \gray V$ is isomorphic to $V$ or $U$, respectively, and $x \gray y$ has the same dimension and shape as $y$ or $x$. 

Let $a\colon x^- \celto x^+$ be a 1\nbd cell in $X$ and $c\colon y^- \celto y^+$ a 1\nbd cell in $Y$. Then $a \gray c$ is a 2\nbd cell of the form
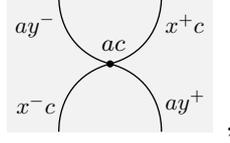
\begin{equation*}
	\input{img/1cube}
\end{equation*}
that is, it is of type $(x^- \gray c)\cp{0}(a \gray y^+) \celto (a \gray y^-) \cp{0} (x^+ \gray c)$ in $X \gray Y$.

Next, let $\varphi\colon a_1 \cp{0} \ldots \cp{0} a_n \celto b_1 \cp{0} \ldots \cp{0} b_m$ be a 2\nbd cell in $X$, and let $\psi\colon c_1 \cp{0} \ldots \cp{0} c_p \celto d_1 \cp{0} \ldots \cp{0} d_q$ be a 2\nbd cell in $Y$. Then $\varphi \gray c$ is a 3\nbd cell of the form
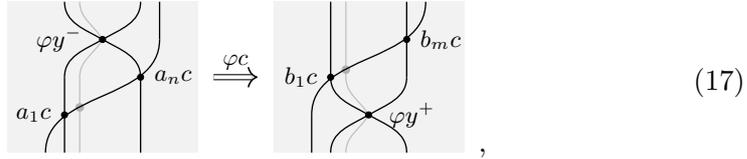
\begin{equation} \label{eq:cylinder_right}
	\input{img/cylinder_right}
\end{equation}
while $a \gray \psi$ is a 3\nbd cell of the form
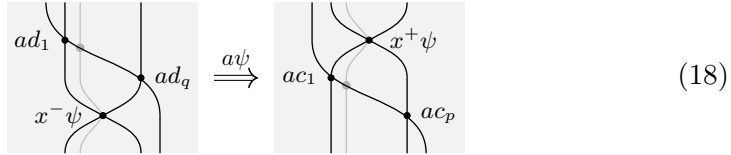
\begin{equation} \label{eq:cylinder_left}
	\input{img/cylinder_left}
\end{equation}
in $X \gray Y$. It is useful to think of these as \emph{sliding moves}: $\varphi \gray c$ slides a 2\nbd cell in the fibre of $\varphi$ left-to-right, top-to-bottom past a 1\nbd cell in the fibre of $c$, while $a \gray \psi$ slides a 2\nbd cell in the fibre of $\psi$ left-to-right, bottom-to-top past a 1\nbd cell in the fibre of $a$.

Next, we consider the 4\nbd cell $\varphi \gray \psi$; to simplify, we depict $\varphi$ and $\psi$ as if they had only 2 inputs and 2 outputs each. Then $\bord{}{-}(\varphi \gray \psi)$ is the 3\nbd diagram 
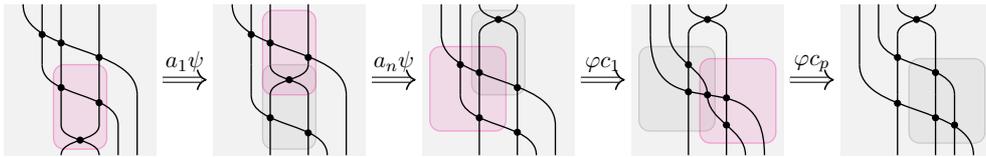
\begin{equation*}
	\input{img/zamolodchikov}
\end{equation*}
where the sequence of sliding moves $a_1 \gray \psi, \ldots, a_n \gray \psi$ is followed by the sequence $\varphi \gray c_1, \ldots, \varphi \gray c_p$, while $\bord{}{+}(\varphi \gray \psi)$ is the 3\nbd diagram
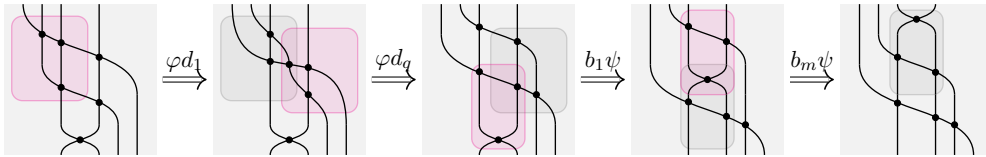
\begin{equation*}
	\input{img/zamolodchikov2}
\end{equation*}
where the sequence of sliding moves $\varphi \gray d_1, \ldots, \varphi \gray d_q$ is followed by the sequence $b_1 \gray \psi, \ldots, b_m \gray \psi$. In the case $n,m,p,q = 2$, one can recognise the two sides of the Zamolodchikov tetrahedron equation \cite{kapranov942categories}.

Next, let $\rho$ be a 3\nbd cell in $X$ and consider the 4\nbd cell $\rho \gray c$. To simplify, we depict $\rho$ as if it were of type $\varphi \celto \varphi'$ where $\varphi$ and $\varphi'$ are both 2\nbd cells. Then $\bord{}{-}(\rho \gray c)$ has the form
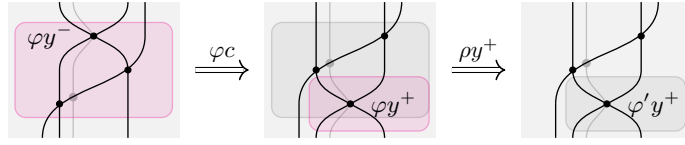
\begin{equation*}
	\input{img/coherence_right}
\end{equation*}
while $\bord{}{+}(\rho \gray c)$ has the form
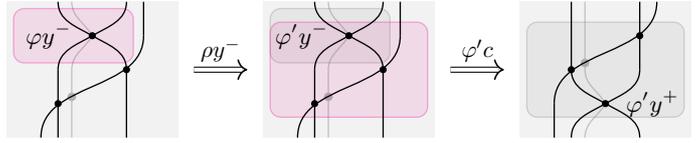
\begin{equation*}
	\input{img/coherence_right2}
\end{equation*}
Dually, if $\tau\colon \psi \celto \psi'$ is a 3\nbd cell in $Y$, $\bord{}{-}(a \gray \tau)$ has the form
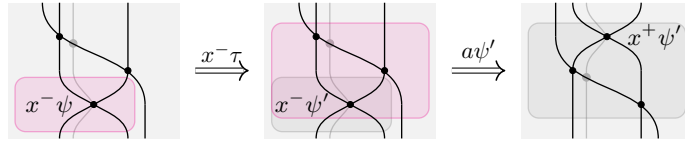
\begin{equation*}
	\input{img/coherence_left}
\end{equation*}
while $\bord{}{+}(a \gray \tau)$ has the form
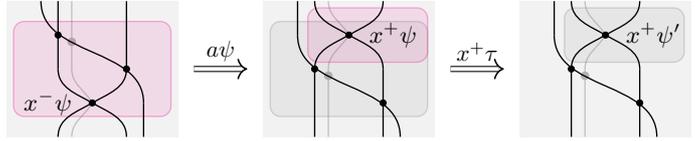
\begin{equation*}
	\input{img/coherence_left2}
\end{equation*}
\end{comm}

\begin{dfn}[Pointed diagrammatic set]
A \emph{pointed diagrammatic set} is a diagrammatic set $X$ together with a distinguished 0\nbd cell $\bullet\colon 1 \to X$, the \emph{basepoint}.

A morphism $f\colon (X, \bullet_X) \to (Y, \bullet_Y)$ of pointed diagrammatic sets is a morphism $f\colon X \to Y$ such that $f(\bullet_X) = \bullet_Y$. With their morphisms, pointed diagrammatic sets form a category $\dgmpoint$.
\end{dfn}

\begin{dfn}
The obvious forgetful functor $\dgmpoint \to \dgmset$ has a left adjoint sending a diagrammatic set $X$ to the coproduct $X + 1$, pointed with the inclusion of $1$ into the coproduct.

The terminal object $1$ of $\dgmset$, pointed with its only 0\nbd cell, is a zero object in $\dgmpoint$, both terminal and initial.
\end{dfn}

\begin{dfn}[Wedge sum]
The \emph{wedge sum} of two pointed diagrammatic sets $(X, \bullet_X)$ and $(Y, \bullet_Y)$ is the pointed diagrammatic set $(X \lor Y, \bullet)$ where
\begin{enumerate}
	\item $X \lor Y$ is the quotient of $X + Y$ by the equation $\bullet_X = \bullet_Y$, and
	\item $\bullet$ is the result of the identification of $\bullet_X$ and $\bullet_Y$. 
\end{enumerate}
\end{dfn}

\begin{dfn}[Smash product]
Let $(X, \bullet_X)$ and $(Y, \bullet_Y)$ be pointed diagrammatic sets. There is an inclusion $X \lor Y \incl X \gray Y$ defined by
\begin{equation*}
	x \mapsto x \gray \bullet_Y, \quad \quad y \mapsto \bullet_X \gray y
\end{equation*}
on cells in $X$ and $Y$, respectively. 

The \emph{smash product} of $(X, \bullet_X)$ and $(Y, \bullet_Y)$ is the pointed diagrammatic set $(X \gsmash Y, \bullet)$ obtained from the pushout diagram
	\begin{equation*}
		\begin{tikzpicture}[baseline={([yshift=-.5ex]current bounding box.center)}]
		\node (0) at (0,1.5) {$X \lor Y$};
		\node (1) at (2.5,0) {$X \gsmash Y$};
		\node (2) at (0,0) {$1$};
		\node (3) at (2.5,1.5) {$X \gray Y$};
		\draw[1cinc] (0) to (3);
		\draw[1c] (0) to (2);
		\draw[1c] (2) to node[auto,arlabel,swap] {$\bullet$} (1);
		\draw[1c] (3) to (1);
		\draw[edge] (1.6,0.2) to (1.6,0.7) to (2.3,0.7);
		\end{tikzpicture}
	\end{equation*}
in $\dgmset$ (the ``quotient of $X \gray Y$ by the subspace $X \lor Y$'').

The smash product is part of a monoidal structure on $\dgmpoint$, whose unit is the diagrammatic set $1 + 1$, pointed with one of the coproduct inclusions, and all structural isomorphisms are derived from those of the Gray product.
\end{dfn}

\begin{comm}
The smash product of pointed diagrammatic sets is a ``directed'' counterpart to the smash product of pointed topological spaces, with the Gray product playing the r\^ole of the cartesian product of spaces. 

The formal correspondence between definitions is made concrete through the geometric realisation of diagrammatic sets \cite[\S 8.38]{hadzihasanovic2020diagrammatic}. This functor $\realis{-}\colon \dgmset \to \cghaus$ sends 0\nbd cells in a diagrammatic set to points in a space, so it lifts to a functor
\begin{equation*}
	\realis{-}\colon \dgmpoint \to \pointed
\end{equation*}
to the category of pointed compactly generated Hausdorff spaces and pointed continuous maps.

We claim that this functor sends smash products in $\dgmpoint$ to smash products in $\pointed$, that is, it is strong monoidal with respect to the two monoidal structures.

\begin{proof}
On regular atoms, $\realis{-}$ is defined as the forgetful functor from $\atom$ to the category of posets and order-preserving maps, followed by the simplicial nerve of posets, followed by the geometric realisation of simplicial sets. The first sends Gray products to cartesian products and the other two preserve finite products. Thus $\realis{U \gray V} \simeq \realis{U} \times \realis{V}$ naturally in $U$ and $V$.

Both Gray products in $\dgmset$ and products in $\cghaus$ are part of a biclosed monoidal structure, so they preserve colimits separately in each variable. Since $\realis{-}$, a left adjoint functor, also preserves colimits, we can extend to an isomorphism $\realis{X \gray Y} \simeq \realis{X} \times \realis{Y}$ natural in the diagrammatic sets $X$ and $Y$.

Finally, $\realis{-}$ also preserves the terminal object, so it sends the colimit diagrams that define wedge sums and smash products in $\dgmpoint$ to the colimit diagrams that define them in $\pointed$.
\end{proof}
\end{comm}

\begin{dfn}
Like the smash product of pointed spaces, the smash product of pointed diagrammatic sets is part of a biclosed structure on $\dgmpoint$.

Left homs and right homs can be computed by a formal argument. If $\rimp{(X,\bullet_X)}{(Y, \bullet_Y)}$ is a right hom in $(\dgmpoint, \gsmash, 1+1)$, cells of shape $U$ in its underlying diagrammatic set correspond to pointed morphisms from $U + 1$ to $\rimp{(X,\bullet_X)}{(Y, \bullet_Y)}$, which correspond to pointed morphisms from $(X,\bullet_X) \gsmash (U + 1)$ to $(Y,\bullet_Y)$. 

Now $X \gsmash (U + 1)$ is isomorphic to the quotient of $X \gray U$ by the subspace $\{\bullet_X\} \gray U$. By the universal property of this quotient, we conclude that there is a bijection between
\begin{enumerate}
	\item cells of shape $U$ in $\rimp{(X,\bullet_X)}{(Y, \bullet_Y)}$ and
	\item morphisms $X \gray U \to Y$ which send $\{\bullet_X\} \gray U$ to $\{\bullet_Y\}$.
\end{enumerate}
Similarly, cells of shape $U$ in the left hom $\limp{(X,\bullet_X)}{(Y, \bullet_Y)}$ correspond bijectively to morphisms $U \gray X \to Y$ sending $U \gray \{\bullet_X\}$ to $\{\bullet_Y\}$. 

In particular, the 0\nbd cells in both the left and the right hom are the pointed morphisms from $(X,\bullet_X)$ to $(Y,\bullet_Y)$. The basepoint, classified by the only morphism from the zero object, is the constant morphism $X \mapsto \bullet_Y$. 
\end{dfn}

\begin{comm} \label{comm:smash_string}
From the string diagram of a cell in $X \gray Y$, it is easy to obtain a picture of the same cell in $X \gsmash Y$: we simply need to identify every cell of the form $x \gray \bullet_Y$ or $\bullet_X \gray y$ and shape $U$ with the cell $!;\bullet$ of the same shape, where $!\colon U \surj 1$ is the unique map to the terminal object. 

We will depict all such 1\nbd cells as dotted wires, and all such 2\nbd cells as dotless nodes, which is consistent with our convention for units and unitors in the nerve of a pro. For example, if $X$ and $Y$ have a single 0\nbd cell, then any 3\nbd cell of the form $\varphi \gray c$ as in (\ref{eq:cylinder_right}) or $a \gray \psi$ as in (\ref{eq:cylinder_left}) becomes
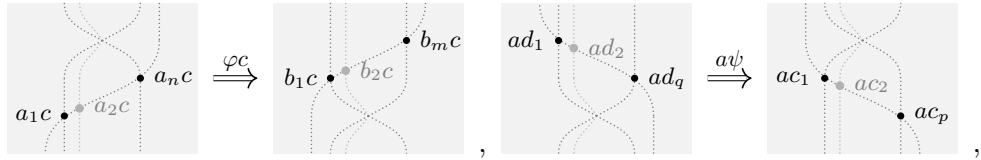
\begin{equation*}
	\input{img/cylinder_smash}
\end{equation*}
respectively, in $X \gsmash Y$.
\end{comm}


\subsection{Comparison of the constructions} \label{sec:comparison}

We are ready to state our main theorem.

\begin{dfn}
Let $(T, \gen{T})$ be a pro. Its diagrammatic nerve $\fun{N}(T,\gen{T})$ has a single 0\nbd cell, so it is canonically pointed, and every morphism in the image of $\fun{N}$ trivially preserves the basepoint. Thus $\fun{N}$, restricted to $\pro$, lifts uniquely to a functor $\fun{N}\colon \pro \to \dgmpoint$. 
\end{dfn}

\begin{thm} \label{thm:mainthm}
The diagram of functors
	\begin{equation*} 
		\begin{tikzpicture}[baseline={([yshift=-.5ex]current bounding box.center)}]
		\node (0) at (-4,1) {$\pro \times \pro$};
		\node (1) at (0,1) {$\prob$};
		\node (1b) at (-4,-1) {$\dgmpoint \times \dgmpoint$};
		\node (2b) at (0,-1) {$\dgmpoint$};
		\node (3) at (3,0) {$\graycat$};
		\draw[1c] (0) to node[auto,swap,arlabel] {$\fun{N} \times \fun{N}$} (1b);
		\draw[1c] (0) to node[auto,arlabel] {$-\otimes-$} (1);
		\draw[1c] (1b) to node[auto,swap,arlabel] {$- \gsmash \oppall{(-)}$} (2b);
		\draw[1c] (1) to node[auto,arlabel] {$\fun{U}_3$} (3);
		\draw[1c] (2b) to node[auto,swap,arlabel] {$\fun{G}$} (3);
		\end{tikzpicture}
	\end{equation*}
commutes up to natural isomorphism.
\end{thm}

\begin{comm}
The form of Theorem \ref{thm:mainthm} does not suggest, at first sight, that the smash product of pointed diagrammatic sets \emph{subsumes} and \emph{generalises} the tensor product of props. Nevertheless, we argue that this is essentially the case.

First of all, since $\fun{U}_3$ is pseudomonic by Remark \ref{rmk:pseudomonic}, if $\fun{G}X$ is isomorphic to $\fun{U}_3(T, \gen{T})$ for some prob $(T, \gen{T})$, then this prob is essentially unique. It follows that, on the image of $\fun{N}- \gsmash \oppall{(\fun{N}-)}$, we can lift $\fun{G}$ to a functor with codomain $\prob$, and compute the tensor product of two pros through the lower leg of the diagram. In this sense, the smash product on $\dgmpoint$ strictly subsumes the ``external'' tensor product of pros. 

From the tensor product of pros, we can recover the tensor product of props via a universal characterisation in $\propp$, independent of the specific construction. If $(T, \gen{T})$ and $(S, \gen{S})$ are two props, we have families of morphisms 
\begin{align*}
	\idd{T} \otimes c\colon  &\fun{FU}(T, \gen{T}) \to \fun{U}(T,\gen{T}) \otimes \fun{U}(S, \gen{S}), \\
	a \otimes \idd{S}\colon  &\fun{FU}(S, \gen{S}) \to \fun{U}(T,\gen{T}) \otimes \fun{U}(S, \gen{S})
\end{align*}
in $\prob$ indexed by $c \in \gen{S}_1$ and $a \in \gen{T}_1$, where $c\colon \mathbb{N} \to \fun{U}(S, \gen{S})$ and $a\colon \mathbb{N} \to \fun{U}(T, \gen{T})$ send the generating 1\nbd cell of $\mathbb{N}$ to $c$ and $a$, respectively; here we use the fact that $- \otimes \mathbb{N}$ and $\mathbb{N} \otimes -$ are naturally isomorphic to $\fun{F}$. Then $(T, \gen{T}) \tensorp (S, \gen{S})$ is the pushout
\begin{equation*} 
\begin{tikzpicture}[baseline={([yshift=-.5ex]current bounding box.center)}]
	\node (0) at (-3.5,1.5) {$\coprod_{c \in \gen{S}_1} \fun{rFU}(T, \gen{T}) + \coprod_{a \in \gen{T}_1} \fun{rFU}(S, \gen{S})$};
	\node (1) at (2.5,0) {$(T, \gen{T}) \tensorp (S, \gen{S})$};
	\node (2) at (-3.5,0) {$\coprod_{c \in \gen{S}_1} (T, \gen{T}) + \coprod_{a \in \gen{T}_1} (S, \gen{S})$};
	\node (3) at (2.5,1.5) {$\fun{r}(\fun{U}(T,\gen{T}) \otimes \fun{U}(S, \gen{S}))$};
	\draw[1c] (0) to (3);
	\draw[1c] (0) to (2);
	\draw[1c] (2) to (1);
	\draw[1c] (3) to (1);
	\draw[edge] (1.6,0.2) to (1.6,0.7) to (2.3,0.7);
\end{tikzpicture}
\end{equation*}
in $\propp$, where the top leg is obtained universally from the family of morphisms $\{\fun{r}(\idd{T} \otimes c), \fun{r}(a \otimes \idd{S}) \mid c \in \gen{S}_1, a \in \gen{T}_1\}$ and the left leg from the counit of the adjunction between $\fun{rF}$ and $\fun{U}$. 
\end{comm}

\begin{dfn}[Skeleta of diagrammatic sets] \label{dfn:skeleta}
For each $n \in \mathbb{N} + \{-1\}$, the restriction functor $\dgmset \to \psh{}{\atom_n}$ has a left adjoint; let $\skel{n}{}$ be the comonad induced by this adjunction. 

The \emph{$n$\nbd skeleton} of a diagrammatic set $X$ is the counit $\skel{n}{X} \to X$. For all $k \leq n$, the $k$\nbd skeleton factors uniquely through the $n$\nbd skeleton of $X$. By a standard argument $X$ is the colimit of the sequence of its skeleta.
\end{dfn}

\begin{proof}[Proof of Theorem \ref{thm:mainthm}] 
Let $(T, \gen{T})$ and $(S, \gen{S})$ be two pros and let $(X,\bullet_X)$ and $(Y,\bullet_Y)$ be equal to $\fun{N}(T,\gen{T})$ and $\oppall{(\fun{N}(S,\gen{S}))}$, respectively. 

As seen in \S \ref{dfn:atom_gray_product}, $X \gray Y$ is the colimit of a diagram of atoms $U \gray V$ indexed by pairs of cells $x\colon U \to X$ and $y \colon V \to Y$, which are transposes of morphisms $x\colon \fun{P}U \to (T,\gen{T})$ and $y\colon \fun{P}\oppall{V} \to (S, \gen{S})$ in $\pro$. The smash product $X \gsmash Y$ is then the colimit of this diagram extended with a morphism $U \gray V \surj 1$ for all atoms $U \gray V$ indexed by $(x,\bullet_Y)$ or by $(\bullet_X, y)$. 

Each diagrammatic set is the colimit of the sequence of its skeleta, and this colimit is preserved by smash products separately in each variable. Because colimits commute with colimits, we can compute $X \gsmash Y$ in steps, increasing $i$ and $j$ separately in $\skel{i}{X} \gsmash \skel{j}{Y}$. This corresponds to restricting the indexing category to pairs $(x,y)$ with $\dmn{x} \leq i$ and $\dmn{y} \leq j$. 

The functor $\fun{G}\colon \dgmset \to \graycat$, a left adjoint, preserves colimits, and we know how to explicitly compute $\fun{G}$ on atoms. We will use this to compute $\fun{G}(X \gsmash Y)$. 

\begin{itemize}
	\item Let $i = 0$ or $j = 0$. Since the only 0\nbd cell in $X$ and $Y$ is their basepoint, both $\skel{0}{X} \gsmash Y$ and $X \gsmash \skel{0}{Y}$ are isomorphic to the terminal diagrammatic set. Their image through $\fun{G}$ is the terminal Gray\nbd category with one 0\nbd cell and no cells of higher rank.
	\item Let $i = j = 1$. The 1\nbd cells in $X$ that do not factor through $\skel{0}{X}$ correspond bijectively to generating 1\nbd cells $a \in \gen{T}_1$, and the 1\nbd cells in $Y$ that do not factor through $\skel{0}{Y}$ to generating 1\nbd cells $c \in \gen{S}_1$. \\
	The boundary of $a \gray c$ contains only 1\nbd cells of the form $!;\bullet$, which $\fun{G}$ sends to units on $\bullet$. Through $\fun{G}$, then, $a \gray c$ becomes a 2\nbd cell of type $\eps{}{\bullet} \celto \eps{}{\bullet}$. Thus $\fun{G}(\skel{1}{X} \gsmash \skel{1}{Y})$ has a single 0\nbd cell, a single 1\nbd cell, and its 2\nbd cells are freely generated by the $a \gray c$: this makes it a prob in the sense of \S \ref{dfn:alternative_prob}, isomorphic to $\fun{F}(\skel{1}(T \otimes S))$. This structure of prob will be inherited by $\fun{G}(\skel{i}{X} \gsmash \skel{j}{Y})$ for all higher $i, j$. 
	\item We fix $j = 1$ and increase $i$; observe that we can stop at $i = 3$, since for $i > 3$ we only include cells of dimension $> 4$, whose contribution through $\fun{G}$ is trivial. \\
	Let $c \in \gen{S}_1$. Each 2\nbd cell $\varphi\colon (a_1,\ldots,a_n) \celto (b_1,\ldots,b_m)$ in $T$ contributes a 3\nbd cell $\varphi \gray c$ in $\skel{2}{X} \gsmash \skel{1}{Y}$; any other 2\nbd cell in $X$ factors through one of this form, so it does not give a contribution. We can read the form of $\varphi \gray c$ from Comment \ref{comm:smash_string}: unravelling the definition of $\fun{G}$ on 3\nbd atoms, we see that it sends $\varphi \gray c$ to a 3\nbd cell of type
	\begin{equation*}
		(a_1 \gray c) \cp{1} \ldots \cp{1} (a_n \gray c) \celto (b_1 \gray c) \cp{1} \ldots \cp{1} (b_m \gray c).
	\end{equation*}
	Extending along colimits, a 2\nbd diagram $x \cp{k} y$ in $X$ induces a diagram $(x \gray c) \cp{k+1} (y \gray c)$ in $\fun{G}(\skel{2}{X} \gsmash \skel{1}{Y})$ for each $k \in \{0,1\}$. \\
	From Proposition \ref{prop:full_faithful}, we know that 3\nbd cells $\rho$ in $X$ exhibit all and only the equations of diagrams $x = y$ that hold in $T$. Reading the form of $\rho \gray c$ from Comment \ref{comm:string_diagrams_gray}, we see that the only part surviving both the smash product quotient and $\fun{G}$ is an equation between the composites of $x \gray c$ and $y \gray c$. For each $c \in \gen{S}_1$, then, we can define a morphism of probs
	\begin{equation*}
		- \gray c\colon \fun{F}(T, \gen{T}) \to \fun{G}(X \gsmash \skel{1}{Y})
	\end{equation*}
	by $x \mapsto x \gray c$ on cells in $T$, extending universally to the free prob, and prove that it is injective. Moreover, the family of the $- \gray c$ is jointly surjective and only overlaps on $\bullet$. We conclude that there is an isomorphism between $\fun{G}(X \gsmash \skel{1}Y)$ and
	\begin{equation*}
		\textstyle \coprod_{c \in \gen{S}_1} \fun{F}(T, \gen{T}) \simeq \fun{F}(\coprod_{c \in \gen{S}_1} (T, \gen{T})),
	\end{equation*}
	the coproducts being in $\prob$ and $\pro$, respectively.
	\item The case where we fix $i = 1$ and increase $j$ is dual, with a subtlety due to the way Gray products change orientations in their second factor depending on the dimension of the first factor. Since we defined $Y$ to be the \emph{dual} of $\fun{N}(S, \gen{S})$, each 2\nbd cell $\psi\colon (c_1, \ldots, c_p) \celto (d_1, \ldots, d_q)$ corresponds to a 2\nbd cell
	\begin{equation*}
		\psi\colon d_q \cp{0} \ldots \cp{0} d_1 \celto c_p \cp{0} \ldots \cp{0} c_1
	\end{equation*}
	in $Y$, which for each $a \in \gen{T}_1$ contributes a 3\nbd cell $a \gray \psi$ in $\skel{1}{X} \gsmash \skel{2}{Y}$. By inspection of the shape of $a \gray \psi$ in Comment \ref{comm:smash_string}, we see that $\fun{G}$ sends it to a 3\nbd cell of type
	\begin{equation*}
		(a \gray c_1) \cp{1} \ldots \cp{1} (a \gray c_p) \celto (a \gray d_1) \cp{1} \ldots \cp{1} (a \gray d_q),
	\end{equation*}
	which matches the original orientation of $\psi$ in $S$. Proceeding as before, then, we construct an isomorphism between $\fun{G}(\skel{1}{X} \gsmash Y)$ and
	\begin{equation*}
		\textstyle\coprod_{a \in \gen{T}_1} \fun{F}(S, \gen{S}) \simeq \fun{F}(\coprod_{a \in \gen{T}_1} (S, \gen{S})).
	\end{equation*}
	It follows that $\fun{G}((X \gsmash \skel{1}{Y}) \cup (\skel{1}{X} \gsmash Y))$ can be computed as the pushout
	\begin{equation*}
		\begin{tikzpicture}[baseline={([yshift=-.5ex]current bounding box.center)}]
		\node (0) at (-2.5,1.5) {$\fun{F}(\skel{1}{(T \otimes S)})$};
		\node (1) at (2.5,0) {$\fun{G}((X \gsmash \skel{1}{Y}) \cup (\skel{1}{X} \gsmash Y))$};
		\node (2) at (-2.5,0) {$\fun{F}(\coprod_{a \in \gen{T}_1} (S, \gen{S}))$};
		\node (3) at (2.5,1.5) {$\fun{F}(\coprod_{c \in \gen{S}_1} (T, \gen{T}))$};
		\draw[1c] (0) to (3);
		\draw[1c] (0) to (2);
		\draw[1c] (2) to (1);
		\draw[1c] (3) to (1);
		\draw[edge] (1.6,0.2) to (1.6,0.7) to (2.3,0.7);
	\end{tikzpicture}
	\end{equation*}
	in $\prob$, isomorphic to $\fun{F}(T \square S)$.
	\item Finally, let $i = j = 2$; increasing either $i$ or $j$ beyond 2 only includes cells of dimension $> 4$, whose contribution is trivial. \\
	Each pair of a 2\nbd cell $\varphi\colon (a_1,\ldots,a_n) \celto (b_1,\ldots,b_m)$ in $T$ and a 2\nbd cell $\psi\colon (c_1,\ldots,c_p) \celto (d_1,\ldots,d_q)$ in $S$ contributes a 4\nbd cell $\varphi \gray \psi$ to $X \gsmash Y$, and any other pair factors through one of this form.  \\
	Remember that the orientation of $\psi$ is reversed in $Y$. Reading the form of $\varphi \gray \psi$ from Comment \ref{comm:string_diagrams_gray} and unravelling the definition of $\fun{G}$ on 4\nbd atoms, we find that the boundaries of $\varphi \gray \psi$ are mapped by $\fun{G}$ to the two sides of the $\varphi \gray \psi$ equation in $\fun{F}(T \square S)$.
\end{itemize}
We conclude that $\fun{G}(X \gsmash Y)$ with its unique prob structure is isomorphic to $(T,\gen{T}) \otimes (S, \gen{S})$. It is straightforward to check that the isomorphism is natural in $(T, \gen{T})$ and $(S, \gen{S})$. 
\end{proof}

\begin{exm} 
We compute an equation of the prob $\theory{Bialg} = \theory{Mon} \otimes \coo{\theory{Mon}}$ through $\dgmpoint$, to illustrate how it arises from a 4\nbd cell in the smash product of $X \eqdef \fun{N}(\theory{Mon})$ and $Y \eqdef \oppall{\fun{N}(\coo{\theory{Mon}})}$.

The two generating 1\nbd cells of $\theory{Mon}$ and $\coo{\theory{Mon}}$ produce a 2\nbd cell $1 \gray 1$ in $X \gsmash Y$. From the generator $\mu\colon (2) \celto (1)$ of $\theory{Mon}$, we obtain a 3\nbd cell $\mu \gray 1$ in $X \gsmash Y$. Because $1 \gray 1$ is the only non-degenerate 2\nbd cell appearing in the boundary of $\mu \gray 1$, we may informally picture this 3\nbd cell as a string diagram in 3\nbd dimensional space, ``tracing the history'' of the various copies of $1 \gray 1$:
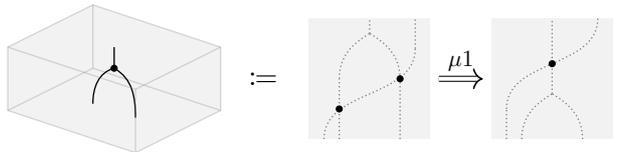
\begin{equation*}
	\input{img/suspension1}
\end{equation*}
Similarly, from the generator $\delta\colon (1) \celto (2)$ of $\coo{\theory{Mon}}$, we obtain a 3\nbd cell $1 \gray \delta$ of the form
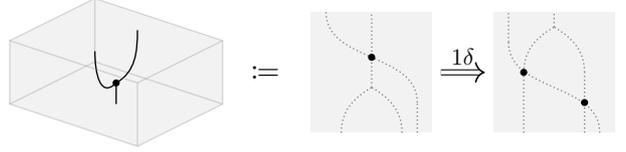
\begin{equation*}
	\input{img/suspension3}
\end{equation*}
Now $\bord{}{-}(\mu \gray \delta)$ has the form 
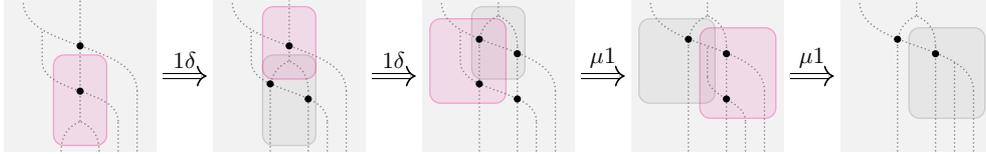
\begin{equation*}
	\input{img/bialgebra2d_input}
\end{equation*}
while $\bord{}{+}(\mu \gray \delta)$ has the form
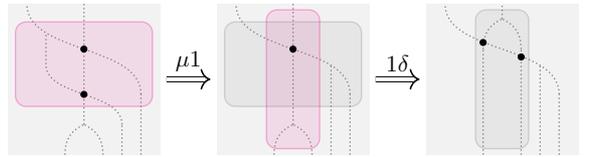
\begin{equation*}
	\input{img/bialgebra2d_output}
\end{equation*}
In 3\nbd dimensional string diagrams, we may picture $\mu \gray \delta$ as
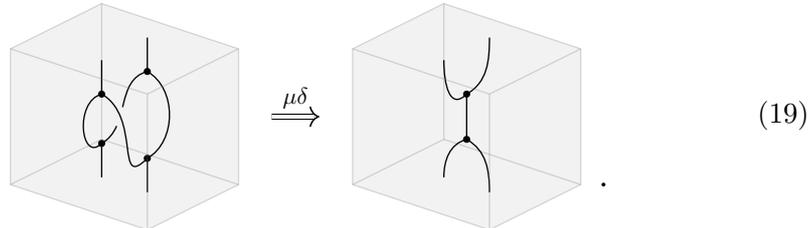
\begin{equation} \label{eq:bialgebra3d}
	\input{img/bialgebra3d}
\end{equation}
If we also ``trace the history'' of the dotted wires to produce \emph{surface diagrams} in the style of \cite{vicary2019coherence}, we recover, up to a deformation, the ``intersecting surfaces'' picture (\ref{eq:sliding_surfaces}). 

The picture of the equation $\mu \gray \delta$ in Example \ref{exm:bialgebras} can be interpreted as a planar projection of (\ref{eq:bialgebra3d}). Technically, the single instance of a braiding in this equation arises, by definition of $\fun{G}$, from the fact that the input 2\nbd cells of the first instance of $\mu \gray 1$ in $\bord{}{-}(\mu \gray \delta)$ are not consecutive in the normal 1\nbd order on the overall 2\nbd diagram. 
\end{exm}

\section{Higher-dimensional cells} \label{sec:higher}

Having established that the smash product of pointed diagrammatic sets generalises the tensor product of pros, we briefly explore the potential of this generalisation in higher-dimensional universal algebra and rewriting.

\begin{dfn}[Diagrammatic complex]
A \emph{diagrammatic complex} is a diagrammatic set $X$ together with a set $\gen{X} = \sum_{n\in \mathbb{N}} \gen{X}_n$ of \emph{generating} cells such that, for all $n \in \mathbb{N}$,  
\begin{equation} \label{eq:diag_complex}
\begin{tikzpicture}[baseline={([yshift=-.5ex]current bounding box.center)}]
	\node (0) at (-1.5,1.5) {$\coprod_{x \in \gen{X}_n} \bord U(x)$};
	\node (1) at (2.5,0) {$\skel{n}{X}$};
	\node (2) at (-1.5,0) {$\skel{n-1}{X}$};
	\node (3) at (2.5,1.5) {$\coprod_{x \in \gen{X}_n} U(x)$};
	\draw[1cinc] (0) to (3);
	\draw[1c] (0) to (2);
	\draw[1cinc] (2) to (1);
	\draw[1c] (3) to node[auto,arlabel] {$(x)_{x \in \gen{X}_n}$} (1);
	\draw[edge] (1.6,0.2) to (1.6,0.7) to (2.3,0.7);
\end{tikzpicture}
\end{equation}
is a pushout in $\dgmset$. Here $U(x)$ denotes the shape of the cell $x$.
\end{dfn}

\begin{rmk}
It follows from the results of \cite[Section 8.3]{hadzihasanovic2020diagrammatic} that the geometric realisation of diagrammatic sets sends a diagrammatic complex $(X, \gen{X})$ to a CW complex with one cell for each generating cell in $\gen{X}$.
\end{rmk}

\begin{prop}
Let $(X, \gen{X})$ and $(Y, \gen{Y})$ be diagrammatic complexes. Then $X \gray Y$ is a diagrammatic complex with 
\begin{equation*}
	(\gen{X} \gray \gen{Y})_n \eqdef \sum_{k=0}^n \{x \gray y \mid x \in \gen{X}_k, y \in \gen{Y}_{n-k} \}.
\end{equation*}
\end{prop}
\begin{proof}
Essentially the same as \cite[Theorem 1.35]{hadzihasanovic2017algebra}, replacing ``polygraph'' with ``diagrammatic complex'' and ``globe'' with ``regular atom''.
\end{proof}

\begin{rmk}
A straightforward consequence: the smash product $X \gsmash Y$ of pointed diagrammatic complexes $(X, \gen{X}, \bullet_X)$ and $(Y, \gen{Y}, \bullet_Y)$ is a pointed diagrammatic complex whose generating cells are $\bullet$ and the pairs $x \gray y$ with $x \in \gen{X} \setminus \{\bullet_X\}$ and $y \in \gen{Y} \setminus \{\bullet_Y\}$. 
\end{rmk}

\begin{dfn}[Diagrammatic presentation]
Let $(T, \gen{T})$ be a bicoloured pro. A \emph{presentation} of $(T, \gen{T})$ is a diagrammatic complex $(X, \gen{X})$ such that $\fun{P}X$ is isomorphic to $(T, \gen{T})$. 

Similarly let $(T, \gen{T})$ be a prob. A \emph{presentation} of $(T, \gen{T})$ is a diagrammatic complex $(X, \gen{X})$ such that $\fun{G}X$ is isomorphic to $\fun{U}_3(T, \gen{T})$. 
\end{dfn}

\begin{exm} \label{exm:monoid_presentation}
The pro of monoids $\theory{Mon}$ admits the following presentation $(X, \gen{X})$. To begin, $\gen{X}_0$ contains a single 0\nbd cell $\bullet$ and $\gen{X}_1$ a single 1\nbd cell $1\colon \bullet \celto \bullet$. Next, $\gen{X}_2$ contains a 2\nbd cell $\mu\colon 1\cp{0}1 \celto 1$ and a 2\nbd cell $\eta\colon \eps{}\bullet \celto 1$, which we picture as
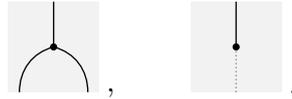
\begin{equation*}
	\input{img/monoid_2gen}
\end{equation*}
Finally, $\gen{X}_3$ contains 3\nbd cells $\alpha, \lambda, \rho$ of the form
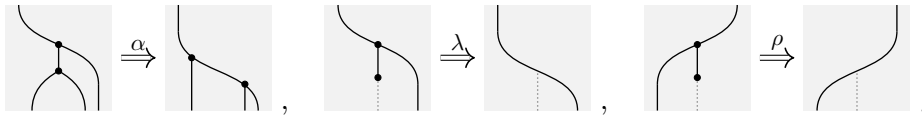
\begin{equation*}
	\input{img/monoid_3gen}
\end{equation*}
\end{exm}

\begin{dfn}
Unless a pro $(T, \gen{T})$ embodies a free monoidal theory, a presentation $(X, \gen{X})$ contains some generating 3\nbd cells, exhibiting equations in $\fun{P}X$. Higher-dimensional generators have no effect on the presented pro, as $\fun{P}$ turns them into trivial ``equations of equations''.

For presentations of a \emph{prob}, the same statement applies shifted by one dimension: generating 4\nbd cells exhibit equations, while higher-dimensional generators are trivialised by $\fun{G}$. 

For the purpose of computing a tensor product of pros, we can replace the nerves of pros $(T, \gen{T})$, $(S, \gen{S})$ with any pair of presentations $(X, \gen{X})$, $(Y, \gen{Y})$, pointed with their unique 0\nbd cell, to obtain a presentation $X \gsmash \oppall{Y}$ of the prob $(T, \gen{T}) \otimes (S, \gen{S})$. 

Even if $(X, \gen{X})$ and $(Y, \gen{Y})$ contain no generating cells in dimension higher than 3, $X \gsmash \oppall{Y}$ contains generating cells up to dimension 6, two more than the threshold of significance for $\fun{G}$. Thus the smash product $X \gsmash \oppall{Y}$ contains strictly more data than the tensor product $(T, \gen{T}) \otimes (S, \gen{S})$. 

We suggest that these data can be interpreted through the lens of higher\nbd dimensional rewriting, and in particular the concepts of \emph{syzygies} and \emph{coherence}; we refer to Yves Guiraud's \emph{th\`ese d'habilitation} \cite{guiraud2019rewriting} for an introduction. 

Rewriting theory is concerned with computational properties of presentations, in particular the properties of confluence and termination. When a presentation is embodied by a polygraph, confluence at a critical branching is exhibited by a pair of parallel cells. In a coherent presentation, this is strengthened to the requirement that the parallel pair be filled by a higher\nbd dimensional cell, sometimes called a \emph{syzygy} \cite{loday2000homotopical}. This requirement can be extended by asking that higher\nbd dimensional parallel pairs also have fillers.

As the following example shows, it appears that the higher\nbd dimensional cells produced by the smash product of two presentations of pros are syzygies for the presentation of their tensor product.
\end{dfn}

\begin{exm}
Let $(X, \gen{X})$ be the presentation of $\theory{Mon}$ from Example \ref{exm:monoid_presentation}. Then $(\oppall{X}, \oppall{\gen{X}})$ is a presentation of $\coo{\theory{Mon}}$, so the smash product $X \gsmash X$ is a presentation of the prob $\theory{Bialg}$ of bialgebras.

Let us compute this smash product. To simplify, we employ the following abuse of notation: we represent a 3\nbd diagram $x$ in $X \gsmash X$ as a 2\nbd diagram in $\theory{Bialg}$ whose image through $\fun{U}_3$ has the same composite as $\fun{G}x$. This allows us to depict $n$\nbd cells in $X \gsmash X$ as if they were $(n-1)$\nbd cells. This is not a faithful representation: most notably, a subdiagram of a 3\nbd diagram is different from a subdiagram of its representation as a 2\nbd diagram. A 4\nbd diagram may look here like a ``3\nbd diagram modulo the axioms of braidings''. 

To begin, $X \gsmash X$ has a single generating 0\nbd cell and no generating 1\nbd cells. The only generating 2\nbd cell is $1 \gray 1$. 

The generating 3\nbd cells are $\mu \gray 1$, $\eta \gray 1$, $1 \gray \mu$, and $1 \gray \eta$. These are the standard generators of $\theory{Bialg}$ as in Example \ref{exm:bialgebras}, and with our abuse of notation, we use the same depiction:
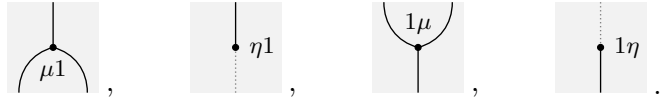
\begin{equation*}
	\input{img/presentation_smash}
\end{equation*}

There are 10 generating 4\nbd cells which can be subdivided into three groups. Those of the form $x \gray 1$ for $x \in \gen{X}_3$ have the same representation as $x$, that is,
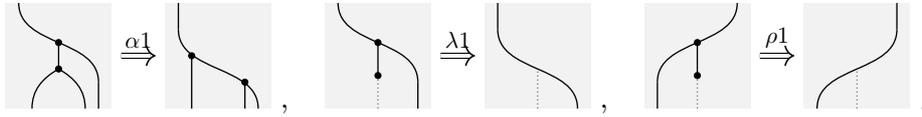
\begin{equation*}
	\input{img/presentation_smash2}
\end{equation*}
Those of the form $1 \gray x$ for $x \in \gen{X}_3$ have the same representation as $\oppall{x}$:
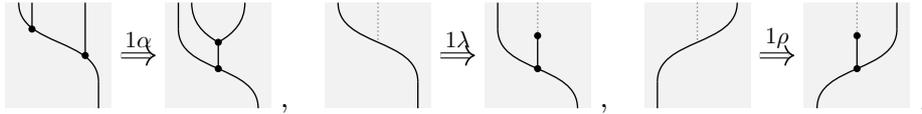
\begin{equation*}
	\input{img/presentation_smash3}
\end{equation*}
Finally, those of the form $x \gray y$ for $x, y \in \gen{X}_2$ present the additional equations of Example \ref{exm:bialgebras} with the orientation
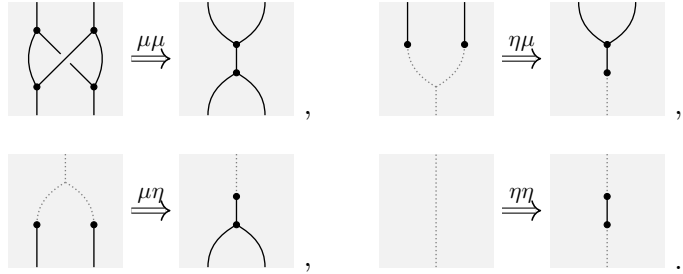
\begin{equation*}
	\input{img/presentation_smash4}
\end{equation*}
\begin{equation*}
	\input{img/presentation_smash5}
\end{equation*}

This presentation of $\theory{Bialg}$ contains new critical branchings that do not correspond to critical branchings in the presentations of $\theory{Mon}$ or $\coo{\theory{Mon}}$. For example, we have the following critical branching involving $\alpha \gray 1$ and $\mu \gray \mu$:
\begin{equation} \label{eq:critical1}
	\input{img/critical1}
\end{equation}
There are 12 generating 5\nbd cells of $X \gsmash X$, of the form $x \gray y$ where either $x \in \gen{X}_3$ and $y \in \gen{X}_2$ or $x \in \gen{X}_2$ and $y \in \gen{X}_3$. We observe that these are syzygies exhibiting confluence at these critical branchings. 

For example, $\bord{}{-}(\alpha \gray \mu)$ is
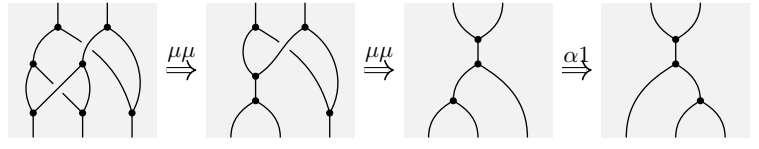
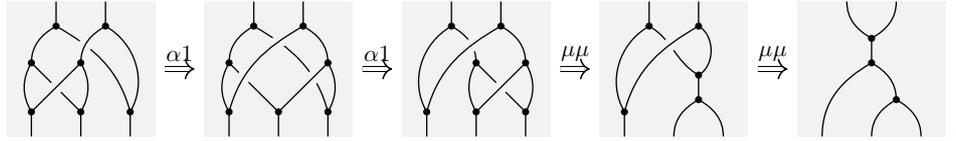
\begin{equation*}
	\input{img/syzygy1}
\end{equation*}
while $\bord{}{+}(\alpha \gray \mu)$ is
\begin{equation*}
	\input{img/syzygy2}
\end{equation*}
which exhibits confluence at the critical branching (\ref{eq:critical1}). 

As another example, $\bord{}{-}(\eta \gray \lambda)$ is
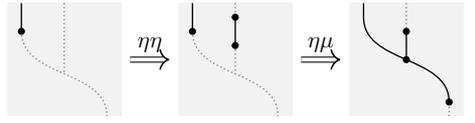
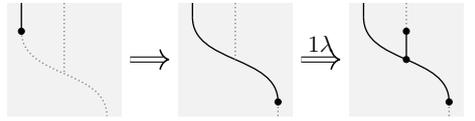
\begin{equation*}
	\input{img/syzygy3}
\end{equation*}
while $\bord{}{+}(\eta \gray \lambda)$ is
\begin{equation*}
	\input{img/syzygy4}
\end{equation*}
Here the unlabelled 4\nbd cell is a degenerate 4\nbd cell \cite[\S 4.14]{hadzihasanovic2020diagrammatic} of the form $\eta \gray x$ where $x$ is a unitor 2\nbd cell. It is mapped by $\fun{G}$ to a diagram in $\theory{Bialg}$ whose image is the unit $\eps{}(\eta \gray 1)$, so we may want to treat $\bord{}{+}(\eta \gray \lambda)$ as a single rewrite step. Thus $\eta \gray \lambda$ exhibits confluence at a critical branching involving $\eta \gray \eta$ and $1 \gray \lambda$. 

These syzygies are oriented, so they can be interpreted as higher\nbd dimensional rewrites creating critical branchings one dimension up. The 9 generating 6\nbd cells of $X \gsmash X$, of the form $x \gray y$ where $x \in \gen{X}_3$ and $y \in \gen{X}_3$, are higher syzygies exhibiting confluence at these higher branchings.
\end{exm}

\begin{comm}
Diagrammatic complexes are closely related to polygraphs, so the definitions of confluent, terminating, and coherent presentations should admit sufficiently straightforward translations to our framework. 

The only difficulty is the treatment of degenerate cells. This can most likely be circumvented by considering finite sub-presheaves of the underlying \emph{combinatorial polygraph} of a diagrammatic complex \cite[Section 6.2]{hadzihasanovic2020diagrammatic}. We note that, in fact, a combinatorial polygraph is equivalent to a polygraph if [Conjecture 7.7, \emph{ibid.}] holds, which for now is proven up to dimension 3.

We would like to make the informal conjecture that, if $(X, \gen{X})$ and $(Y, \gen{Y})$ are two presentations of pros that are coherent in a suitable sense, then $X \gsmash \oppall{Y}$ is a coherent diagrammatic presentation of their tensor product; or, at least, a coherent presentation can be extracted from it. We leave the formal development of this problem to future work.
\end{comm}

\begin{dfn} 
To conclude, we may want to leave behind the interpretation of diagrammatic complexes as presentations of pros or probs and consider them directly as embodiments of higher\nbd dimensional theories, such as homotopical algebraic theories. 
\end{dfn}

\begin{dfn}[$n$-Tuply monoidal diagrammatic set]
For each $n > 0$, a pointed diagrammatic set $(X, \bullet)$ is \emph{$n$\nbd tuply monoidal} if $\bullet\colon 1 \to \skel{n-1}{X}$ is an isomorphism.

We say \emph{monoidal} instead of 1\nbd tuply monoidal and \emph{doubly monoidal} instead of 2\nbd tuply monoidal.
\end{dfn}

\begin{exm}
A presentation of a pro is monoidal, while a presentation of a prob is doubly monoidal. 
\end{exm}

\begin{dfn}
We propose the following basic setup for higher\nbd dimensional universal algebra in diagrammatic sets:
\begin{itemize}
	\item a presentation of a $k$\nbd tuply monoidal higher\nbd dimensional theory is embodied by a $k$\nbd tuply monoidal diagrammatic complex $(X, \gen{X}, \bullet_X)$;
	\item a ``semantic universe'' for such a theory is a $k$\nbd tuply monoidal \emph{diagrammatic set with weak composites} $(M, \bullet_M)$ \cite[\S 6.1]{hadzihasanovic2020diagrammatic}, a form of weak higher category;
	\item both the right and the left hom $\rimp{(X,\bullet_X)}{(M, \bullet_M)}$, $\limp{(X,\bullet_X)}{(M, \bullet_M)}$ are spaces of models of the theory in $M$. These coincide on 0\nbd cells, which are pointed morphisms from $(X,\bullet_X)$ to $(M,\bullet_M)$, but have different (``lax'' or ``oplax'') higher transformations.
\end{itemize}
Diagrammatic sets with weak composites encompass strict $\omega$\nbd categories via the diagrammatic nerve construction of [Section 7.2, \emph{ibid.}] but also homotopy types via the right adjoint of geometric realisation, leading to a strictly more general class of semantic universes compared to the theory of polygraphs.
\end{dfn}

\begin{exm}
We extend the presentation of Example \ref{exm:monoid_presentation} to a presentation of the 2\nbd dimensional theory of \emph{pseudomonoids} \cite{street1997monoidal}. 

First we add generating 4\nbd cells $\pi, \tau$ where 
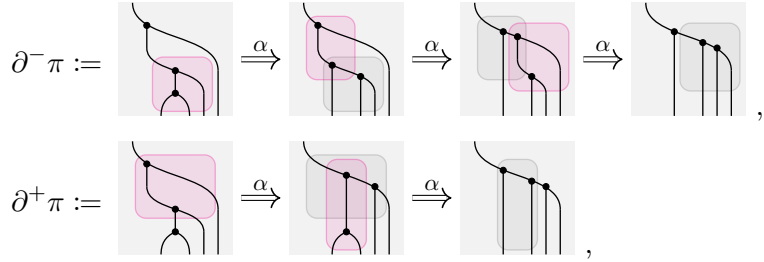
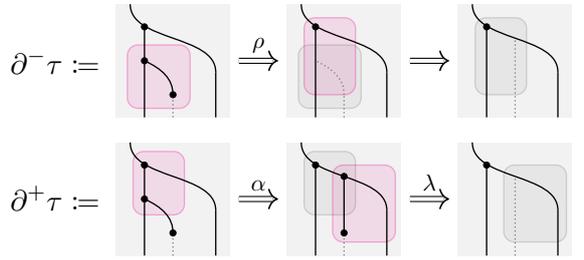
\begin{align*}
	\bord{}{-}\pi & \eqdef \; \input{img/pentagonator_input} \\
	\bord{}{+}\pi & \eqdef \; \input{img/pentagonator_output}
\end{align*}
while
\begin{align*}
	\bord{}{-}\tau & \eqdef \; \input{img/triangle_input} \\
	\bord{}{+}\tau & \eqdef \; \input{img/triangle_output}
\end{align*}
Here the unlabelled 3\nbd cell in $\bord{}{-}\tau$ is a degenerate 3\nbd cell of the form $p;\mu$ for an appropriate surjective map of atoms $p$. 

We let $\theory{PsMon}$ be the localisation of this diagrammatic complex at the set $\{\alpha,\lambda,\rho\}$ \cite[\S 6.4]{hadzihasanovic2020diagrammatic}. This operation weakly inverts the generating 3\nbd cells. We note that a localisation of a diagrammatic complex is always a diagrammatic complex.
\end{exm}

\begin{exm}
The paradigmatic 2\nbd category has small categories as 0\nbd cells, functors as 1\nbd cells, and natural transformations as 2\nbd cells. This can be given a cartesian monoidal structure, and this monoidal structure can be strictified, producing a strict monoidal 2\nbd category. 

This is equivalent to a 3\nbd category with a single 0\nbd cell. If we restrict to suitably small categories, we can make sure that this defines an object of $\ncat{3}$, which through the diagrammatic nerve produces a pointed diagrammatic set $\theory{Cat}_\times$ that has weak composites and is monoidal.

A model of $\theory{PsMon}$ in $\theory{Cat}_\times$, that is, a pointed morphism $\theory{PsMon} \to \theory{Cat}_\times$ in $\dgmpoint$, is then precisely a small monoidal category.
\end{exm}

\begin{rmk}
Let $(X, \bullet_X)$ be an $n$\nbd tuply monoidal and $(Y, \bullet_Y)$ an $m$\nbd tuply monoidal diagrammatic set. Then $X \gsmash Y$ is $(n+m)$\nbd tuply monoidal. As a special case, we recover the fact that the smash product of two presentations of pros is doubly monoidal, so it presents a prob.

Following Baez and Dolan's stabilisation hypothesis \cite{baez1995higher}, we expect a $k$\nbd tuply monoidal diagrammatic complex to present a prop when $k > 2$. We have not defined a realisation functor that would make this statement precise. Nevertheless, this gives us an idea of what iterated smash products of monoidal theories produce: a single smash product yields a braided monoidal theory, any number above that a symmetric monoidal theory. 

For higher\nbd dimensional theories, this ought to generalise as follows: the smash product of $m$ different $k$\nbd tuply monoidal diagrammatic sets, interpreted as presentations of $n$\nbd dimensional theories, presents a symmetric monoidal $n$\nbd dimensional theory as soon as $mk > n + 1$. 
\end{rmk}

\subsection*{Acknowledgements}

This work was supported by the ESF funded Estonian IT Academy research measure (project 2014-2020.4.05.19-0001).

I am grateful to Simon Forest, Fosco Loregian, Fran\c{c}ois M\'etayer, Samuel Mimram, Pawe\l{} Soboci\'nski, and Sophie Turner for various forms of help.

\bibliographystyle{alpha}
\small \bibliography{main}

\end{document}